\documentclass[11pt,a4paper]{article}
\usepackage[T1]{fontenc}
\usepackage[utf8]{inputenc}
\usepackage{lmodern}
\usepackage{iftex}
\ifdefined\pdfmapfile\pdfmapfile{+lm.map}\pdfmapfile{+symbols.map}\pdfmapfile{+rsfs.map}\fi
\usepackage[margin=27mm]{geometry}
\usepackage{amsmath,amssymb,amsthm,mathtools,mathrsfs}
\usepackage{microtype,booktabs,longtable,array,enumitem}
\ifpdf
\usepackage{xcolor}
\usepackage{hyperref}
\else
\usepackage[dvipdfmx]{xcolor}
\usepackage[dvipdfmx]{hyperref}
\fi
\hypersetup{hidelinks,pdftitle={A self-normalized law of the iterated logarithm for Rademacher random multiplicative functions},pdfauthor={Dariusz Buraczewski and Alexander Iksanov}}
\newcommand{\E}{\mathbb E}
\newcommand{\PP}{\mathbb P}

\newcommand{\FF}{\mathcal F}

\newcommand{\RR}{\mathbb R}
\newcommand{\cP}{\mathcal P}
\newcommand{\1}{\mathbf 1}
\newcommand{\dd}{\,\mathrm d}
\DeclareMathOperator{\Var}{Var}

\newcommand{\sfcount}{\Phi_{\mathrm{sf}}}

\newcommand{\cD}{\mathcal D}
\newcommand{\cC}{\mathcal C}

\newtheorem{theorem}{Theorem}[section]
\newtheorem{proposition}[theorem]{Proposition}
\newtheorem{lemma}[theorem]{Lemma}
\newtheorem{corollary}[theorem]{Corollary}

\theoremstyle{remark}

\numberwithin{equation}{section}
\setlist[itemize]{leftmargin=*,itemsep=3pt,topsep=5pt}
\allowdisplaybreaks[2]
\title{A self-normalized law of the iterated logarithm\\
for Rademacher random multiplicative functions}
\author{Dariusz Buraczewski\footnote{Mathematical Institute, University of Wroclaw, Poland; e-mail address: dariusz.buraczewski@math.uni.wroc.pl} \ \ and \ \ Alexander Iksanov\footnote{Faculty of Computer Science and Cybernetics, Taras Shevchenko National University of Kyiv, Ukraine; e-mail address:
iksan@univ.kiev.ua}}
\date{}

\begin{document}
\maketitle
\begin{abstract}
Let \(f\) be a Rademacher random multiplicative function supported on the
squarefree integers, and let \(M(x)=\sum_{n\le x}f(n)\). We prove a two-sided law of the iterated logarithm normalized by the full predictable quadratic variation
\[
Q(x)=\sum_{p\le x}\bigg(\sum_{\substack{m\le x/p\\P(m)<p}}f(m)\bigg)^2
\]
of the martingale obtained by revealing the prime signs. Almost surely,
\[
\limsup_{x\to\infty}\frac{M(x)}{\sqrt{2Q(x)\log\log x}}=1,
\qquad
\liminf_{x\to\infty}\frac{M(x)}{\sqrt{2Q(x)\log\log x}}=-1.
\]
The same equalities hold with \(\log\log Q(x)\) in the denominator. The almost sure cluster set of the normalized sum is \([-1,1]\). This gives a self-normalized version of the questions of Wintner and Erd\H{o}s on almost sure fluctuations. We also prove that \(Q(x)/q(x)\) cannot converge in probability
to \(1\) for any positive deterministic function \(q\).
Moreover,
\[
 \limsup_{x\to\infty}
 \frac{|M(x)|}{\sqrt{x}(\log\log x)^{1/4}}
\]
is either an almost surely finite and
positive random variable with nondegenerate distribution or infinite
almost surely.
We do not determine which alternative holds.
\end{abstract}
\noindent\textit{Key words:} law of the iterated logarithm;
Rademacher random multiplicative functions; self-normalization;
predictable quadratic variation; probabilistic number theory;
random Euler products.

\smallskip
\noindent\textit{2020 Mathematics Subject Classification:}
Primary: 60F15; Secondary: 11K65, 11N37, 60G50, 60G42.

\section{Introduction and main results}\label{sec:intro}
Let \(\cP\) be the set of primes. On a probability space
\((\Omega,\FF,\PP)\), let \((\varepsilon_p)_{p\in\cP}\) be independent
Rademacher random variables, that is
\(\PP(\varepsilon_p=1)=\PP(\varepsilon_p=-1)=1/2\).
A positive integer is squarefree if no prime square divides it.
Define the squarefree Rademacher random multiplicative function by
\[
f(1)=1,\qquad
f(n)=
\begin{cases}
\displaystyle\prod_{p\mid n}\varepsilon_p,&n\text{ squarefree},\\
0,&n\text{ not squarefree}.
\end{cases}
\]
Put
\[
M(x)=\sum_{n\le x}f(n).
\]
For \(n\ge2\), write \(P(n)\) for its largest prime factor,
and set \(P(1)=1\). All logarithms are natural. We write
\(\log_1x=\log x\) and \(\log_{j+1}x=\log(\log_jx)\)
whenever these expressions are defined.

Almost sure fluctuations of \(M(x)\) arise in Wintner's question
on an arithmetic analogue of the law of the iterated logarithm
and in Erd\H{o}s Problem~520. Wintner
\cite{Wintner} introduced the squarefree Rademacher model in 1944 as a
probabilistic model for the M\"obius function. Erd\H{o}s  asked whether there is a deterministic
constant \(c>0\) for which
\begin{equation}\label{eq:erdos}
\limsup_{N\to\infty}
\frac{M(N)
}{\sqrt{N\log_2 N}}=c
\quad\text{almost surely}.
\end{equation}
This question is listed as Problem~520 in \cite{Bloom}.
The recent upper bounds of Durkan and Pearce-Crump
\cite{DPC} show that the quotient in
\eqref{eq:erdos} converges to zero almost surely.
In this paper we replace the deterministic variance scale \(N\) by the
full predictable quadratic variation \(Q(N)\), defined in \eqref{eq:ap}, and obtain the exact signed constants \(1\) and \(-1\). Our main result is Theorem \ref{thm:main}.

For comparison, let $(\xi_n)_{n\ge1}$ be independent
Rademacher random variables, with
$\mathbb P(\xi_n=1)=\mathbb P(\xi_n=-1)=1/2$, and put
$S_N=\sum_{n=1}^N\xi_n$. The classical law of the iterated logarithm (see, for instance, \cite[Theorem~8.5.2]{Durrett}) states that, almost surely,
\[
\limsup_{N\to\infty}\frac{S_N}{\sqrt{2N\log_2 N}}=1,
\qquad
\liminf_{N\to\infty}\frac{S_N}{\sqrt{2N\log_2 N}}=-1.
\]
In the multiplicative model, the values $f(n)$ are orthogonal but
not independent: for example, $f(6)=f(2)f(3)$.
Although $\mathbb E Q(x)\sim 6x/\pi^2$, the quantity $Q(x)$ is random.
The normalization in our theorem involves this random quantity itself.

\subsection{Main results}\label{sec:model}

\subsubsection{The self-normalized law of the iterated logarithm}
For each prime \(p\), define
\begin{equation}\label{eq:ap}
a_p(x)=\sum_{\substack{m\le x/p\\P(m)<p}}f(m) \qquad
\text{and} \qquad
Q(x)=\sum_p a_p(x)^2.
\end{equation}

For squarefree \(n>1\), there is a unique decomposition \(n=pm\)
with \(p=P(n)\) and \(P(m)<p\). Hence
\begin{equation}\label{eq:decomp}
M(x)=1+\sum_p\varepsilon_p a_p(x).
\end{equation}
For each fixed \(x\), the partial sums
\(
1+\sum_{p\le y}\varepsilon_p a_p(x)
\)
form a martingale with respect to
\(\FF_y=\sigma(\varepsilon_q:q\le y)\).
Indeed, for \(\FF_{p-}=\sigma(\varepsilon_q:q<p)\),
\(a_p(x)\) is \(\FF_{p-}\)-measurable and
\(\varepsilon_p\) is independent of \(\FF_{p-}\), with mean zero.
For each fixed \(x\), \(Q(x)\) is the terminal predictable
quadratic variation of the martingale obtained by revealing
the prime signs in increasing order. Indeed,
\[
 Q(x)
 =\sum_{p\le x}
   \Var\bigl(\varepsilon_pa_p(x)\mid\FF_{p-}\bigr)
 =\sum_{p\le x}a_p(x)^2.
\]
In particular,
\[
 \E Q(x)=\Var(M(x)).
\]
Our main result is the following self-normalized law of the
iterated logarithm.

\begin{theorem}[Self-normalized LIL]\label{thm:main}
Almost surely,
\begin{equation}\label{eq:main}
\limsup_{x\to\infty}
\frac{M(x)}{\sqrt{2Q(x)\log_2x}}=1,
\qquad
\liminf_{x\to\infty}
\frac{M(x)}{\sqrt{2Q(x)\log_2x}}=-1.
\end{equation}
The same equalities hold if \(\log_2x\) is replaced by
\(\log_2Q(x)\).
\end{theorem}

For a random function \(A\), denote by \(\operatorname*{Clust}_{x\to\infty}A(x)\)
its cluster set as \(x\to\infty\), that is, the set of its subsequential
limits, understood realization-wise. By Theorem~\ref{thm:main}, the
supremum and the infimum of the cluster set of the normalized sum are
equal to \(1\) and \(-1\), respectively, almost surely. Our next result
identifies all its points.

\begin{corollary}\label{cor:cluster}
Almost surely,
\begin{equation}\label{eq:cluster}
\operatorname*{Clust}_{x\to\infty}\frac{M(x)}{\sqrt{2Q(x)\log_2x}}=[-1,1].
\end{equation}
The same equality holds if \(\log_2x\) is replaced by \(\log_2Q(x)\).
\end{corollary}

We stress that the random normalization in Theorem~\ref{thm:main}
is not merely a technical device. By Theorem~\ref{dn:clock} below,
\(Q(x)\) has no deterministic asymptotic equivalent, and, by
Theorem~\ref{dn:critical} below, a deterministic normalization at the
critical scale \(\sqrt{x}(\log_2x)^{1/4}\) cannot produce a finite
deterministic limit constant. Thus, the randomness of \(Q(x)\) is an
intrinsic feature of the model. Theorem~\ref{thm:main} shows that,
once \(M(x)\) is normalized by its own predictable quadratic
variation, the constants \(1\) and \(-1\) of the classical law of the
iterated logarithm reappear. This is in line with the role of
self-normalization in martingale limit theory, see, for instance,
\cite{HallHeyde}. For example, in the martingale central limit
theorem, normalization by the quadratic variation turns a mixed normal
limit into a standard normal one, see \cite[Chapter~3]{HallHeyde}.
For the Steinhaus model, deterministically normalized sums converge in
distribution to Gaussian scale mixtures \cite{HardyRestricted,GW}.
Theorem~\ref{thm:main} may be viewed as an almost sure analogue of
this phenomenon. 

We also note that the classical martingale laws of the iterated
logarithm, see \cite{Stout} and \cite[Chapter~4]{HallHeyde}, do not
apply directly. For each fixed \(x\), the variable \(M(x)-1\) is the
terminal value of a martingale indexed by the primes, but
\(x\mapsto M(x)\) is not a martingale, and the martingales corresponding
to different values of \(x\) are built from the same prime signs. The
proofs of both bounds in Theorem~\ref{thm:main} therefore require
control of this dependence uniformly in \(x\).

\subsubsection{Deterministic normalization}
\label{sec:deterministic-results}

A natural question arising from Theorem~\ref{thm:main} is whether
the random normalization can be replaced by a deterministic one
while retaining a finite positive LIL constant. We do not resolve
this question in general. We establish two obstructions and a
zero--one alternative for the absolute limsup at the critical scale.

First, the predictable quadratic variation \(Q\) has no
deterministic asymptotic equivalent, even along a deterministic
subsequence.

\begin{theorem}
\label{dn:clock}
For every deterministic sequence \(x_k\to\infty\) and every
positive deterministic sequence \(q_k\), the convergence
\[
\frac{Q(x_k)}{q_k}\xrightarrow{\mathbb P}1
\]
cannot hold. In particular, no deterministic sequence is
asymptotically equivalent to \(Q(n)\) almost surely.
\end{theorem}

This result alone does not determine the limsup of a
deterministically normalized sum. Our second result concerns
the critical scale \(\sqrt{x}(\log_2x)^{1/4}\) directly.
Put
\begin{equation}\label{dn:amplitude-definition}
b(x)=\sqrt{x}(\log_2x)^{1/4},\qquad
R^*=\limsup_{x\to\infty}\frac{|M(x)|}{b(x)}.
\end{equation}

\begin{theorem}
\label{dn:critical}
For every finite \(R>0\), the following two inequalities hold
simultaneously with positive probability:
\[
\limsup_{x\to\infty}\frac{M(x)}{b(x)}>R,
\qquad
\limsup_{x\to\infty}\frac{-M(x)}{b(x)}>R.
\]
In particular, \(R^*\) has no finite deterministic essential
upper bound, and cannot equal a finite deterministic constant
almost surely.
\end{theorem}

The preceding theorem leaves open whether \(R^*\) is finite.
The next proposition shows that \(R^*\) is either finite
almost surely or infinite almost surely. In the finite case,
it is positive almost surely and has a nonconstant distribution
with unbounded support.

\begin{proposition}[Zero--one alternative]\label{dn:zero-one}
The event \(\{R^*<\infty\}\) belongs to the
tail sigma-field of the prime signs.
Consequently, precisely one of the following alternatives holds:
\[
0<R^*<\infty\quad\text{almost surely},
\qquad\text{or}\qquad
R^*=\infty\quad\text{almost surely}.
\]
In the first case, \(R^*\) has unbounded support and is therefore
nonconstant.
\end{proposition}

Verreault \cite[Conjecture~1.4]{Verreault} conjectures a law of
the iterated logarithm for the Steinhaus model with a random
amplitude determined by critical multiplicative chaos.
Our results concern the squarefree Rademacher model. They exclude
a finite deterministic amplitude at the critical scale, but
leave open which alternative in Proposition~\ref{dn:zero-one}
holds. In particular, a finite nonconstant random amplitude
remains possible. These results do not settle Verreault's
conjecture.

These results do not determine whether there exist a positive
deterministic slowly varying function \(\ell(x)\to\infty\)
and a constant \(c\in(0,\infty)\) such that
\[
 \limsup_{x\to\infty}
 \frac{|M(x)|}
 {\sqrt{x}(\log_2x)^{1/4}\ell(x)}
 =c
 \qquad\text{almost surely}.
\]

\subsection{Historical background}

The squarefree Rademacher model is studied alongside the Steinhaus
model, in which the values $g(p)$ are independent and uniformly
distributed on the unit circle, and
\[
g(n)=\prod_{p\mid n}g(p)^{v_p(n)},\qquad g(1)=1,
\] where $v_p(n)$ denotes the exponent of $p$ in the prime
factorization of $n$.
Thus $g$ is completely multiplicative, whereas the Rademacher
function $f$ vanishes on integers that are not squarefree.
Since the methods for the two models are closely related, we include both in the following overview. We retain
$f$ and $M(x)=\sum_{n\le x}f(n)$ for the Rademacher model and write
\[
M_{\mathrm S}(x)=\sum_{n\le x}g(n)
\]
for the Steinhaus model. In statements applying to both models,
$\mathcal M$ denotes either $M$ or $M_{\mathrm S}$; the labels R and S in the table indicate the models to which a result applies.

\subsubsection{Almost sure bounds}

The study of almost sure bounds began with Wintner \cite{Wintner}
in 1944. For the squarefree Rademacher model, he proved that
$M(x)=O(x^{1/2+\varepsilon})$ almost surely for every
$\varepsilon>0$, whereas $M(x)=O(x^{1/2-\varepsilon})$ fails
almost surely. These results determine the exponent $1/2$ but leave open the size of the fluctuations around $\sqrt{x}$.
Wintner suggested that an arithmetic analogue of the law of the
iterated logarithm might describe them. 

Hal\'asz \cite{Halasz} used conditioning on small primes and moment estimates to obtain more precise upper and lower bounds. Basquin \cite{Basquin} and
Lau, Tenenbaum and Wu \cite{LTW} subsequently reduced the loss
beyond $\sqrt{x}$ to a power of $\log\log x$. The corrected version of \cite{LTW} gives
\[
|M(x)|\ll\sqrt{x}(\log_2x)^{2+\varepsilon}.
\]

Harper's early work \cite{HarperGaussian} used Gaussian-process
estimates to improve the lower bounds. In \cite{HarperLarge}, he proved that, for either model and every fixed
deterministic function $V(x)\to\infty$, almost surely
\[
|\mathcal M(x)|\ge
\frac{\sqrt{x}(\log_2x)^{1/4}}{V(x)}
\]
for arbitrarily large $x$. In particular,
$\mathcal M(x)=O(\sqrt{x})$ fails almost surely.

The subsequent upper bounds approach this lower bound. Caich \cite{Caich} reduced the upper exponent of $\log_2x$
to $3/4+\varepsilon$. Further improvements appeared in the summer and early autumn of 2026. Durkan and Pearce-Crump
\cite{DPC} obtained, for both models,
\[
|\mathcal M(x)|\ll\sqrt{x}(\log_2x)^{1/4+\varepsilon}.
\]
Together with Harper's lower bound, this determines the exponent
$1/4$. Verreault \cite{Verreault} further refined the upper bound to
\[
|\mathcal M(x)|\ll
\sqrt{x}(\log_2x)^{1/4}(\log_3x)^{1+\varepsilon}.
\]
These developments also show that the quotient in
\eqref{eq:erdos} converges to zero almost surely.
Table~\ref{tab:history} summarizes these almost sure bounds.
{
\small
\setlength{\tabcolsep}{5pt}
\renewcommand{\arraystretch}{1.18}
\begin{longtable}{@{}>{\raggedright\arraybackslash}p{39mm}>{\raggedright\arraybackslash}p{10mm}>{\raggedright\arraybackslash}p{100mm}@{}}
\caption{Development of almost sure bounds for the full sums
$M(x)$ and $M_{\mathrm S}(x)$.}\label{tab:history}\\
\toprule
Work & Model & Result or contribution\\
\midrule
\endfirsthead
\toprule Work & Model & Result or contribution\\ \midrule
\endhead
\bottomrule
\endfoot
Wintner (1944) \cite{Wintner}
& R & $M(x)=O(x^{1/2+\varepsilon})$;
$M(x)\ne O(x^{1/2-\varepsilon})$.\\
Hal\'asz (1983) \cite{Halasz}
& R & $M(x) \ll
\sqrt{x}\exp\{C\sqrt{\log_2x\,\log_3x}\}$.\\
Basquin (2012) \cite{Basquin}; Lau--Tenenbaum--Wu (2013)
\cite{LTW}
& R & $M(x)\ll\sqrt{x}(\log_2x)^{2+\varepsilon}$.\\
Harper (2010 preprint; 2013) \cite{HarperGaussian}
& R & $M(x)\ne O(\sqrt{x}(\log_2x)^{-5/2-\varepsilon})$.\\
Harper (2020 preprint; 2023) \cite{HarperLarge}
& R, S & $|\mathcal M(x)|\gg\sqrt{x}(\log_2x)^{1/4}/V(x)$
for every
deterministic $V(x)\to\infty$.\\
Caich (2023 preprint; September 2026) \cite{Caich}
& R, S & $|\mathcal M(x)|\ll\sqrt{x}(\log_2x)^{3/4+\varepsilon}$.\\
Durkan--Pearce-Crump (July 2026) \cite{DPC}
& R, S & $|\mathcal M(x)|\ll\sqrt{x}(\log_2x)^{1/4+\varepsilon}$.\\
Verreault (August 2026; rev. September 2026) \cite{Verreault}
& R, S & $|\mathcal M(x)|\ll\sqrt{x}(\log_2x)^{1/4}G(\log_3x)$
if $G$ is eventually nondecreasing and
$\sum_{h\ge3}h\log h/G(h)^2<\infty$.
One may take $G(t)=t^{1+\varepsilon}$.\\
\end{longtable}
}

\subsubsection{Restricted sums and related almost sure results}
The difficulties in studying the full sum have also led to the study of sums restricted by the largest prime factor. Put
\[
M_{>}(x)=\sum_{\substack{n\le x\\P(n)>\sqrt{x}}}f(n),
\qquad
M_{\mathrm S,>}(x)=
\sum_{\substack{n\le x\\P(n)>\sqrt{x}}}g(n),
\]
where $P(n)$ is the largest prime factor of $n$, with $P(1)=1$.
Mastrostefano \cite{Mastro} proved, for both models, that
\[
|\mathcal M_{>}(x)|\ll\sqrt{x}(\log_2x)^{1/4+\varepsilon},
\]
where $\mathcal M_{>}$ denotes the corresponding restricted sum.

Buraczewski and Iksanov
\cite[Theorem~1.1 and Corollary~1.2]{BIRestricted}
prove a self-normalized law of the iterated logarithm for
the restricted Rademacher sum \(M_{>}(x)\). With
\[
Q_{>}(x)=\sum_{\sqrt{x}<p\le x}M(x/p)^2,
\]
they show that, almost surely,
\[
\limsup_{x\to\infty}
\frac{M_{>}(x)}{\sqrt{2Q_{>}(x)\log_2Q_{>}(x)}}=1,
\qquad
\liminf_{x\to\infty}
\frac{M_{>}(x)}{\sqrt{2Q_{>}(x)\log_2Q_{>}(x)}}=-1.
\]
The cluster set of the normalized sum is \([-1,1]\).
The present paper treats the full sum \(M(x)\), whose
predictable quadratic variation includes all prime scales.

Shala \cite{Shala} developed a framework for almost sure large
fluctuations over arithmetic subsets in regimes where the sums
are approximately Gaussian, including admissible short intervals
and polynomial value sets. Gou and Xu \cite{GX} proved large
fluctuations of both signs for the extended Rademacher model.
That model is completely multiplicative and is nonzero on squares;
the result concerns Erd\H{o}s Problem~1144 and must be
distinguished from the squarefree model considered here.

Almost sure bounds have also been investigated for weighted sums. Aymone, Heap and Zhao
\cite{AHZ} and Hardy \cite{HardyWeighted} studied almost sure
bounds for $\sum_{n\le x}g(n)/\sqrt n$, while Atherfold
\cite{Atherfold} obtained almost sure estimates for
$\sum_{n\le x}f(n)/\sqrt n$. These sums have different
normalizations from the unweighted sum. For squarefree Rademacher
functions, Aymone, Heap and Zhao \cite{AHZsign} also proved
that $M(x)$ changes sign infinitely often. Our theorem determines the size of both positive and negative excursions relative to $Q(x)$.

Laws of the iterated logarithm for random Dirichlet series were obtained by Aymone, Fr\'ometa and Misturini \cite{AFM}, by Buraczewski, Dong, Iksanov and Marynych \cite{BDIM}, by Dong and Iksanov \cite{DI} and by Iksanov and Kostohryz
\cite{IK}. These results concern series of independent variables
and a different limiting parameter; \cite{DI} also gives
consequences for the logarithm of the Rademacher random Dirichlet
series near its convergence boundary. For general deterministic
weights, \cite{IK} gives conditions for the classical variance
normalization and examples showing that variance growth alone
need not suffice. They do not yield a law of the iterated logarithm for the unweighted partial sum $M(x)$.

\subsubsection{Moments and distributional limits}

The moments and distributional limits of these sums have been studied in parallel with their almost sure bounds.
Harper's work on low moments \cite{HarperLow} established,
for both models,
\[
\mathbb E|\mathcal M(x)|\asymp
\frac{\sqrt{x}}{(\log_2x)^{1/4}},
\]
and connected the sums with critical multiplicative chaos.
By Markov's inequality,
\[
\frac{\mathcal M(x)}{\sqrt{x}}
\xrightarrow{\mathbb P}0,
\]
although $\mathbb E M(x)^2\sim6x/\pi^2$ and
$\mathbb E|M_{\mathrm S}(x)|^2=\lfloor x\rfloor$.
Thus the square-root normalization gives a degenerate
distributional limit.
For an account of the moment methods, including high moments, we refer to \cite{HarperSurvey}.

\paragraph{Gaussian mixtures in the Steinhaus model}
For the Steinhaus model, Hardy \cite{HardyRestricted} proved convergence in distribution
of $(\log_2x)^{1/4}/\sqrt{x} \cdot M_{\mathrm S,>}(x)$ to a
complex Gaussian scale mixture. Gorodetsky and Wong \cite{GW}
subsequently established the corresponding limit for the full sum:
\[
\frac{(\log_2x)^{1/4}}{\sqrt{x}}M_{\mathrm S}(x)
\xrightarrow{\mathrm d}\sqrt{V_\infty}\,G,
\]
where $G$ is a standard complex Gaussian with
$\mathbb E|G|^2=1$, independent of the almost surely finite,
strictly positive random variable $V_\infty$. This variable is defined by an integral with respect to a critical multiplicative-chaos measure. Hardy's restricted sum has the same deterministic
scaling but its own limiting random variance.
The limiting variances in these Gaussian mixtures are random.

\paragraph{Central limit theorems in short intervals}
For squarefree Rademacher functions, Chatterjee and Soundararajan
\cite{ChatterjeeSoundararajan} proved a central limit theorem
in the range $x^{1/5}\log x\ll h=o(x/\log x)$.
Soundararajan and Xu \cite[Corollary~9.2]{SoundararajanXu} enlarged the range;
in particular, for every fixed $\varepsilon>0$, it includes
\[
x^{1/5}\log x\ll h\le
\frac{x}{(\log x)^{2\log 2-1+\varepsilon}}.
\]
Put
\[
V_{\mathrm R}(x,h)=
\#\{x<n\le x+h:\ n\text{ is squarefree}\}.
\]
Then
\[
\frac{1}{\sqrt{V_{\mathrm R}(x,h)}}
\sum_{x<n\le x+h}f(n)
\xrightarrow{\mathrm d}\mathcal N(0,1).
\]
Soundararajan and Xu also give general criteria for Gaussian
limits over subsets, based on controlling non-diagonal
contributions to the fourth moment.

For Steinhaus functions, Harper, Soundararajan and Xu \cite{HSX}
proved a Gaussian limit throughout the range $h\to\infty$
and $h=o(x)$. There is a deterministic scaling factor
$V_{\mathrm S}(x,h)$ such that
\[
\frac{1}{\sqrt{V_{\mathrm S}(x,h)}}
\sum_{x<n\le x+h}g(n)
\xrightarrow{\mathrm d}\mathcal N_{\mathbb C}(0,1),
\]
with
\[
V_{\mathrm S}(x,h)\asymp
h\min\!\biggl\{1,
\frac{\log(x/h)}{\sqrt{\log_2x}}\biggr\}.
\]

Xu \cite{XuSupport} further showed that, for subsets
$A=A_N\subset[1,N]$, a standard complex Gaussian limit for
$|A|^{-1/2}\sum_{n\in A}g(n)$ requires $|A|=o(N)$.
Thus both the support and the normalization enter the conditions for a Gaussian limit.

Dong and Iksanov \cite{DI} also obtained a functional central
limit theorem for random Dirichlet series summed over primes.
The underlying variables are independent. As in their LIL, the limiting parameter differs from the one used for unweighted partial sums.

\subsection{Notation and organization of the proof}
For \(y\ge1\), write
\[
\Psi(u,y)=\E[M(u)\mid\FF_y]= \sum_{\substack{n\le u\\P(n)\le y}}f(n),\qquad
\Psi(u,y^-)=\E[M(u)\mid\FF_{y-}]=\sum_{\substack{n\le u\\P(n)<y}}f(n).
\]
 In particular, \(a_p(x)=\Psi(x/p,p^-)\).

Section~\ref{sec:number-theory} collects the deterministic
number-theoretic estimates. Section~\ref{sec:upper} proves the
upper bound on a logarithmic grid and reduces interpolation to
Proposition~\ref{prop:interpolation}, proved in
Appendix~\ref{sec:interpolation}. The interpolation controls both
the sum and the full coefficient vector.

Section~\ref{sec:loweroutline} gives the proof of the lower bounds
in consecutive steps. We divide the primes into thin layers and
replace each coefficient by one determined before its layer is
exposed. We estimate the discarded terms and compare the full
quadratic variation with that of the retained sum. We then select
nested sets of endpoints: the thresholds must produce the required
signed gain while leaving a nonempty set after the last layer.
This extends the endpoint-selection method of
\cite[Sections~4, 6 and~7]{BIRestricted} to all significant prime scales.
The auxiliary estimates are collected in
Appendix~\ref{app:lower}. They concern the discarded terms,
the variation of the layer variances, the covariances, and a
quantitative counting estimate. Corollary~\ref{cor:cluster} is proved at the end of Section~\ref{sec:loweroutline}.

Section~\ref{sec:deterministic} proves the results on deterministic
normalization. Theorem~\ref{dn:clock} uses Harper's low-moment
estimates and Burkholder's square-function inequality, independently
of the LIL. Theorem~\ref{dn:critical} uses the simultaneous signed
block exceedances and the comparison with the full quadratic
variation from the lower-bound proof. The proof of
Proposition~\ref{dn:zero-one} is based on an exact identity for
changing one prime sign and Kolmogorov's zero--one law.

\section{Number-theoretic estimates}\label{sec:number-theory}

This section collects the deterministic estimates used in the
proofs. Our main reference for the number-theoretic estimates used
below is Montgomery and Vaughan~\cite{MV}. For background
on smooth integers and Rankin's method, we also refer to
Hildebrand and Tenenbaum~\cite{HT}.

Sums and products indexed by \(p\) run over primes. Write
\(\pi(v)\) for the prime-counting function and
\[
\vartheta(v)=\sum_{p\le v}\log p.
\]
For \(n\ge2\), let \(P(n)\) and \(P^-(n)\) be its largest
and smallest prime factors, respectively; set
\(P(1)=1\) and \(P^-(1)=\infty\). We write \(\mu\) for the
M\"obius function and \(\omega(n)\) for the number of distinct
prime factors of \(n\). Thus \(\mu^2(n)\) is the indicator
that \(n\) is squarefree. Put
\[
\sfcount(v,y)
=\sum_{\substack{n\le v\\P(n)\le y}}\mu^2(n).
\]
An integer is called \(y\)-smooth if \(P(n)\le y\), and
\(y\)-rough if \(P^-(n)>y\). Every squarefree integer has
a unique factorization into a squarefree \(y\)-rough part
and a squarefree \(y\)-smooth part. These two parts are
coprime. We use this factorization when conditioning on the
signs at primes at most \(y\).
Also, every squarefree \(n>1\) has a unique representation
\(n=pm\), with \(p=P(n)\) and \(P(m)<p\).

\begin{lemma}\label{lem:prime-estimates}
There is an absolute constant \(c>0\) such that
\begin{equation}\label{eq:pnt}
\vartheta(v)=v+O\bigl(ve^{-c\sqrt{\log v}}\bigr),
\qquad v\ge2.
\end{equation}
In particular, \(\pi(v)\sim v/\log v\). Moreover,
uniformly for \(y\ge3\),
\begin{equation}
\sum_{p\le y}\frac1p=\log_2y+O(1), \qquad
\sum_{p\le y}\frac{\log p}{p}\ll\log y, \qquad
\sum_{p\le y}\frac{(\log p)^2}{p}\ll(\log y)^2,
\label{eq:prime-harmonic}
\end{equation}
\begin{equation}
\prod_{p\le y}(1+p^{-1})\asymp\log y, \qquad
\prod_{p\le y}(1-p^{-1})\asymp(\log y)^{-1}, \qquad
\sum_{p>y}\frac1{p\log p}\ll\frac1{\log y}.
\label{eq:prime-products}
\end{equation}
\end{lemma}
\begin{proof}
Equation~\eqref{eq:pnt} follows from the prime number theorem
in de la Vall\'ee Poussin's form
\cite[Theorem~6.9]{MV}.
Partial summation gives \(\pi(v)\sim v/\log v\) and
\eqref{eq:prime-harmonic}. For the classical Mertens
estimates, see also \cite[Theorem~2.7]{MV}.
Since \(\sum_p p^{-2}<\infty\), taking logarithms gives
the two product estimates in \eqref{eq:prime-products}.
Finally, partial summation and \(\pi(v)\ll v/\log v\)
give
\[
\sum_{p>y}\frac1{p\log p}
\ll\int_y^\infty\frac{\dd v}{v(\log v)^2}
=\frac1{\log y}.
\]
\end{proof}

\begin{lemma}\label{lem:osc-prime-sum}
Uniformly for \(3\le z\le U\) and real \(v\),
\begin{equation}\label{eq:primesosc}
\sum_{z<p\le U}p^{-1-iv}
=\int_{\log z}^{\log U}\frac{e^{-ivs}}s\dd s
+O\bigl((1+|v|)e^{-c\sqrt{\log z}}\bigr),
\end{equation}
where \(c>0\) is absolute.
\end{lemma}
\begin{proof}
This is \cite[Lemma~6.1]{BIRestricted}, with lower cutoff \(z\)
and upper cutoff \(U\).
\end{proof}

\begin{lemma}\label{lem:prime-averaging}
Let \(0<\delta\le1/2\), \(a\ge3\), and \(b>a\). Define
\[
w_{a,b,\delta}(t)
=\sum_{\substack{a<p\le b\\p\le t\le(1+\delta)p}}
\frac1{\delta p}.
\]
If
\((1+\delta)a<t\le b\), then
\begin{equation}\label{eq:prime-averaging}
w_{a,b,\delta}(t)
=\frac{1+O\bigl(\delta+
\delta^{-1}e^{-c\sqrt{\log t}}\bigr)}{\log t}.
\end{equation}
The same conclusions hold with \(b=\infty\).
\end{lemma}

\begin{proof}
Suppose first that \((1+\delta)a<t\le b\), and put
\(v_0=t/(1+\delta)\). The conditions on \(t\) imply that
\(a<v_0\le t\le b\). Hence the restrictions involving \(a\)
and \(b\) can be omitted, and
\[
w_{a,b,\delta}(t)
=\frac1\delta\sum_{v_0\le p\le t}\frac1p.
\]
Write \(\vartheta(v)=v+R(v)\). By \eqref{eq:pnt},
\(
|R(v)|\le Cv e^{-c\sqrt{\log v}}.
\)
Partial summation gives
\[
\sum_{v_0<p\le t}\frac1p
=\int_{v_0}^t\frac{\dd v}{v\log v}
+\frac{R(t)}{t\log t}
-\frac{R(v_0)}{v_0\log v_0}
+\int_{v_0}^t
R(v)\frac{1+\log v}{v^2(\log v)^2}\dd v.
\]
Since \(2t/3\le v_0\le t\), the two boundary terms and
the last integral are
\(O(e^{-c\sqrt{\log t}}/\log t)\), after reducing \(c\)
if necessary.  Therefore,
\[
w_{a,b,\delta}(t)
=\frac1\delta\int_{v_0}^t\frac{\dd v}{v\log v}
+O\bigg(\frac{e^{-c\sqrt{\log t}}}{\delta\log t}\bigg).
\]

For \(v_0\le v\le t\), we have
\(\log v=\log t+O(\delta)\). Consequently,
\[
\frac1\delta\int_{v_0}^t\frac{\dd v}{v\log v}
=\frac{1+O(\delta)}{\delta\log t}
  \int_{v_0}^t\frac{\dd v}{v}
=\frac{1+O(\delta)}{\log t}
  \frac{\log(1+\delta)}{\delta}
=\frac{1+O(\delta)}{\log t}.
\]
This proves \eqref{eq:prime-averaging}.

For an arbitrary \(t\) in the support, the primes contributing
to \(w_{a,b,\delta}(t)\) form a subset of those in
\([t/(1+\delta),t]\). Thus
\[
0\le w_{a,b,\delta}(t)
\le\frac1\delta
   \sum_{t/(1+\delta)\le p\le t}\frac1p
\le\frac{C\bigl(1+\delta^{-1}
e^{-c\sqrt{\log t}}\bigr)}{\log t},
\]
by the same calculation. The argument is unchanged when \(b=\infty\).
\end{proof}

\begin{lemma}\label{lem:rankin}
For every fixed \(A>0\), uniformly for \(v\ge1\) and \(\log y>A\),
with \(\sigma=1-A/\log y\),
\begin{equation}\label{eq:rankin-fixed}
\sfcount(v,y)
\le C_Av(\log y)
\exp\bigg(-\frac{A\log v}{\log y}\bigg).
\end{equation}
\end{lemma}
\begin{proof}
For \(\sigma>0\), Rankin's method
\cite[\S1, (1.3)]{HT}, \cite[(7.17), p.~204]{MV} gives
\begin{equation}\label{eq:rankin-general}
\sfcount(v,y)
\le
\sum_{\substack{n\le v\\P(n)\le y}}
\mu^2(n)\bigg(\frac vn\bigg)^\sigma
\le
v^\sigma\prod_{p\le y}(1+p^{-\sigma}).
\end{equation}
Applying Lemma \ref{lem:prime-estimates}, we write
\[
\log\frac{\prod_{p\le y}(1+p^{-\sigma})}
{\prod_{p\le y}(1+p^{-1})}
=
\sum_{p\le y}
\log\bigg(1+\frac{p^{A/\log y}-1}{p+1}\bigg)\le
\frac{Ae^A}{\log y}
\sum_{p\le y}\frac{\log p}{p}
\le C_A.
\]
Exponentiating and applying again  Lemma \ref{lem:prime-estimates}
gives
\[
\prod_{p\le y}(1+p^{-\sigma})
\le e^{C_A}\prod_{p\le y}(1+p^{-1})
\ll_A\log y.
\]
Combining this with \eqref{eq:rankin-general} proves
\eqref{eq:rankin-fixed}.
\end{proof}

\section{Proof of the upper bound}\label{sec:upper}
We prove the upper bound on
logarithmic blocks   {\([Y,Y^2]\)}, where {\(Y=\exp(2^r)\)}  for positive integers
\(r\). On each block, choose a geometric grid with logarithmic spacing   {\((\log Y)^{-\alpha}\)} (see \eqref{eq:coarse} below). At a fixed grid point \(z\), by representation
\eqref{eq:decomp}, the partial sums obtained by revealing the prime signs in \(M(z)-1\) form a
martingale with predictable quadratic variation \(Q(z)\).
An exponential martingale inequality, followed by a union bound over the
grid, gives the desired upper estimate there. This is proved in Proposition~\ref{prop:gridupper}.

It remains to extend the estimate to every real \(x\) in
the block. We choose a nearby grid point \(z\le x\) and control both
\(M(x)-M(z)\) and the distance
\begin{equation}\label{eq:distance}
d(z,x):=\|a(x)-a(z)\|_2,\qquad a(x):=(a_p(x))_p.
\end{equation}
For fixed \(x\) the vector \(a(x)\) has finite support, because the
sum defining \(a_p(x)\) is empty when \(p>x\). We shall also use the representation
\begin{equation}\label{eq:vector}
a(x)=\sum_{\substack{2\le n\le x\\n\ \mathrm{squarefree}}}
f(n/P(n))e_{P(n)}.
\end{equation}
Here \(e_p\) is the \(p\)th coordinate vector in \(\ell^2(\cP)\).
In particular, \(Q(x)=\|a(x)\|_2^2\). The distance \(d(z,x)\)
also controls the change in the predictable variance, since
\(|\sqrt{Q(x)}-\sqrt{Q(z)}|\le d(z,x)\).
A finer auxiliary grid and a fourth-moment maximal estimate give uniform control in \(x\) (Proposition~\ref{prop:interpolation}).
Finally, a lower bound for \(Q(x)\) (Lemma~\ref{lem:clocklower}) shows that
the interpolation errors are negligible on the scale used in the LIL.

Fix \(0<\eta<1/2\) and take \(\alpha=\eta/2\).
For each integer \(r\ge1\), set   {\(Y=\exp(2^r)\)}. Define

\begin{equation}\label{eq:coarse}
h=(\log Y)^{-\alpha},\qquad
\cC_Y=\{Ye^{kh}:0\le k\le\lfloor (\log Y)/h\rfloor\}\cup\{Y^2\}.
\end{equation}

Then {\(|\cC_Y|\le C(\log Y)^{1+\alpha}\)}, and every   {\(x\in[Y,Y^2]\)}
is within logarithmic distance at most \(h\) of a point of   {\(\cC_Y\)}
below it. Since \(h\to0\), the grid becomes
multiplicatively finer as   {\(Y\to\infty\)}. This property will be used in the interpolation argument.

\begin{proposition}\label{prop:gridupper}
For each fixed \(0<\eta<1/2\), almost surely, for all sufficiently large
\(r\) and all   {\(z\in\cC_{\exp(2^r)}\)},
\begin{equation}\label{eq:gridupper}
|M(z)-1|\le(1+\eta)\sqrt{2Q(z)\log_2z}.
\end{equation}
\end{proposition}
Before proving Proposition~\ref{prop:gridupper}, we establish deterministic upper and lower bounds for \(Q(x)\) and an exponential martingale inequality.

\begin{lemma}\label{lem:basicclock}
For all sufficiently large \(x\),
\begin{equation}\label{eq:basicclock}
\frac{x}{3\log x}\le Q(x)<x^2,\qquad
\E Q(x)=\#\{2\le n\le x:n\text{ squarefree}\}\sim6x/\pi^2.
\end{equation}
\end{lemma}
\begin{proof}
For \(x/2<p\le x\) the only summand in \(a_p(x)\) is \(m=1\), so
\(a_p(x)=1\) and hence \(Q(x)\ge\pi(x)-\pi(x/2)\), where
\(\pi(x):=\#\{p\le x:p\text{ is prime}\}\) is the prime-counting
function. Lemma~\ref{lem:prime-estimates} therefore gives
\[
Q(x)\ge\frac{x}{3\log x}.
\]
Also, \(|a_p(x)|\le x/p\), so that
\[
Q(x)\le x^2\sum_p p^{-2}<x^2.
\]
For squarefree \(m\ne n\), the product \(f(m)f(n)\) contains at least one
Rademacher sign to the first power. Hence
\begin{equation}\label{eq:orth}
\E[f(m)f(n)]=\1_{\{m=n\}}\mu^2(n).
\end{equation}
Here \(\mu\) is the M\"obius function; \(\mu^2(n)\) indicates that \(n\)
is squarefree. Expanding the squares gives
\[
\E Q(x)=\sum_p\sum_{\substack{m\le x/p\\P(m)<p}}\mu^2(m)
=\#\{2\le n\le x:n\text{ is squarefree}\},
\]
where the second equality follows from the unique representation \(n=pm\),
where \(p=P(n)\).
The final conclusion is the statement of Theorem 2.2 in \cite{MV}.
\end{proof}

The following lemma is a conditionally symmetric martingale inequality;
see, for instance, Bercu and Touati \cite[(1.4)]{BercuTouati}, and \cite[Lemma~2.4]{BIRestricted} for a short proof. Their total quadratic
variation \(\sum_p(\varepsilon_pb_p)^2\) equals \(\sum_pb_p^2\) here,
since \(\varepsilon_p^2=1\).
\begin{lemma}\label{lem:exponential}
Suppose \(b_p\) is real and \(\FF_{p-}\)-measurable, and \(b_p=0\) for all primes \(p\) outside a deterministic finite set. For deterministic \(t,v>0\),

\begin{equation}\label{eq:exponential}
\PP\biggl(\biggl|\sum_p\varepsilon_pb_p\biggr|\ge t,\
\sum_pb_p^2\le v\biggr)\le2e^{-t^2/(2v)}.
\end{equation}

\end{lemma}
\begin{proof}[Proof of Proposition~\ref{prop:gridupper}]
We first estimate the probability that \eqref{eq:gridupper} fails for
a fixed   {\(z\in\cC_Y\)}, where \(Y\) is sufficiently large.
We use a geometric grid for the variance
\(Q(z)\). Define
\[
F(z)=\frac{z}{(\log z)^2},\qquad U(z)=z(\log z)^2.
\]
If \(Y\) is large enough, Lemma~\ref{lem:basicclock} entails \(Q(z)>F(z)\).
On the other hand, Lemma~\ref{lem:basicclock} and Markov's inequality give
\[
\PP(Q(z)>U(z))\le\frac{\E Q(z)}{U(z)}\le(\log z)^{-2}.
\]
On the interval \([F(z),U(z)]\) put

\[
v_{z,j}=F(z)(1+\eta/2)^j,\quad 0\le j\le J_z,\qquad
J_z=\biggl\lceil\frac{\log(U(z)/F(z))}{\log(1+\eta/2)}\biggr\rceil
=O_\eta(\log_2z).
\]
Thus, Lemma~\ref{lem:exponential} with \(b_p=a_p(z)\) yields

\begin{align*}
&\PP\bigl(|M(z)-1|>(1+\eta)\sqrt{2Q(z)\log_2z}\bigr)\\
&\quad\le
\sum_{j=0}^{{J_z-1}}
\PP\bigl(|M(z)-1|>(1+\eta)\sqrt{2Q(z)\log_2z},\
v_{z,j}<Q(z)\le v_{z,j+1}\bigr)+(\log z)^{-2}\\
&\quad\le
\sum_{j=0}^{{J_z-1}}
\PP\bigl(|M(z)-1|>(1+\eta)\sqrt{2v_{z,j}\log_2z},\
Q(z)\le v_{z,j+1}\bigr)+(\log z)^{-2}\\
&\quad\le
2\sum_{j=0}^{{J_z-1}}
\exp\biggl\{-\frac{(1+\eta)^2\,2v_{z,j}\log_2z}
{2(1+\eta/2)v_{z,j}}\biggr\}
+(\log z)^{-2}\\
&\quad\le C(\log z)^{-(1+\eta)},
\end{align*}
because \((1+\eta)^2/(1+\eta/2)>1+\eta\) and
\(J_z=O_\eta(\log_2z)\).
Taking the union over   {\(\cC_Y\)} and using {\((\log Y)\le\log z\le2(\log Y)\)} for every {\(z\in\cC_Y\)},

\begin{align*}
&\PP\bigl(|M(z)-1|>(1+\eta)\sqrt{2Q(z)\log_2z}
\text{ for some }z\in\cC_Y\bigr)\\
&\qquad\le\sum_{z\in\cC_Y}
\PP\bigl(|M(z)-1|>(1+\eta)\sqrt{2Q(z)\log_2z}\bigr)
\le C|\cC_Y|(\log Y)^{-(1+\eta)}\le C_\eta (\log Y)^{-\eta/2}.
\end{align*}
For {\(Y=\exp(2^r)\)}  the last bound is \(C2^{-\eta r/2}\), which is summable
over \(r\). The Borel--Cantelli lemma proves that \eqref{eq:gridupper}
holds simultaneously at every point of   {\(\cC_{\exp(2^r)}\)}  for all sufficiently
large \(r\), almost surely.
\end{proof}

For interpolation from the grid   {\(\cC_Y\)}
to arbitrary points, we need two further estimates.
\begin{lemma}\label{lem:clocklower}
For every fixed \(b>0\), almost surely, for all sufficiently large
{\(Y\)} with dyadic {\(\log Y\)}, uniformly for every   {\(x\in[Y,Y^2]\)},
\begin{equation}\label{eq:clocklower}
Q(x)\ge x(\log Y)^{-b}.
\end{equation}
\end{lemma}
\begin{proof}
For \(p>\sqrt{x}\), we have \(a_p(x)=M(x/p)\). Hence
\[
Q(x)\ge Q_{>}(x)
:=\sum_{\sqrt{x}<p\le x}M(x/p)^2.
\]
The assertion follows from
\cite[Lemma~2.5]{BIRestricted}, with \(T=\log Y\).
\end{proof}

The following proposition provides the uniform interpolation estimate.
Its proof is given in Appendix~\ref{sec:interpolation}.
\begin{proposition}\label{prop:interpolation}
Fix \(0<\alpha<1/4\). Almost surely, for all sufficiently large   {\(Y\)} with dyadic {\(\log Y\)}, for every   {\(x\in[Y,Y^2]\)} there is {\(z=z_Y(x)\in\cC_Y\)},
with \(z\le x\), such that   {\(\log(x/z)\le2(\log Y)^{-\alpha}\)} and

\[
|M(x)-M(z)|+d(x,z)\le C\sqrt{x}\,(\log Y)^{-\alpha/12}.
\]

\end{proposition}
\begin{proof}[Proof of the upper bound in Theorem~\ref{thm:main}]
For   {\(x\in[Y,Y^2]\)} and {\(z=z_Y(x)\)}  as in
Proposition~\ref{prop:interpolation}, almost surely and uniformly
throughout the block as   {\(Y\to\infty\)},
\begin{equation}\label{eq:interpolationratios}
\frac{Q(z)}{Q(x)}=1+o(1),\qquad
\frac{|M(x)-M(z)|}{\sqrt{Q(x)\log_2x}}=o(1).
\end{equation}
Indeed, applying Lemma~\ref{lem:clocklower} with \(b=\alpha/12\) and
Proposition~\ref{prop:interpolation},

\[
\frac{|\sqrt{Q(z)}-\sqrt{Q(x)}|}{\sqrt{Q(x)}}
\le\frac{d(x,z)}{\sqrt{Q(x)}}
\le\frac{C\sqrt{x}(\log Y)^{-\alpha/12}}{\sqrt{x}(\log Y)^{-\alpha/24}}
=C(\log Y)^{-\alpha/24}\longrightarrow0.
\]
The second ratio in \eqref{eq:interpolationratios} is at most
\(C(\log Y)^{-\alpha/24}/\sqrt{\log_2Y}\longrightarrow0\).

Finally, fix \(0<\eta<1/2\). Combining
Proposition~\ref{prop:gridupper} with \eqref{eq:interpolationratios}
we obtain
\begin{align*}
\frac{|M(x)|}{\sqrt{2Q(x)\log_2x}}
&\le\frac1{\sqrt{2Q(x)\log_2x}}
+(1+\eta)\sqrt{\frac{Q(z)}{Q(x)}}\sqrt{\frac{\log_2z}{\log_2x}}
+\frac{|M(x)-M(z)|}{\sqrt{2Q(x)\log_2x}}\\
&\le(1+\eta)(1+o(1)).
\end{align*}
Letting \(\eta\) tend to \(0\) completes the proof of the upper bound.
\end{proof}



\section{Proof of the lower bounds}
\label{sec:loweroutline}

\begin{proof}[Proof of Theorem~\ref{thm:main}: the lower bounds.]

We seek positive and negative fluctuations of \(M(x)\) of order
\(\sqrt{2Q(x)\log_2x}\). We shall show that, on a suitable
deterministic grid, the normalized sum reaches both positive
and negative levels arbitrarily close to \(1\) and \(-1\),
respectively. We use the endpoint-selection method of
\cite[Sections~4, 6 and~7]{BIRestricted}, now applied successively
to thin prime layers. The auxiliary results are stated and
proved in Appendix~\ref{app:lower}, so that the main steps
of the proof can be read together.
For \(Y\to\infty\), consider the grid

\begin{equation}\label{eq:blockparameters}
\mathcal X_Y=
\biggl\{
Ye^j:
j=0,\ldots,\biggl\lfloor\frac{\log Y}{(\log_2Y)^{100}}\biggr\rfloor
\biggr\}.
\end{equation}

For every fixed \(c\in(0,1)\), we shall show that, with probability
tending to one, there exist \(x^+,x^-\in\mathcal X_Y\) such that
\begin{equation}\label{eq:blockexceed}
\frac{M(x^+)}{\sqrt{2Q(x^+)\log_2x^+}}\ge c,
\qquad
\frac{M(x^-)}{\sqrt{2Q(x^-)\log_2x^-}}\le-c
\end{equation}
From this, the almost sure lower bounds are derived in Step~9.

Fix \(c\in(0,1)\). Choose \(0<\gamma<1/20\) so small that
\[
\frac{1-4\gamma}{\sqrt{1+\gamma}}>c.
\]
Throughout the proof, \(\gamma\) is fixed and \(Y\to\infty\).

\medskip\noindent
{\sc Step 1.}
All the points of the grid $\mathcal X_Y$, defined in \eqref{eq:blockparameters}, belong to
\(
\bigl[Y,Y^{1+(\log_2Y)^{-100}}\bigr].
\)
Consequently,
\begin{equation}\label{eq:blockproof-grid}
\begin{aligned}
|\mathcal X_Y|&\le\log Y,\\
\log|\mathcal X_Y|
&=\log_2Y-100\log_3Y+O(1),\\
\log_2x
&=\log_2Y+O((\log_2Y)^{-100}),
\qquad x\in\mathcal X_Y.
\end{aligned}
\end{equation}
The last estimate is uniform over the grid.

We group the terms in the largest-prime decomposition of
\(M(x)\) into layers according to their largest prime factors.
Put \(\delta=(\log_2Y)^{-1/20}\) and define the layer boundaries by
\[
y_j=\exp\bigg(
\frac{e^{j\delta}\log Y}{(\log_2Y)^2}
\bigg),
\qquad j\ge0.
\]
Let \(K\) be the least positive integer such that
\(
y_K\ge Y^{1+(\log_2Y)^{-100}}
\), so that every grid point is at most \(y_K\). In particular,
\begin{equation}\label{eq:layercount}
K=O\bigl((\log_2Y)^{1/20}\log_3Y\bigr).
\end{equation} Define the layer
\[
I_j=(y_{j-1},y_j],
\qquad 1\le j\le K.
\]

We shall construct a decreasing sequence of subsets of the grid $\mathcal X_Y$. The sets are nested, so every endpoint remaining at the end
satisfies the conditions imposed in all preceding selections. Before exposing layer \(j\), the current set will
be determined by the signs in the preceding layers. After
revealing the signs in \(I_j\), we retain only those endpoints
at which the layer sum exceeds a prescribed threshold with
the required sign. To carry out each selection conditionally on the preceding
signs, both the current candidate set and the threshold must
be known before the layer is exposed.
We also need the layer
sum to be a linear combination of its independent signs with
coefficients known at that time.

This motivates the following
decomposition.
Recall that
\(
a_p(x)=\Psi(x/p,p^-)
\)
is \(\FF_{p-}\)-measurable. This gives a martingale-difference
representation, but the coefficients within a layer need not
be determined before that layer is exposed. Indeed, if
\(q<p\) belong to \(I_j\) and \(pq\le x\), then \(a_p(x)\)
contains the term \(\varepsilon_q\). Thus the contribution
\(\varepsilon_pa_p(x)\) contains a product of two signs from
the same layer.

To obtain a linear combination of the layer signs with
coefficients determined before the layer is exposed, we put
\[
b_p(x)=\Psi(x/p,y_{j-1}),
\qquad
d_p(x)=a_p(x)-b_p(x),
\qquad p\in I_j.
\]
Then
\[
d_p(x)=
\sum_{\substack{m\le x/p\\y_{j-1}<P(m)<p}}f(m).
\]
All the coefficients \(b_p(x)\), \(p\in I_j\), are
\(\FF_{y_{j-1}}\)-measurable. Consequently, conditionally on
\(\FF_{y_{j-1}}\), the retained sum
\[
R_j(x)=\sum_{p\in I_j}\varepsilon_pb_p(x),
\qquad 1\le j \le K,
\]
is a linear combination of independent Rademacher signs
with fixed coefficients. This is the structure required
by the counting estimate used below.

We treat \(d_p(x)\) as an error. Every nonzero term discarded
in this way corresponds to an integer whose two largest
prime factors belong to the same layer. The thinness of
the layers allows us to control the total quadratic
variation of these errors.

For \(1\le j\le K\), put
\[
\begin{aligned}
C_j(x,z)&=\sum_{p\in I_j}b_p(x)b_p(z),\\
V_j(x)&=C_j(x,x)=\sum_{p\in I_j}b_p(x)^2.
\end{aligned}
\]
Conditionally on \(\FF_{y_{j-1}}\), all coefficients in
\(R_j(x)\) are fixed and the signs in \(I_j\) remain
independent Rademacher variables. Therefore,
\[
\E[R_j(x)\mid\FF_{y_{j-1}}]=0,
\qquad
\E[R_j(x)R_j(z)\mid\FF_{y_{j-1}}]=C_j(x,z).
\]
In particular,
\[
\Var(R_j(x)\mid\FF_{y_{j-1}})=V_j(x).
\]

Since \(y_K\ge x\) for every \(x\in\mathcal X_Y\), the following
decomposition is exact:
\begin{equation}\label{eq:blockproof-decomposition}
M(x)
=
\Psi(x,y_0)
+\sum_{j=1}^K R_j(x)
+\sum_{y_0<p\le y_K}\varepsilon_pd_p(x),
\qquad x\in\mathcal X_Y.
\end{equation}

Finally, define
\begin{equation}
\label{eq:q}
v_j=\max_{x\in\mathcal X_Y}\frac{V_j(x)}x,
\qquad
\mathcal J=
\bigl\{j:v_j\ge(\log_2Y)^{-59/100}\bigr\},
\qquad
q=\sum_{j\in\mathcal J}v_j,
\end{equation}
and
\[
S_*(x)=\sum_{j\in\mathcal J}R_j(x).
\]
A layer is called active precisely when \(j\in\mathcal J\).
The value \(v_j\), and hence the decision whether layer \(j\)
is active, is \(\FF_{y_{j-1}}\)-measurable. Thus this decision
is made before the signs in that layer are revealed.
The final value of \(q\), however, need not be known at that
time.

\medskip\noindent
{\sc Step 2.}
With the notation introduced above,
\begin{equation}\label{eq:blockproof-remainder}
M(x)-1-S_*(x)
=
\bigl(\Psi(x,y_0)-1\bigr)
+\sum_{j\notin\mathcal J}R_j(x)
+\sum_{y_0<p\le y_K}\varepsilon_pd_p(x).
\end{equation}
The three terms on the right-hand side above are
predictable Rademacher sums.
Define
\[
B_0(x)=\sum_{p\le y_0}a_p(x)^2,
\qquad
D(x)=\sum_{y_0<p\le y_K}d_p(x)^2.
\]
Then  \(B_0(x)\), \(\sum_{j\notin\mathcal J}V_j(x)\) and $D(x)$ are predictable quadratic variations of the corresponding terms on the right side of \eqref{eq:blockproof-remainder}.
Indeed, for the second sum, this follows because both \(b_p(x)\)
and the decision \(j\notin\mathcal J\) are
\(\FF_{y_{j-1}}\)-measurable.

We first show that, with probability \(1-o(1)\),
\begin{equation}
\label{eq:small_variance}
\max_{x\in \mathcal X_Y} \frac 1x\bigg(
B_0(x) + \sum_{j\notin\mathcal J}V_j(x) + D(x)\bigg) \le C(\log_2 Y)^{-53/100}.
\end{equation}

To estimate the first term,
orthogonality \eqref{eq:orth} and Lemma~\ref{lem:rankin}, with \(A=1\), give
\[
\E B_0(x)=\sfcount(x,y_0)-1
\le C x\exp\bigl(-\tfrac12(\log_2Y)^2\bigr).
\]
Hence, by Markov's inequality, a union bound and
\(|\mathcal X_Y|\le\log Y\),
\[
\PP\bigg(
\max_{x\in\mathcal X_Y}\frac{B_0(x)}x
>(\log Y)^{-20}
\bigg)
\le C
(\log Y)^{21}
\exp\bigl(-\tfrac12(\log_2Y)^2\bigr).
\]
For the second one, by the definition of \(\mathcal J\),
\begin{equation}\label{eq:blockproof-inactive}
\sup_{x\in\mathcal X_Y}\frac1x
\sum_{j\notin\mathcal J}V_j(x)
\le K(\log_2Y)^{-59/100}
=o\bigl((\log_2Y)^{-53/100}\bigr),
\end{equation}
where we used
\(K = O((\log_2Y)^{1/20}\log_3Y)\).

Finally, by Proposition~\ref{prop:flatness}, with probability $1-o(1)$,
\[
\max_{x\in \mathcal X_Y}\frac{D(x)}{x} \le 2
(\log_2Y)^{-53/100}.
\]
This proves \eqref{eq:small_variance}.

For the retained layers, Proposition~\ref{prop:goodlayers}
gives, with probability \(1-o(1)\),
\begin{equation}\label{eq:blockproof-layerbounds}
c_0(\log_2Y)^{-51/100}\le q\le2\log_2Y,\qquad
(1-(\log_2Y)^{-1})v_j
\le\frac{V_j(x)}x\le v_j
\end{equation}
simultaneously for all \(j\in\mathcal J\) and
\(x\in\mathcal X_Y\).
Summing over \(j\in\mathcal J\), we obtain
\begin{equation}\label{eq:blockproof-retainedclock}
(1-(\log_2Y)^{-1})xq
\le\sum_{j\in\mathcal J}V_j(x)\le xq.
\end{equation}
Thus the retained sum has predictable quadratic variation
asymptotic to \(xq\), uniformly over the grid. By
\eqref{eq:small_variance} and the lower bound for \(q\),
the quadratic variation of the omitted layers is \(o(xq)\).

We compare the full quadratic variation with the
retained one. Write \(a(x)=(a_p(x))_p\), and define
\[
a_{*,p}(x)=
\begin{cases}
b_p(x),&p\in I_j\text{ for some }j\in\mathcal J,\\
0,&\text{otherwise}.
\end{cases}
\]
Thus \(\|a(x)\|_2^2=Q(x)\), whereas
\(\|a_*(x)\|_2^2\) is the retained quadratic variation
in \eqref{eq:blockproof-retainedclock}.
On an active layer the coefficient error is \(d_p(x)\);
on an omitted layer it is \(b_p(x)+d_p(x)\).
Consequently, with probability $1-o(1)$
\[
\|a(x)-a_*(x)\|_2^2
\le B_0(x)+2D(x)
+2\sum_{j\notin\mathcal J}V_j(x)
\le Cx(\log_2Y)^{-53/100}.
\]
The triangle inequality and
\eqref{eq:blockproof-retainedclock} now give
\[
\begin{split}
\bigg|\sqrt{\frac{Q(x)}x}-\sqrt q\bigg|
&\le
\frac{\|a(x)-a_*(x)\|_2}{\sqrt x}
+\bigg|\frac{\|a_*(x)\|_2}{\sqrt x}-\sqrt q\bigg|\\
&\le C(\log_2Y)^{-53/200}
+C\sqrt q\,(\log_2Y)^{-1}.
\end{split}
\]
By the lower bound for \(q\) in
\eqref{eq:blockproof-layerbounds},
dividing by \(\sqrt q\) and then squaring, we obtain
\begin{equation}\label{eq:blockproof-fullclock}
\frac{Q(x)}{xq}
=1+O\bigl((\log_2Y)^{-1/100}\bigr),
\qquad x\in\mathcal X_Y.
\end{equation}
On the events under consideration, the error is bounded
uniformly by a deterministic function tending to zero.

It remains to control the omitted sums themselves.
Apply Lemma~\ref{lem:exponential} separately to the three
sums in \eqref{eq:blockproof-remainder}, with
\(
t=\tfrac13\sqrt{x}(\log_2Y)^{24/100},
v=Cx(\log_2Y)^{-53/100}.
\)
For each fixed endpoint, the probability that one of these
sums exceeds \(t\) in absolute value while its quadratic
variation is at most \(v\) is bounded by
\(
2\exp\bigl(-c_1(\log_2Y)^{101/100}\bigr)
\). Denote by $\mathcal G_Y$ the set on which all the above estimates hold.
A union bound over the three sums and all endpoints gives
\[
\begin{split}
&\PP\bigg(
\mathcal G_Y\cap
\bigg\{
\max_{x\in\mathcal X_Y}
\frac{|M(x)-1-S_*(x)|}{\sqrt x}
>(\log_2Y)^{24/100}
\bigg\}
\bigg)\\
&\qquad\le
6(\log Y)
\exp\bigl(-c_1(\log_2Y)^{101/100}\bigr)
=o(1).
\end{split}
\]
Since \(\PP(\mathcal G_Y)\to1\), with probability \(1-o(1)\),
\begin{equation}\label{eq:blockproof-smallremainder}
\max_{x\in\mathcal X_Y}
\frac{|M(x)-1-S_*(x)|}{\sqrt x}
\le(\log_2Y)^{24/100}.
\end{equation}

Together, \eqref{eq:blockproof-fullclock} and
\eqref{eq:blockproof-smallremainder} reduce the problem to
finding positive and negative values of \(S_*(x)\) close to
\(\sqrt{2xq\log_2Y}\). Thus the omitted terms are negligible at the required scale.

\medskip\noindent
{\sc Step 3. }. For each layer \(j\), join distinct endpoints \(x,z\in\mathcal X_Y\)
when
\begin{equation}
\label{eq:cov}
|C_j(x,z)|>2(\log_2Y)^{-3/5}\sqrt{xz}.
\end{equation}
We claim that, with probability \(1-o(1)\), all these graphs
satisfy
\begin{equation}\label{eq:layerdegree}
\max_{1\le j\le K}\max_{x\in\mathcal X_Y}
\#\{z\in\mathcal X_Y\setminus\{x\}:\text{\((x,z)\) satisfies \eqref{eq:cov}}\}
\le(\log Y)^\gamma.
\end{equation}

To verify this, apply Proposition~\ref{prop:covariance} with
the prime cutoffs \(y_{j-1},y_j\), the points of the grid
\(\mathcal X_Y\), and \(\theta=\gamma/2\).
Indeed,
\[
C_j(x,z)
=
\sum_{y_{j-1}<p\le y_j}
\Psi(x/p,y_{j-1})\Psi(z/p,y_{j-1}).
\]
We check that the assumptions of that proposition hold
uniformly in \(j\).
The layer construction gives \(y_j\le y_{j-1}^2\). The grid starts at \(Y\), which is larger than
\((\log y_{j-1})^9\) for all sufficiently large \(Y\), and
\[
|\mathcal X_Y|
\ll\frac{\log Y}{(\log_2Y)^{100}}
\le
\frac{\log y_{j-1}}{(\log_2 y_{j-1})^{10}}.
\]
Thus Proposition~\ref{prop:covariance} applies to every layer.

For all sufficiently large \(Y\), uniformly in \(j\),
\[
(\log_2 y_{j-1})^{-3/5}
\le2(\log_2Y)^{-3/5}.
\]
Hence every edge of our graph is also an edge of the graph
controlled by that proposition. Its maximum degree is therefore
at most
\[
(\log y_{j-1})^{\gamma/2}
\le(2\log Y)^{\gamma/2}
\le(\log Y)^\gamma,
\]
outside the exceptional event in
Proposition~\ref{prop:covariance}.

Finally, a union bound over the \(K\) layers bounds the
probability of any failure by
\[
O_\gamma\bigg(
K\frac{(\log_3Y)^4}{(\log_2Y)^{1/10}}
\bigg)
=o(1).
\]

\medskip\noindent
{\sc Step 4.}
Before revealing the signs in \(I_j\), standardize the retained
sums by
\[
Z_{j,x}=\frac{R_j(x)}{\sqrt{V_j(x)}},\qquad x\in\mathcal X_Y.
\]
On the event in Proposition~\ref{prop:goodlayers}, these are
conditionally centered Rademacher sums of variance one. Their
coefficients satisfy \eqref{eq:smallcoeff}, that is
\begin{equation}
\label{eq:c6}
\max_{\substack{j\in\mathcal J,\ p\in I_j\\x\in\mathcal X_Y}}
\frac{|b_p(x)|}{\sqrt{V_j(x)}}
\le\beta_Y:=e^{-(\log Y)/(3(\log_2Y)^2)}.
\end{equation}

Fix an active layer \(j\in\mathcal J\) and distinct endpoints
\(x,z\in\mathcal X_Y\) not joined by an edge in Step~3.
Since \(V_j(x)\) and \(V_j(z)\) are
\(\FF_{y_{j-1}}\)-measurable,
\[
\E[Z_{j,x}Z_{j,z}\mid\FF_{y_{j-1}}]
=
\frac{C_j(x,z)}{\sqrt{V_j(x)V_j(z)}}.
\]
The absence of an edge means that
\(
|C_j(x,z)|
\le2(\log_2Y)^{-3/5}\sqrt{xz}.
\). By \eqref{eq:blockproof-layerbounds} and the definition
of \(\mathcal J\), we obtain
\[
\big|\E[Z_{j,x}Z_{j,z}\mid\FF_{y_{j-1}}]\big|
\le
\frac{2(\log_2Y)^{-3/5}}
{\bigl(1-(\log_2Y)^{-1}\bigr)v_j}
\le4(\log_2Y)^{-3/5+59/100}
=4(\log_2Y)^{-1/100}.
\]
Together with the degree bound, these are the hypotheses of
Lemma~\ref{lem:count}.
We first state the counting estimate for an arbitrary threshold;
the thresholds used in the construction will be chosen in Step~5.

Fix an active layer \(j\). Let \(S\subset\mathcal X_Y\) and
\(u\ge0\) be determined by \(\FF_{y_{j-1}}\). Here \(S\) is the set of endpoints
retained from the preceding layers, and \(u\) is the threshold
for the normalized sum \(R_j(x)/\sqrt{V_j(x)}\).

Put
\[
u_0=C_{\gamma,10}\sqrt{\log_3Y}.
\]
If
\[
u\ge u_0,\qquad
|S|e^{-(1+\gamma)u^2/2}\ge(\log Y)^{3\gamma},
\]
then, for either fixed sign \(s\in\{-1,1\}\), at least
\[
\frac{|S|}{8}e^{-(1+\gamma)u^2/2}
\]
points of \(S\) satisfy \(sR_j(x)\ge u\sqrt{V_j(x)}\) with probability
\(1-O_\gamma((\log_2Y)^{-10})\).

\medskip\noindent
{\sc Step 5.}

We choose the thresholds by comparing two quantities:
the accumulated contribution to \(S_*(x)\) and the number
of endpoints remaining after each selection.
At every active layer, we retain only those endpoints at
which the layer sum exceeds the chosen threshold.
A higher threshold gives a larger guaranteed contribution
to \(S_*(x)\), but leaves fewer endpoints available for the
subsequent layers. The thresholds must therefore be large
enough to produce the required fluctuation and small enough
for the selection to continue through all active layers,
leaving a nonempty final set.

The sets are nested. Hence every endpoint in the final set
satisfies all the preceding threshold conditions, and the
corresponding contributions add up at that same endpoint.
We seek a surviving point \(x\in\mathcal X_Y\) at which
\(S_*(x)/\sqrt x\) is close to \(\sqrt{2q\log_2Y}\).
The same construction, with the opposite sign, gives the
negative fluctuation.

We first examine how an unspecified threshold \(u_j\) affects
these two quantities. We shall then use the resulting
estimates to choose the thresholds. Suppose for the moment
that, in layer \(j\), we retain endpoints satisfying
\[
R_j(x)\ge u_j\sqrt{V_j(x)},
\]
where \(u_j\) is to be chosen.
By \eqref{eq:blockproof-layerbounds}, each retained endpoint
then contributes at least
\[
\frac{R_j(x)}{\sqrt x}
\ge u_j\sqrt{(1-(\log_2Y)^{-1})v_j}.
\]
We introduce a parameter $\lambda>0$, to be chosen later and define
\[
u_j = \max\bigg\{
\lambda\sqrt{\frac{v_j}{1-(\log_2Y)^{-1}}}, u_0 \bigg\}.
\]
If an endpoint survives every active layer, its total
contribution then satisfies
\[
\frac{S_*(x)}{\sqrt x}
\ge\lambda q.
\]
Thus a value of \(\lambda\) close to
\(\sqrt{2\log_2Y/q}\) would give a value of \(S_*(x)\)
close to \(\sqrt{2xq\log_2Y}\), as required.

At the same time, the thresholds must allow the selection
to continue. By Step 4, a selection with threshold \(u_j\)
retains at least
\(
({|S|}/{8})e^{-(1+\gamma)u_j^2/2}
\). Taking logarithms, this means
\[
\log|S_{\rm new}|
\ge\log|S|-\log8-\frac{1+\gamma}{2}u_j^2.
\]
Thus, if all active-layer selections are performed and their
counting bounds hold, the final set satisfies
\[
\log|S_{\rm final}|
\ge
\log|\mathcal X_Y|
-|\mathcal J|\log8
-\sum_{j\in\mathcal J}\frac{1+\gamma}{2}u_j^2.
\]
If we temporarily omit the minimum threshold \(u_0\),
the sum of the exponential losses over all active layers is
\[
\sum_{j\in\mathcal J}\frac{1+\gamma}{2}u_j^2
=
\frac{(1+\gamma)\lambda^2q}
{2(1-(\log_2Y)^{-1})}.
\]
Since \(\log|\mathcal X_Y|\sim\log_2Y\) and
\(K=o(\log_2Y)\), a fixed fractional margin in the main
exponential cost can absorb the losses \(K\log8\) and
\(Ku_0^2\). Step~7 verifies this explicitly.

One difficulty remains in choosing the thresholds: the value
of \(q\) is random and depends on the coefficients in all
active layers. It is therefore not known when the first
thresholds must be specified.
To overcome this, we replace \(q\) by a deterministic parameter
\(w\) and carry out the construction for every \(w\) in a
finite deterministic set. After all layers have been exposed,
we choose a value of \(w\) sufficiently close to the realized
value of \(q\).

Let
\begin{equation}\label{eq:guesses}
\mathcal W_Y=
\bigg\{(\log_2Y)^{-1}(1+\gamma)^a:
a=0,\ldots,
\bigg\lceil\frac{2\log_3Y+\log2}{\log(1+\gamma)}\bigg\rceil
\bigg\}.
\end{equation}
This set has \(O_\gamma(\log_3Y)\) elements. Its smallest
element is \((\log_2Y)^{-1}\), and its largest element is
at least \(2\log_2Y\). On the event in
\eqref{eq:blockproof-layerbounds}, for sufficiently large \(Y\),
\[
(\log_2Y)^{-1}\le q\le2\log_2Y.
\]
Choose the smallest element \(w\in\mathcal W_Y\) with \(w\ge q\).
Since consecutive elements differ by a factor \(1+\gamma\),
this choice satisfies
\begin{equation}\label{eq:correctguess}
q\le w\le(1+\gamma)q.
\end{equation}
We use this choice only after all constructions have been
performed.

For each deterministic \(w\in\mathcal W_Y\), put
\begin{equation}\label{eq:thresholds}
\lambda_w=(1-4\gamma)
\sqrt{\frac{2\log_2Y}{w}},
\qquad
u_j(w)=
\max\bigg\{
u_0,\,
\lambda_w\sqrt{\frac{v_j}{1-(\log_2Y)^{-1}}}
\bigg\}.
\end{equation}

\medskip\noindent
{\sc Step 6.}
The parameters and the rule for choosing the thresholds have
now been specified. We next define the successive sets of
retained endpoints and verify that the required conditions
hold throughout the construction. In particular, we must
check that the counting estimate applies at every active
layer and that the final sets are nonempty.

For each \(s\in\{-1,1\}\) and each \(w\in\mathcal W_Y\),
start with \(S_0^{s,w}=\mathcal X_Y\).
If \(j\notin\mathcal J\), put \(S_j^{s,w}=S_{j-1}^{s,w}\).
If \(j\in\mathcal J\), check the local variance comparison, that is, the second part of
\eqref{eq:blockproof-layerbounds}, the coefficient bound \eqref{eq:c6},
and the degree bound \eqref{eq:layerdegree} for this layer.
All three checks are \(\FF_{y_{j-1}}\)-measurable.
They do not involve the final value of \(q\).
Stop this construction if any check fails.

Otherwise test
\begin{equation}\label{eq:branchbudget}
|S_{j-1}^{s,w}|e^{-(1+\gamma)u_j(w)^2/2}
\ge(\log Y)^{3\gamma}.
\end{equation}
If this fails, stop this construction. If it holds, expose the
signs in \(I_j\) and retain
\[
S_j^{s,w}=
\bigg\{x\in S_{j-1}^{s,w}:
sR_j(x)\ge u_j(w)\sqrt{V_j(x)}\bigg\}.
\]
Step 4 gives
\begin{equation}\label{eq:transition}
|S_j^{s,w}|\ge\frac{|S_{j-1}^{s,w}|}{8}
e^{-(1+\gamma)u_j(w)^2/2}
\end{equation}
except with conditional probability
\(O_\gamma((\log_2Y)^{-10})\) at each tested transition.
There are at most \(2K|\mathcal W_Y|\) transitions. Taking
expectations of the conditional bounds and then a union bound
shows that the probability of any tested transition failing is
at most
\begin{equation}\label{eq:alltransition}
O_\gamma\bigl(K(\log_3Y)(\log_2Y)^{-10}\bigr)=o(1).
\end{equation}

Let \(\mathcal H_Y\) be the intersection of the event that
all tested transitions satisfy \eqref{eq:transition}, the event
in Proposition~\ref{prop:goodlayers}, and the events in Step 3.
Then \(\PP(\mathcal H_Y)\to1\).

\medskip\noindent
{\sc Step 7.}
Fix a realization in \(\mathcal H_Y\). Choose a value
\(w\in\mathcal W_Y\) satisfying \eqref{eq:correctguess}.
Both constructions for this value have already been included
in \eqref{eq:alltransition}.
Using \(\max\{a,b\}^2\le a^2+b^2\) for \(a,b\ge0\), we get
\begin{align}
\sum_{j\in\mathcal J}\frac{1+\gamma}{2}u_j(w)^2
\le\frac{1+\gamma}{2}Ku_0^2
+\frac{(1+\gamma)(1-4\gamma)^2}
{1-(\log_2Y)^{-1}}\frac qw\log_2Y
\le(1-5\gamma)\log_2Y
\label{eq:totalbudget}
\end{align}
for sufficiently large \(Y\). Indeed,
\(Ku_0^2=O_\gamma\bigl((\log_2Y)^{1/20}(\log_3Y)^2\bigr)=o(\log_2Y)\),
\(q/w\le1\), and
\((1+\gamma)(1-4\gamma)^2=1-7\gamma+8\gamma^2+16\gamma^3\le1-6\gamma\)
for \(0<\gamma<1/20\).

We now verify that neither construction stops. Suppose the
preceding transitions have succeeded. Iterating
\eqref{eq:transition} gives, immediately before an active layer,
\begin{align*}
\log|S_{j-1}^{s,w}|-\frac{1+\gamma}{2}u_j(w)^2
&\ge\log|\mathcal X_Y|-K\log8
-\sum_{i\in\mathcal J}\frac{1+\gamma}{2}u_i(w)^2\\
&\ge5\gamma\log_2Y-100\log_3Y-K\log8+O(1)\\
&\ge4\gamma\log_2Y.
\end{align*}
Thus \eqref{eq:branchbudget} holds. The local assumptions also
hold on \(\mathcal H_Y\). Induction proves that all transitions
are performed and that both final sets are nonempty; in fact,
their cardinalities are at least \(\exp(4\gamma\log_2Y)/8\)
for sufficiently large \(Y\).

Choose \(x^+\in S_K^{1,w}\), and write \(x=x^+\) in the
following calculation.
Because the sets are nested, an endpoint in the final set has
satisfied every earlier active-layer inequality. For every
\(j\in\mathcal J\), we have
\begin{align*}
\frac{R_j(x)}{\sqrt{x}}
\ge u_j(w)\sqrt{\frac{V_j(x)}x}
\ge\lambda_w
\sqrt{\frac{v_j}{1-(\log_2Y)^{-1}}}
\sqrt{(1-(\log_2Y)^{-1})v_j}
=\lambda_wv_j.
\end{align*}
Summing over the active layers, at this same endpoint, gives
\[
\frac{S_*(x)}{\sqrt{x}}
\ge\lambda_w\sum_{j\in\mathcal J}v_j
=\lambda_wq
=(1-4\gamma)\sqrt{2q\log_2Y}\sqrt{\frac qw}.
\]
As \(w\le(1+\gamma)q\), we obtain
\begin{equation}\label{eq:signedgain}
\frac{S_*(x^+)}{\sqrt{x^+}}
\ge\frac{1-4\gamma}{\sqrt{1+\gamma}}\sqrt{2q\log_2Y}
\end{equation}
Choose \(x^-\in S_K^{-1,w}\). The corresponding inequalities
with \(s=-1\) give
\begin{equation}\label{eq:signedgain1}
\frac{S_*(x^-)}{\sqrt{x^-}}
\le-\frac{1-4\gamma}{\sqrt{1+\gamma}}\sqrt{2q\log_2Y}.
\end{equation}

\medskip\noindent
{\sc Step 8.}
By \eqref{eq:blockproof-layerbounds},
\(
\sqrt{q\log_2Y}\ge\sqrt{c_0}(\log_2Y)^{49/200}
\), therefore
\begin{equation}\label{eq:blockproof-errorratio}
\frac{(\log_2Y)^{24/100}}{\sqrt{q\log_2Y}}
\le C (\log_2Y)^{-1/200}\longrightarrow0.
\end{equation}
Equations \eqref{eq:blockproof-smallremainder} and
\eqref{eq:blockproof-errorratio} therefore show that all three
omitted sums in \eqref{eq:blockproof-remainder}, as well as the
constant \(1\), are negligible relative to
\(\sqrt{2xq\log_2Y}\).
By \eqref{eq:blockproof-grid} and
\eqref{eq:blockproof-fullclock}, uniformly over the grid,
\[
\frac{\sqrt{2Q(x)\log_2x}}{\sqrt{2xq\log_2Y}}
=\sqrt{\frac{Q(x)}{xq}\frac{\log_2x}{\log_2Y}}
=1+O\bigl((\log_2Y)^{-1/100}\bigr).
\]
Intersect \(\mathcal H_Y\) with the good event from Step~2.
This intersection still has probability tending to one.
By \eqref{eq:signedgain} and \eqref{eq:signedgain1}, for both
\((s,x)=(1,x^+)\) and \((s,x)=(-1,x^-)\),
\[
\frac{sM(x)}{\sqrt{2Q(x)\log_2x}}
\ge\frac{1-4\gamma}{\sqrt{1+\gamma}}-o(1) \ge c,
\]
for all sufficiently large \(Y\) on this intersection. This proves
\eqref{eq:blockexceed}.

\medskip\noindent
{\sc Step 9.}
Let \((Y_r)\) be any deterministic sequence tending to infinity.
Let \(E_r(c)\) denote the event that both inequalities
in \eqref{eq:blockexceed} hold at some points of
\(\mathcal X_{Y_r}\). By Step~8,
\(\PP(E_r(c))\to1\).

For every positive integer \(N\),

\[
\PP\biggl(\bigcup_{r\ge N}E_r(c)\biggr)
\ge\sup_{r\ge N}\PP(E_r(c))=1.
\]

Consequently, the event
\[
A(c)=\bigcap_{N\ge1}\bigcup_{r\ge N}E_r(c)
\]
has probability one. On this event, both signed exceedances
occur in infinitely many blocks. Since every point of
\(\mathcal X_{Y_r}\) is at least \(Y_r\), and \(Y_r\to\infty\),
we obtain, almost surely,
\[
\limsup_{x\to\infty}
\frac{M(x)}{\sqrt{2Q(x)\log_2x}}\ge c,
\qquad
\liminf_{x\to\infty}
\frac{M(x)}{\sqrt{2Q(x)\log_2x}}\le-c.
\]
Letting \(c\uparrow1\) along a countable sequence, we arrive at the lower bounds. Together with
Section~\ref{sec:upper}, this gives \eqref{eq:main}. Finally,
\eqref{eq:basicclock} implies
\(\log_2Q(x)/\log_2x\to1\), which proves the assertion
with \(\log_2Q(x)\) in the denominator.
\end{proof}

\begin{proof}[Proof of Corollary~\ref{cor:cluster}]
Put \(\varrho(x)=\sqrt{2Q(x)\log_2x}\) for \(x\ge3\) and
\(T_n=M(n)/\varrho(n)\) for integer \(n\ge3\). For \(x\ge3\), we have
\(M(x)=M(\lfloor x\rfloor)\) and \(Q(x)=Q(\lfloor x\rfloor)\), because
\(a_p(x)=a_p(\lfloor x\rfloor)\) for each prime \(p\). Hence,
\[
\frac{M(x)}{\varrho(x)}=T_{\lfloor x\rfloor}
\bigg(\frac{\log_2\lfloor x\rfloor}{\log_2x}\bigg)^{1/2},\qquad x\ge3.
\]
By Theorem~\ref{thm:main}, the sequence \((T_n)\) is eventually bounded
almost surely. Since \(\log_2\lfloor x\rfloor/\log_2x\to1\) as
\(x\to\infty\), the cluster set in \eqref{eq:cluster} coincides with the set of
subsequential limits of \((T_n)\), and
\[
\limsup_{n\to\infty}T_n=1,\qquad\liminf_{n\to\infty}T_n=-1
\qquad\text{a.s.}
\]

For integer \(n\ge3\), we have \(M(n)-M(n-1)=f(n)\). Also, by
\eqref{eq:vector}, \(a(n)-a(n-1)=\mu^2(n)f(n/P(n))e_{P(n)}\). Thus,
\[
|M(n)-M(n-1)|\le1
\qquad\text{and}\qquad
\big|\sqrt{Q(n)}-\sqrt{Q(n-1)}\big|\le\|a(n)-a(n-1)\|_2\le1.
\]
In view of the lower bound in \eqref{eq:basicclock}, this entails
\(\lim_{n\to\infty}Q(n-1)/Q(n)=1\) and thereupon
\(\lim_{n\to\infty}\varrho(n-1)/\varrho(n)=1\). Writing
\[
T_n-T_{n-1}=\frac{M(n)-M(n-1)}{\varrho(n)}
+T_{n-1}\bigg(\frac{\varrho(n-1)}{\varrho(n)}-1\bigg)
\]
and using the eventual boundedness of \((T_n)\), we conclude that
\(\lim_{n\to\infty}(T_n-T_{n-1})=0\) a.s.

Fix a sample point at which the last limit relation and both limit
relations for \((T_n)\) hold. For any \(z\in(-1,1)\), the sequence
\((T_n)\) crosses \(z\) upwards infinitely often. Hence, there exists an
increasing sequence \((n_k)\), depending on the chosen sample point, such
that \(T_{n_k-1}\le z<T_{n_k}\). Therefore,
\[
0<T_{n_k}-z\le T_{n_k}-T_{n_k-1}\to0,\qquad k\to\infty,
\]
so that \(z\) is a subsequential limit. The endpoints \(-1\) and \(1\) are
subsequential limits, and there are no subsequential limits outside
\([-1,1]\). This proves \eqref{eq:cluster}. Finally, \eqref{eq:basicclock}
implies \(\log_2Q(x)/\log_2x\to1\), which proves the assertion with
\(\log_2Q(x)\) in the denominator.
\end{proof}


\section{Deterministic normalization}\label{sec:deterministic}
\begin{lemma}\label{dn:limits}
Put {\(Z_x=Q(x)\sqrt{\log_2x}/x\)}.
Then, for some \(x_0>e\), the family \((Z_x)_{x\ge x_0}\) is tight.
On the other hand, for
every fixed \(B>0\), there is \(\delta_B>0\) such that
\begin{equation}\label{dn:tail}
\PP(Z_x>B)\ge\delta_B
\qquad\text{for all sufficiently large }x.
\end{equation}

\end{lemma}

\begin{proof}

\medskip
\noindent
{\sc Step 1.}
We first prove that, uniformly for \(3/4\le r\le1\) and
all sufficiently large \(x\),

\begin{equation}\label{dn:z-moments}
\frac{c}{((\log_2x)^{-1/2}+1-r)^r}
\le\E Z_x^r\le
\frac{C}{((\log_2x)^{-1/2}+1-r)^r},
\end{equation}
for some constants  \(c,C>0\).

We use
Harper's theorem for squarefree Rademacher functions
\cite[Theorem~2]{HarperLow}, which says that uniformly for large \(x\)
and \(0\le s\le1\),

\begin{equation}\label{dn:harper}
\E|M(x)|^{2s}\asymp
\biggl(\frac{x}{1+(1-s)\sqrt{\log_2x}}\biggr)^s.
\end{equation}

For every fixed $x$,
the martingale structure \eqref{eq:decomp} combined with  Burkholder's square-function inequality
\cite[Theorem~3.1]{Burkholder} allows us to compare for
$3/4\le r\le1$ moments of $M(x)-1$ and $Q(x)$:
\begin{equation}\label{dn:square-bound}
\frac14\E|M(x)-1|^{2r}
\le\E Q(x)^r\le4\E|M(x)-1|^{2r}.
\end{equation}
Indeed, we can apply Burkholder's inequality to the martingale-difference sequence in \eqref{eq:decomp}, with \(p=2r\).
For \(3/4\le r\le1\), we have \(3/2\le p\le2\) and
\(p^*-1\le2\), where \(p^*=\max\{p,p/(p-1)\}\).
Moreover, the square function is exactly \(\sqrt{Q(x)}\),
since \((\varepsilon_pa_p(x))^2=a_p(x)^2\).

To compare moments of  $M(x)-1$ and  of  $M(x)$, we use
Minkowski's inequality
\[
\big|\|M(x)-1\|_{2r}-\|M(x)\|_{2r}\big|\le1,
\]
together with the lower bound in \eqref{dn:harper},
\[
\inf_{3/4\le r\le1}\|M(x)\|_{2r}
\ge c\sqrt{x/\sqrt{\log_2x}}\longrightarrow\infty,
\]
to conclude that  for all sufficiently large \(x\), uniformly in
\(3/4\le r\le1\),
\[
\frac12\|M(x)\|_{2r}
\le\|M(x)-1\|_{2r}\le2\|M(x)\|_{2r}.
\]
Combining \eqref{dn:square-bound} and
\eqref{dn:harper} yields

\begin{equation}\label{dn:q-moments}
\E Q(x)^r\asymp
\biggl(\frac{x}{1+(1-r)\sqrt{\log_2x}}\biggr)^r,
\qquad 3/4\le r\le1.
\end{equation}

Multiplying both sides  by \((\sqrt{\log_2x}/x)^r\) gives \eqref{dn:z-moments}.

\medskip
\noindent
{\sc Step 2.}
The upper bound in \eqref{dn:z-moments}, at \(r=3/4\), gives
\(\sup_{x\ge x_0}\E Z_x^{3/4}<\infty\). Markov's inequality
therefore proves tightness.

\medskip
\noindent
{\sc Step 3.}
To prove \eqref{dn:tail}, fix \(B>0\). The lower bound
in \eqref{dn:z-moments} allows us to choose
\(r\in[3/4,1)\) such that
\[
\E Z_x^r\ge B^r+1
\]
for all sufficiently large \(x\).
Fix \(s\in(r,1)\). The upper bound in
\eqref{dn:z-moments} gives
\(\E Z_x^s\le C_s\) for all sufficiently large \(x\),
where \(C_s\ge1\) is finite.
By H\"older's inequality,
\[
\E Z_x^r
\le B^r+
\E\big[Z_x^r\1_{\{Z_x>B\}}\big]
\le B^r+
C_s^{r/s}\PP(Z_x>B)^{1-r/s}.
\]
Consequently,
\[
\PP(Z_x>B)\ge C_s^{-r/(s-r)}.
\]
This proves \eqref{dn:tail} with
\(\delta_B=C_s^{-r/(s-r)}>0\).
\end{proof}

\begin{proof}[Proof of Theorem~\ref{dn:clock}]
Suppose that \(Q(x_k)/q_k\to1\) in probability. Put
\(b_k=q_k\sqrt{\log_2{x_k}}/x_k\). Then
\begin{equation}\label{dn:deterministic-ratio}
Z_{x_k}=b_k\,\frac{Q(x_k)}{q_k}.
\end{equation}
If \(b_k\to\infty\) along a subsequence, then \(Z_{x_k}\to\infty\)
in probability along that subsequence, contradicting tightness.
If \(b_k\to0\) along a subsequence, then \(Z_{x_k}\to0\) in
probability, contradicting Lemma~\ref{dn:limits}.
Thus \((b_k)\) is eventually bounded above and bounded away
from zero. Choose a subsequence along which
\(b_k\to b\in(0,\infty)\). By \eqref{dn:deterministic-ratio},
\(Z_{x_k}\to b\) in probability, contradicting \eqref{dn:tail}, because
\(\PP(Z_{x_k}>b+1)\to0\).
The almost sure assertion follows since almost sure convergence
implies convergence in probability.
\end{proof}

\begin{proof}[Proof of Theorem~\ref{dn:critical}]
Recall \(b(x)\) from \eqref{dn:amplitude-definition}. Then
\[
\frac{M(x)}{b(x)}
=
\frac{M(x)}{\sqrt{2Q(x)\log_2x}}\sqrt{2Z_x}.
\]
We shall use this identity to transfer large values of \(Z_Y\)
to positive and negative excursions of \(M(x)/b(x)\) at
points in the same block.

\medskip\noindent
{\sc Step 1. }
Consider the deterministic grid $\mathcal X_Y$ defined in \eqref{eq:blockparameters}.
By Step~8 of the proof of the lower bounds in Theorem \ref{thm:main}, with
\(c=1/2\), and the clock comparison
\eqref{eq:blockproof-fullclock}, there are events \(E_Y\),
with \(\PP(E_Y)\to1\), on which the following two conclusions
hold simultaneously:
first, there exist \(x_+,x_-\in\mathcal X_Y\) such that
\begin{equation}\label{dn:block-signs}
\frac{M(x_+)}{\sqrt{2Q(x_+)\log_2x_+}}\ge\frac12,
\qquad
\frac{-M(x_-)}{\sqrt{2Q(x_-)\log_2x_-}}\ge\frac12,
\end{equation}
and \begin{equation}\label{dn:flatness}
\sup_{x\in\mathcal X_Y}
\bigg|\frac{Q(x)}{xq_Y}-1\bigg|
\le e_Y,
\qquad e_Y\longrightarrow0,
\end{equation}
where \(q_Y>0\) is the total \(q\) from \eqref{eq:q},
and \(e_Y\) is deterministic.

Since \(Y\in\mathcal X_Y\), \eqref{dn:flatness} gives
\[
\frac{Q(x)/x}{Q(Y)/Y}
=
\frac{Q(x)/(xq_Y)}{Q(Y)/(Yq_Y)}
=1+O(e_Y),
\qquad x\in\mathcal X_Y,
\]
uniformly on \(E_Y\).

By \eqref{eq:blockproof-grid},
\(\log_2x/\log_2Y=1+O((\log_2Y)^{-101})\) uniformly
on the grid. Combining these estimates gives
\begin{equation}\label{dn:normalized-flatness}
\frac{Z_x}{Z_Y}
=
\frac{Q(x)/x}{Q(Y)/Y}
\sqrt{\frac{\log_2x}{\log_2Y}}
=
1+O(e_Y)+O((\log_2Y)^{-101}),
\qquad x\in\mathcal X_Y.
\end{equation}
The error bound is deterministic and uniform on \(E_Y\).
In particular, for all sufficiently large \(Y\) on $E_Y$, for all \(x\in\mathcal X_Y\),
\[
Z_x\ge\frac12 Z_Y.
\]

\medskip\noindent
{\sc Step 2.}
At the endpoints in \eqref{dn:block-signs},
\[
\begin{split}
\frac{M(x_+)}{b(x_+)}
&\ge\frac12\sqrt{2Z_{x_+}}
\ge\frac12\sqrt{Z_Y},\\
\frac{-M(x_-)}{b(x_-)}
&\ge\frac12\sqrt{2Z_{x_-}}
\ge\frac12\sqrt{Z_Y}.
\end{split}
\]
Thus, for all sufficiently large \(Y\),
\begin{equation}\label{dn:block-transfer}
\max_{x\in\mathcal X_Y}\frac{M(x)}{b(x)}
\ge\frac12\sqrt{Z_Y},
\qquad
\max_{x\in\mathcal X_Y}\frac{-M(x)}{b(x)}
\ge\frac12\sqrt{Z_Y}
\quad\text{on }E_Y.
\end{equation}

Fix \(R>0\), and put \(B=16R^2\).
By \eqref{dn:tail}, there is \(\delta_B>0\) such that
\[
\PP(Z_Y>B)\ge\delta_B
\]
for all sufficiently large \(Y\).
Define
\(
\mathcal A_Y=E_Y\cap\{Z_Y>B\}
\), then
\[
\PP(\mathcal A_Y)
=\PP(Z_Y>B)-\PP(\{Z_Y>B\}\cap E_Y^c)
\ge\delta_B-\PP(E_Y^c).
\]
It follows that
\[
\liminf_{Y\to\infty}\PP(\mathcal A_Y)\ge\delta_B.
\]
On \(\mathcal A_Y\), \eqref{dn:block-transfer} yields
\[
\max_{x\in\mathcal X_Y}\frac{M(x)}{b(x)}>2R,
\qquad
\max_{x\in\mathcal X_Y}\frac{-M(x)}{b(x)}>2R.
\]

\medskip\noindent
{\sc Step 3.}
Let \((Y_j)\) be any deterministic sequence tending to infinity,
and put
\[
\mathcal A
=
\bigcap_{N\ge1}\bigcup_{j\ge N}\mathcal A_{Y_j}.
\]
Then, \[
\PP(\mathcal A)
=\lim_{N\to\infty}
\PP\bigg(\bigcup_{j\ge N}\mathcal A_{Y_j}\bigg)\ge\lim_{N\to\infty}
\sup_{j\ge N}\PP(\mathcal A_{Y_j})=\limsup_{j\to\infty}\PP(\mathcal A_{Y_j})
\ge\delta_B>0.
\]

For every realization in \(\mathcal A\), there are infinitely
many indices \(j\) for which we can choose
\(x_{j,+},x_{j,-}\in\mathcal X_{Y_j}\) satisfying
\[
\frac{M(x_{j,+})}{b(x_{j,+})}>2R,
\qquad
\frac{-M(x_{j,-})}{b(x_{j,-})}>2R.
\]
Since the selected endpoints tend to infinity, both signed
limsups are at least \(2R\), and hence greater than \(R\),
on the same event of positive probability. Finally,
\(R^*\) dominates both signed limsups, so
\(\PP(R^*>R)>0\) for every finite \(R>0\).
\end{proof}


\begin{proof}[Proof of Proposition~\ref{dn:zero-one}]

\medskip\noindent
{\sc Step 1. }
Fix a prime \(p\), and let \(\widehat f\) be obtained from \(f\)
by replacing \(\varepsilon_p\) by \(-\varepsilon_p\), leaving
all other prime signs unchanged. Write
\[
\widehat M(x)=\sum_{n\le x}\widehat f(n),\qquad
\widehat R^*=\limsup_{x\to\infty}
\frac{|\widehat M(x)|}{b(x)}.
\]
The sum over integers not divisible by \(p\),
\[
M_p(x)=\sum_{\substack{n\le x\\p\nmid n}}f(n),
\]
is the same for the two sign configurations.
Every squarefree integer divisible by \(p\) is uniquely
\(pm\), with \(p\nmid m\). Therefore,
\[
M(x)=M_p(x)+\varepsilon_pM_p(x/p),\qquad
\widehat M(x)=M_p(x)-\varepsilon_pM_p(x/p).
\]
Rearranging the first identity gives
\(M_p(x)=M(x)-\varepsilon_pM_p(x/p)\).
Iterating this relation, with all sums zero below one, and writing
\(k_x=\lfloor\log x/\log p\rfloor\), gives
\[
M_p(x)=\sum_{k=0}^{k_x}(-\varepsilon_p)^kM(x/p^k).
\]
Substitution into the second identity yields the exact formula
\begin{equation}\label{dn:prime-flip-identity}
\widehat M(x)
=M(x)+2\sum_{k=1}^{k_x}(-\varepsilon_p)^kM(x/p^k).
\end{equation}

\medskip\noindent
{\sc Step 2. }
Suppose first that \(R^*<\infty\), and fix \(\epsilon>0\).
Choose \(A>e^e\), depending on the sign configuration, such that
\[
|M(u)|\le(R^*+\epsilon)b(u),\qquad u\ge A.
\]
For \(x/p^k\ge A\), the definition of \(b\) gives
\[
\frac{b(x/p^k)}{b(x)}
=p^{-k/2}
\bigg(\frac{\log_2(x/p^k)}{\log_2x}\bigg)^{1/4}
\le p^{-k/2}.
\]
For the remaining terms in \eqref{dn:prime-flip-identity},
the arguments lie in \([1,A)\). There are at most
\(1+\log A/\log p\) such terms, and each has absolute
value at most \(A\). Thus, for \(x\ge A\),
\[
\begin{split}
\frac{|\widehat M(x)|}{b(x)}
&\le(R^*+\epsilon)
\bigg(1+2\sum_{k\ge1}p^{-k/2}\bigg)
+\frac{2A(1+\log A/\log p)}{b(x)}\\
&=(R^*+\epsilon)\frac{\sqrt p+1}{\sqrt p-1}
+\frac{2A(1+\log A/\log p)}{b(x)}.
\end{split}
\]
Taking the limsup and then letting \(\epsilon\downarrow0\)
shows that
\(\widehat R^*\le((\sqrt p+1)/(\sqrt p-1))R^*\).
This inequality is trivial when \(R^*=\infty\).
By symmetry,  the same bound holds in the reverse direction. Consequently, in the extended
nonnegative real line,
\begin{equation}\label{dn:prime-flip-comparison}
\frac{\sqrt p-1}{\sqrt p+1}R^*
\le\widehat R^*
\le\frac{\sqrt p+1}{\sqrt p-1}R^*.
\end{equation}
In particular, both zero and finiteness are preserved.
Iterating the comparison proves invariance of the events
\(\{R^*=0\}\) and \(\{R^*<\infty\}\) under changes of
finitely many prime signs.

\medskip\noindent
{\sc Step 3. }
Enumerate the primes as \(p_1<p_2<\cdots\), and put
\(\FF_\infty=\sigma(\varepsilon_p:p\text{ prime})\).
Since \(M\) is constant on \([n,n+1)\) and \(b\) is eventually
increasing,
\[
R^*=\limsup_{n\to\infty}\frac{|M(n)|}{b(n)}.
\]
Thus \(R^*\) is an extended nonnegative
\(\FF_\infty\)-measurable random variable.

Fix \(m\ge1\), set the signs at \(p_1,\ldots,p_m\) equal
to \(+1\), and leave every remaining sign unchanged.
Denote the corresponding amplitude by \(R^{*,m}\).
It is measurable with respect to
\(\sigma(\varepsilon_{p_r}:r>m)\). By Step~2,
\[
\{R^*<\infty\}=\{R^{*,m}<\infty\},\qquad
\{R^*=0\}=\{R^{*,m}=0\}.
\]
Both events therefore belong to
\(\sigma(\varepsilon_{p_r}:r>m)\) for every \(m\), and hence
to the tail sigma-field of the prime signs.
Kolmogorov's zero--one law
\cite[Theorem~2.5.3]{Durrett} gives
\[
\PP(R^*<\infty)\in\{0,1\},\qquad
\PP(R^*=0)\in\{0,1\}.
\]
Theorem~\ref{dn:critical} implies \(\PP(R^*>1)>0\), so
\(\PP(R^*=0)=0\). This proves the stated alternative.
If \(R^*\) is finite almost surely, its support is unbounded
by Theorem~\ref{dn:critical}, so it is nonconstant.
\end{proof}


\appendix

\section{Upper estimates: proof of Proposition \ref{prop:interpolation}}\label{sec:interpolation}
Throughout this section, \(Y=\exp(2^r)\), where the integer
\(r\) tends to infinity, and \(0<\alpha<1/4\) is fixed. We write
\begin{equation}\label{eq:upperparameters}
h=(\log Y)^{-\alpha},\quad
y_0=\exp(\log Y/(\log_2Y)^2),\quad N=\lfloor Y^2\rfloor,\quad\delta=(\log Y)^{-250}.
\end{equation}
Here \(h\) is the logarithmic spacing of the coarse grid, \(y_0\) is the
cutoff separating small and large primes, \(N\) truncates the sums used
below, and \(\delta\) is the relative width of the intervals used to
smooth the prime index.

We estimate the distance \(d(x,y)\) defined in
\eqref{eq:distance}. We first introduce auxiliary notation.
For \(u\ge0\) and \(b\ge1\) define
\begin{equation}\label{eq:psin}
\Psi_N(u,b)=\sum_{\substack{n\le\min(u,N)\\P(n)\le b}}f(n).
\end{equation}
For a fixed cutoff \(b\), the function \(\Psi_N(u,b)\) sums the terms
whose arguments are at most \(u\) and whose prime factors are at most
\(b\). In particular, for \(z\in[Y,Y^2]\),
\begin{equation}\label{eq:apsin}
a_p(z)=\Psi_N(z/p,p^-).
\end{equation}
Put
\begin{align}
A_N(s)&=\int_0^\infty
\sup_{b\ge1}|\Psi_N(e^su,b)-\Psi_N(u,b)|^2\frac{\dd u}{u^2},
\label{eq:an}\\
R_Y(z)&=\sum_{p>y_0}\frac1{\delta p}\int_p^{(1+\delta)p}
|\Psi_N(z/p,p^-)-\Psi_N(z/t,p^-)|^2\dd t.
\label{eq:rt}
\end{align}
The supremum over \(b\) in \(A_N\) provides a single integral bound for all the prime-dependent cutoffs. We begin with the following estimate.

\begin{lemma}\label{lem:smoothingupper}
For all \(Y\le x\le y\le Y^2\) with \(\theta=\log(y/x)\le h\),
we have
\begin{equation}\label{eq:smoothingupper}
\frac{d(x,y)^2}{x}\le C\biggl(
\sup_{Y\le z\le Y^2}\frac{\|a_{\le y_0}(z)\|_2^2}{z}
+\frac{(\log_2Y)^2A_N(\theta)}{(\log Y)}+\frac{R_Y(x)}x+\frac{R_Y(y)}y\biggr),
\end{equation}
where \(a_{\le y_0}(x)=(a_p(x)\1_{\{p\le y_0\}})_p\).
\end{lemma}
\begin{proof}
The first term on the right-hand side accounts for the contribution to \(d(x,y)\) from small primes.
It therefore suffices to consider primes \(p>y_0\).
For \(t\in[p,(1+\delta)p]\), define
\[
g_p(t)=\Psi_N(y/t,p^-)-\Psi_N(x/t,p^-).
\]
Thus for \(t=p\), by \eqref{eq:apsin},
\(a_p(y)-a_p(x)=g_p(p)\). For general \(t\) in this interval we have
the exact decomposition
\begin{align*}
g_p(p)=g_p(t)
&+\bigl(\Psi_N(y/p,p^-)-\Psi_N(y/t,p^-)\bigr)
-\bigl(\Psi_N(x/p,p^-)-\Psi_N(x/t,p^-)\bigr).
\end{align*}
Applying \((A+B+C)^2\le2A^2+4B^2+4C^2\), averaging over
\(t\in[p,(1+\delta)p]\) with measure \(\dd t/(\delta p)\),
and summing over \(p>y_0\)
\begin{equation}\label{eq:largepdiff}
\sum_{p>y_0}|a_p(y)-a_p(x)|^2\le2I+4R_Y(y)+4R_Y(x),
\end{equation}
where \[
I=\sum_{p>y_0}\frac1{\delta p}\int_p^{(1+\delta)p}|g_p(t)|^2\dd t.
\]
It remains to bound \(I\).

Put
\[
H_Y(t)=\sum_{\substack{p>y_0\\p\le t\le(1+\delta)p}}
\frac1{\delta p}.
\]
Lemma~\ref{lem:prime-averaging}, with lower cutoff \(y_0\), gives
\begin{equation}\label{eq:ht}
H_Y(t)\ll\frac1{\log t}\le\frac1{\log y_0}
=\frac{(\log_2Y)^2}{(\log Y)}.
\end{equation}
Indeed, \(\delta^{-1}e^{-c\sqrt{\log y_0}}\to 0\), by
\eqref{eq:upperparameters}.
By Tonelli's theorem and \eqref{eq:ht},
\begin{align*}
I&=\int_{y_0}^\infty
\sum_{\substack{p>y_0\\p\le t\le(1+\delta)p}}
\frac{|g_p(t)|^2}{\delta p}\dd t\\
&\le\frac{C(\log_2Y)^2}{(\log Y)}\int_{y_0}^\infty
\sup_{b\ge1}|\Psi_N(y/t,b)-\Psi_N(x/t,b)|^2\dd t.
\end{align*}
Setting \(\theta=\log(y/x)\) and substituting \(u=x/t\), we have
\[
I\le\frac{C(\log_2Y)^2}{(\log Y)}\,xA_N(\theta).
\]
\end{proof}

We next bound the terms on the right-hand side of \eqref{eq:smoothingupper}.
\begin{lemma}\label{lem:smallprimes}
Almost surely, for all sufficiently large \(Y\) with dyadic \(\log Y\),

\begin{equation}\label{eq:smallprimes}
\sup_{Y\le x\le Y^2}\frac{\|a_{\le y_0}(x)\|_2^2}{x}\le (\log Y)^{-20}.
\end{equation}
\end{lemma}
\begin{proof}
{\sc Step 1.}   By Lemma~\ref{lem:rankin}, with {\(A=1\)},
for {\(Y\le X\le2Y^2\)},

\begin{equation}\label{eq:rankinupper}
\sfcount(X,y_0)
\le CX(\log y_0)\exp\bigg(-\frac{\log X}{\log y_0}\bigg)
\le C_1Xe^{-(\log_2Y)^2/2}
\end{equation}
for  all sufficiently large   {\(Y\)}, since
{\(\log X/\log y_0\ge (\log_2Y)^2\)} and {\(\log_2y_0=O((\log_2Y))\)}.

\medskip
\noindent
{\sc Step 2.} Fix \(X\) with   {\(Y\le X\le2Y^2\)}.
Enumerate the squarefree integers \(n\) satisfying
\(1<n\le X\) and   {\(P(n)\le y_0\)} :
\(n_1<n_2<\cdots<n_K\), and for \(1\le j\le K\) put
\[
V_j=f(n_j/P(n_j))e_{P(n_j)}.
\]
We apply the Hilbert-space-valued Menshov--Rademacher maximal inequality, see, e.g.,
\cite[Lemma~1, (11)]{MM}. For distinct \(n_i,n_j\) with the same
largest prime, the scalar products have mean zero by \eqref{eq:orth};
vectors with different largest primes are orthogonal in \(\ell^2\).
Each \(V_j\) has norm one. Since the vectors \(V_1,\ldots,V_K\)
are orthonormal in \(L^2(\Omega;\ell^2)\), that inequality gives

\begin{equation}\label{eq:menshov}
\E\max_{1\le k\le K}\biggl\|\sum_{j=1}^kV_j\biggr\|_2^2
\le C(\log(2K))^2K.
\end{equation}

By the coefficient representation \eqref{eq:vector}, for every
\(1\le x\le X\),   {\(a_{\le y_0}(x)=\sum_{j:n_j\le x}V_j\)}.
Since {\(K=\sfcount(X,y_0)-1\le X\le2Y^2\)},
a combination of \eqref{eq:rankinupper} and \eqref{eq:menshov} gives
\begin{equation}\label{eq:smallmaxmean}
\begin{aligned}
\E\sup_{1\le x\le X}\|a_{\le y_0}(x)\|_2^2
&=\E\max_{1\le k\le K}\biggl\|\sum_{j=1}^kV_j\biggr\|_2^2\\
&\le C(\log Y)^2\sfcount(X,y_0)\le C(\log Y)^2Xe^{-(\log_2Y)^2/2}.
\end{aligned}
\end{equation}
The assertion is immediate when \(K=0\).

\medskip
\noindent
{\sc Step 3.} Define
{\(B_i=Y2^i\)}  for \(0\le i<J\), where
{\(J=\lceil (\log Y)/\log2\rceil=O((\log Y))\)}.
The intervals \([B_i,2B_i]\) cover   {\([Y,Y^2]\)}, and each
\(X_i:=2B_i\) satisfies   {\(Y\le X_i\le2Y^2\)}.
Thus, Markov's inequality together with \eqref{eq:smallmaxmean} yields

\[
\PP\biggl(\sup_{x\in[B_i,2B_i]}
\frac{\|a_{\le y_0}(x)\|_2^2}{x}>(\log Y)^{-20}\biggr)
\le C(\log Y)^{22}e^{-(\log_2Y)^2/2}.
\]

Taking the union over the   {\(J=O((\log Y))\)}  intervals, we find

\[
\PP\biggl(\sup_{Y\le x\le Y^2}
\frac{\|a_{\le y_0}(x)\|_2^2}{x}>(\log Y)^{-20}\biggr)
\le C(\log Y)^{23}e^{-(\log_2Y)^2/2}.
\]

For {\(Y=\exp(2^r)\)}, the right-hand side is summable over \(r\), hence
the Borel--Cantelli lemma proves \eqref{eq:smallprimes}.
\end{proof}

\begin{lemma}\label{lem:anbound}
There is a constant \(C\) such that, almost surely for all sufficiently
large   {\(Y\)} with dyadic {\(\log Y\)}  and uniformly for
{$0\le\theta\le (\log Y)^{-\alpha}$},

\begin{equation}\label{eq:anbound}
A_N(\theta)\le C(\log Y)^{1-\alpha/2}.
\end{equation}

\end{lemma}
\begin{proof}
{\sc Step 1.} Fix \(u>0\) and \(s\ge0\), and put
\[
H_{u,s}:=\sum_{\substack{u<n\le e^su\\n\le N}}f(n).
\]
For \(b\ge1\), put
\[
D_s(u,b)=\Psi_N(e^su,b)-\Psi_N(u,b).
\]
List the primes not exceeding \(N\) in increasing order:
\(p_1<\cdots<p_K\).  Then
\(D_s(u,p_j)=\E(H_{u,s}\mid\FF_{p_j})\) is the martingale obtained by
revealing the prime signs in \(H_{u,s}\).
Applying Doob's \(L^2\) maximal inequality
(see e.g.\ \cite[Theorem~4.4.4]{Durrett}) and \eqref{eq:orth}
we obtain
\[
\E\sup_{b\ge1}|\Psi_N(e^su,b)-\Psi_N(u,b)|^2
=\E\sup_{b\ge1}|D_s(u,b)|^2
\le4\E H_{u,s}^2
\le4\sum_{\substack{u<n\le e^su\\n\le N}}\mu^2(n).
\]

\medskip
\noindent
{\sc Step 2.} By Step~1 and Tonelli's theorem, for \(0\le s\le1\),

\begin{align*}
\E A_N(s)
&=\int_0^\infty
\E\sup_{b\ge1}|\Psi_N(e^su,b)-\Psi_N(u,b)|^2\frac{\dd u}{u^2}\\
&\le4\sum_{n\le N}\mu^2(n)
\int_0^\infty\1_{\{u<n\le e^su\}}\frac{\dd u}{u^2}
\le4(e^s-1)\sum_{n\le N}\frac{\mu^2(n)}n
\le C(\log Y)s.
\end{align*}

\medskip
\noindent
{\sc Step 3.} We next pass from fixed \(s\) to all \(0\le\theta\le h\).
Fix such a \(\theta\) and take any \(t\in[h,2h]\).
The identity
\[
D_\theta(u,b)=D_t(u,b)-D_{t-\theta}(e^\theta u,b)
\]
implies, after squaring, taking the supremum over \(b\), and substituting
\(v=e^\theta u\) in the second integral,
\[
A_N(\theta)\le2A_N(t)+2e^\theta A_N(t-\theta).
\]
The left-hand side does not depend on \(t\), so we may integrate this
inequality over \(t\in[h,2h]\) and divide by \(h\):
\begin{align*}
A_N(\theta)
\le\frac2h\int_h^{2h}A_N(t)\dd t
+\frac{2e^\theta}{h}\int_{h-\theta}^{2h-\theta}A_N(s)\dd s
\le\frac Ch\int_0^{2h}A_N(s)\dd s.
\end{align*}
Using Step~2 and Tonelli once more gives

\[
\E\sup_{0\le\theta\le h}A_N(\theta)
\le\frac Ch\int_0^{2h}(\log Y)s\dd s\le C(\log Y)h=C(\log Y)^{1-\alpha}.
\]

Consequently, Markov's inequality yields

\[
\PP\biggl(\sup_{0\le\theta\le h}A_N(\theta)>(\log Y)^{1-\alpha/2}\biggr)
\le C(\log Y)^{-\alpha/2}.
\]

This bound is summable over   {\(Y=\exp(2^r)\)}, so the Borel--Cantelli lemma completes the proof.
\end{proof}

\begin{lemma}\label{lem:rtbound}
For every deterministic \(z\),   {\(\E R_Y(z)\le\delta z+1\)}.
If {\(G_Y\subset[Y,Y^2]\)}  is deterministic with
{\(|G_Y|\le C(\log Y)^{122}\)}, then almost surely for all sufficiently
large   {\(Y=\exp(2^r)\)},

\begin{equation}\label{eq:rtbound}
\max_{z\in G_Y}\frac{R_Y(z)}z\le (\log Y)^{-\alpha/2}.
\end{equation}

\end{lemma}
\begin{proof}
By orthogonality \eqref{eq:orth}, the expected value of the square in
\eqref{eq:rt} is the number of squarefree \(m\) with
\(z/t<m\le z/p\) and \(P(m)<p\).
Enlarge the range of \(m\) and count the pairs \((p,m)\) satisfying

\[
\frac z{1+\delta}<pm\le z,\qquad p>y_0,\quad P(m)<p.
\]

For a given pair \((p,m)\), the averaging weight of the values of \(t\)
for which \(z/t<m\le z/p\) is at most one.
Each squarefree integer \(pm\) is counted at most once, since
\(p=P(pm)\). There are at most \(\delta z+1\) integers in this
interval, so   {\(\E R_Y(z)\le\delta z+1\)}.
Markov's inequality and a union bound give

\[
\PP\biggl(\max_{z\in G_Y}\frac{R_Y(z)}z>(\log Y)^{-\alpha/2}\biggr)
\le C(\log Y)^{122+\alpha/2}((\log Y)^{-250}+e^{-(\log Y)}).
\]

The last bound is summable over   {\(Y=\exp(2^r)\)}, hence the Borel--Cantelli
lemma completes the proof.
\end{proof}

The next estimate follows from the hypercontractive and
maximal inequalities in \cite[Lemma~3.5, Step~2, and Lemma~3.4]{BIRestricted}.
We record the application to the full coefficient vector.
\begin{lemma}\label{lem:fourth}
For sufficiently large \(u\) and \(v\) such that \(u+1\le v\le3u/2\),
\begin{equation}\label{eq:fourth}
\E\sup_{u\le x\le v}\bigl(|M(x)-M(u)|+d(u,x)\bigr)^4
\le Cu^{2/3}(v-u)^{4/3}(\log u)^{24}.
\end{equation}
\end{lemma}
\begin{proof}
For squarefree \(n>1\), put
\(V_n=f(n/P(n))e_{P(n)}\). The scalar hypercontractive bound
\cite[(3.17)]{BIRestricted}, see also
\cite[Section~9.2]{ODonnell}, gives
\[
\bigg(\E\bigg|\sum_n c_nf(n)\bigg|^4\bigg)^{1/2}
\le\sum_n3^{\omega(n)}c_n^2.
\]
For the vector sum, apply this inequality to each coordinate
and then use Minkowski's inequality, as in
\cite[(3.18)--(3.19)]{BIRestricted}:
\[
\bigg(\E\bigg\|\sum_n c_nV_n\bigg\|_2^4\bigg)^{1/2}
\le\sum_p\bigg(\E\bigg|
\sum_{n:P(n)=p}c_nf(n/p)\bigg|^4\bigg)^{1/2}
\le\sum_n3^{\omega(n)}c_n^2.
\]
This calculation does not require the restriction
\(n<P(n)^2\) used for the restricted sum in that paper.
List the squarefree integers in \((u,v]\) in increasing order,
and put \(W=\sum_{u<n\le v}\mu^2(n)3^{\omega(n)}\).
The maximal inequality \cite[Lemma~3.4, (3.12)]{BIRestricted}
applies to both families \((f(n))\) and \((V_n)\). Therefore,
\[
\E\sup_{u\le x\le v}
\bigl(|M(x)-M(u)|+d(u,x)\bigr)^4
\le C W^2.
\]

By \cite[(3.21)]{BIRestricted}, applied with \(k=27\),
\[
 \sum_{n\le v}\mu^2(n)27^{\omega(n)}
 \le C v(1+\log v)^{26}.
\]
Hence H\"older's inequality gives
\[
 W \le (2(v-u))^{2/3}
 \bigg(\sum_{n\le v}\mu^2(n)27^{\omega(n)}\bigg)^{1/3}\le C u^{1/3}(v-u)^{2/3}(\log u)^{26/3}.
\]
Combining this with the preceding maximal bound and noting that \(52/3<24\) proves the assertion.
\end{proof}

\begin{proof}[Proof of Proposition~\ref{prop:interpolation}]
Introduce the following finer deterministic grid on the block
{\([Y,Y^2]\)}:

\begin{equation}\label{eq:fine}
\cD_Y=
\{Ye^{j(\log Y)^{-120}}:0\le j\le\lfloor(\log Y)^{121}\rfloor\}
\cup\{Y^2\}.
\end{equation}

\medskip
\noindent
{\sc Step 1.} We first compare the values of \(M\) and \(Q\) at arbitrary
points with values at the finer grid.
For \([u,v]\) equal to a cell between consecutive points of   {\(\cD_Y\)},
apply Lemma~\ref{lem:fourth} on \([u,\max\{v,u+1\}]\)
and use Markov's inequality to obtain

\[
\PP\biggl(\sup_{u\le x\le v}(|M(x)-M(u)|+d(u,x))
>\sqrt u\,(\log Y)^{-2}\biggr)\le C(\log Y)^{-128}.
\]

There are \(O((\log Y)^{121})\) cells, so the probability that the bound fails in any cell is   {\(O((\log Y)^{-7})\)}, which is summable over the integers \(r\).
By Borel--Cantelli, almost surely for all sufficiently large   {\(Y\)} with dyadic {\(\log Y\)},

\begin{equation}\label{eq:fineerror}
\sup_{u\le x\le v}(|M(x)-M(u)|+d(u,x))\le\sqrt u\,(\log Y)^{-2}
\end{equation}
simultaneously in all cells.

\medskip
\noindent
{\sc Step 2.} We now pass from the grid   {\(\cD_Y\)} to {\(\cC_Y\)}.
For {\(y\in\cD_Y\)}, let {\(c=c_Y(y)\)}  be the largest point of
{\(\cC_Y\)}  that does not exceed \(y\). Then \(\log(y/c)\le h\).
Apply Lemma~\ref{lem:smoothingupper} together with Lemmas~\ref{lem:smallprimes}, \ref{lem:anbound} and \ref{lem:rtbound} to the deterministic set   {\(G_Y=\cC_Y\cup\cD_Y\)},
of size   {\(O((\log Y)^{121})\)}, almost surely, for all sufficiently large   {\(Y\)},
simultaneously for all such pairs \((c,y)\), we have

\begin{equation}\label{eq:gridmetric}
\frac{d(c,y)^2}{c}\le C(\log Y)^{-20}+C(\log_2Y)^2(\log Y)^{-\alpha/2}
+C(\log Y)^{-\alpha/2}\le (\log Y)^{-\alpha/3}.
\end{equation}

Put
\[
M(y)-M(c)=\sum_p\varepsilon_p(a_p(y)-a_p(c)),\qquad
d(c,y)^2=\sum_p(a_p(y)-a_p(c))^2.
\]
 With \(b_p=a_p(y)-a_p(c)\),
Lemma~\ref{lem:exponential} applies with
{\(t=\sqrt c\,(\log Y)^{-\alpha/12}\)}, {\(v=c(\log Y)^{-\alpha/3}\)}.
Thus, defining
\[
F_{(\log Y),y}=\biggl\{|M(y)-M(c_Y(y))|>\sqrt{c_Y(y)}(\log Y)^{-\alpha/12},\
d(c_Y(y),y)^2\le c_Y(y)(\log Y)^{-\alpha/3}\biggr\},
\]
Lemma~\ref{lem:exponential} gives
\[\PP(F_{(\log Y),y})\le2e^{-(\log Y)^{\alpha/6}/2}.\]
For each   {\((Y,y)\)}, the pair   {\((c_Y(y),y)\)}  is deterministic;
the random bound on   {\(d(c_Y(y),y)\)}  is used only after the following
union bound and Borel--Cantelli step.
Taking a union bound over the deterministic fine grid,

\[
\PP\biggl(\bigcup_{y\in\cD_Y}F_{(\log Y),y}\biggr)
\le C(\log Y)^{121}e^{-(\log Y)^{\alpha/6}/2}.
\]

The last bound is summable over   {\(Y\)} with dyadic {\(\log Y\)}.
By the Borel--Cantelli lemma and \eqref{eq:gridmetric},
almost surely, for all sufficiently large \(r\), simultaneously
for every \(y\in\cD_Y\),

\begin{equation}\label{eq:coarsefine}
|M(y)-M(c_Y(y))|\le\sqrt{c_Y(y)}(\log Y)^{-\alpha/12},\qquad
d(c_Y(y),y)\le\sqrt{c_Y(y)}(\log Y)^{-\alpha/6}.
\end{equation}

\noindent
{\sc Step 3.} Finally, let   {\(x\in[Y,Y^2]\)}  be arbitrary.
Let \(y\) be the largest point of   {\(\cD_Y\)}  with \(y\le x\), and put
{\(z=c_Y(y)\)}. Consecutive points of \(\cD_Y\) differ by at most
\((\log Y)^{-120}\) on the logarithmic scale. Hence   {\(\log(x/y)\le (\log Y)^{-120}\)}.
By the definition of   {\(c_Y(y)\)}  and the spacing of   {\(\cC_Y\)},
{\(\log(y/z)\le h=(\log Y)^{-\alpha}\)}. Hence {\(\log(x/z)\le2(\log Y)^{-\alpha}\)}.
Combining \eqref{eq:fineerror} with \eqref{eq:coarsefine}
we obtain, almost surely, for large   {\(Y\)},

\[
|M(x)-M(z)|+d(z,x)\le3\sqrt x\,(\log Y)^{-\alpha/12}.
\]

This completes the proof.
\end{proof}

\section{Auxiliary estimates for the lower bounds}\label{app:lower}

\subsection{Proposition \ref{prop:flatness}: variation of the layer variances}

We recall the notation used in the following proposition.
The endpoint grid is
\[
\mathcal X_Y=
\bigg\{
Ye^r:
0\le r\le
\bigg\lfloor
\frac{\log Y}{(\log_2Y)^{100}}
\bigg\rfloor
\bigg\}.
\]
Put \(\delta=(\log_2Y)^{-1/20}\) and
\[
y_j=\exp\bigg(
\frac{e^{j\delta}\log Y}{(\log_2Y)^2}
\bigg),\qquad j\ge0.
\]
Let \(K\) be the least positive integer such that
\(y_K\ge Y^{1+(\log_2Y)^{-100}}\), and write
\(I_j=(y_{j-1},y_j]\) for the \(j\)th prime layer.
In particular, \(y_K\ge x\) for every \(x\in\mathcal X_Y\).

For \(p\in I_j\), the retained and discarded coefficients are
\[
b_p(x)=\Psi(x/p,y_{j-1}),
\qquad
d_p(x)=\Psi(x/p,p^-)-\Psi(x/p,y_{j-1}).
\]
The quantities appearing below are
\[
V_j(x)=\sum_{p\in I_j}b_p(x)^2,
\qquad
D(x)=\sum_{j=1}^K\sum_{p\in I_j}d_p(x)^2.
\]
Thus \(V_j(x)\) is the conditional variance of the retained
layer sum \(R_j(x)\), given \(\FF_{y_{j-1}}\), whereas
\(D(x)\) is the predictable quadratic variation of the
discarded sum. The following proposition controls their
variation after division by \(x\), uniformly over the
endpoint grid.

\begin{proposition}
\label{prop:flatness}
With probability \(1-o(1)\) as \(Y\to\infty\), the following
bounds hold simultaneously:

\begin{align}
\sum_{j=1}^K\max_{x\in\mathcal X_Y}
\biggl|
\frac{V_j(x)}x-\frac{V_j(Y)}Y
\biggr|
&\le(\log_2Y)^{-10},
\label{eq:flatV}\\
\max_{x\in\mathcal X_Y}
\biggl|
\frac{D(x)}x-\frac{D(Y)}Y
\biggr|
&\le(\log_2Y)^{-10},
\label{eq:flatD}\\
\max_{x\in\mathcal X_Y}\frac{D(x)}x
&\le2(\log_2Y)^{-53/100}.
\label{eq:uniformcollision}
\end{align}

\end{proposition}

Our first task is to control \(D(Y)\), the discarded energy
at the left endpoint of the block.

\begin{proposition}\label{prop:collision} We have

\begin{equation}\label{eq:collision}
\PP\big(
D(Y)>Y(\log_2Y)^{-53/100}
\big)\longrightarrow0
\qquad \text{as\ } Y\to\infty.
\end{equation}
\end{proposition}

\begin{lemma}\label{lem:conditionalsmooth}
Suppose \(y\) is sufficiently large, \(w\ge y^{12}\), and
\(s=\log w/\log y\ge\log_2y\).
For every \(z\ge y\), there is a nonnegative
\(\FF_y\)-measurable variable \(E_{w,z}\) such that
\begin{equation}\label{eq:conditionalsmooth}
\frac1w\E[\Psi(w,z)^2\mid\FF_y]
\le C\frac{\mathcal E_y}{\log y}+E_{w,z},
\qquad \E E_{w,z}\le Ce^{-s},
\end{equation}
where
\[
\mathcal E_y
=\int_1^\infty\frac{\Psi(u,y)^2}{u^2}\dd u.
\]
\end{lemma}
\begin{proof}
Every squarefree integer \(n\) has a unique factorization
\(n=ab\), where \(P^-(a)>y\) and \(P(b)\le y\).
Consequently, for \(1\le y\le z\),
\[
\Psi(w,z)
=\sum_{\substack{a\le w,\ a\ {\rm squarefree}\\
P^-(a)>y,\ P(a)\le z}}
f(a)\Psi(w/a,y).
\]
Each coefficient \(\Psi(w/a,y)\) is \(\FF_y\)-measurable.
Since
\(
\E[f(a)f(a')\mid\FF_y]=\1_{\{a=a'\}},
\)
expanding the square gives
\begin{equation}\label{eq:roughfactor}
\E[\Psi(w,z)^2\mid\FF_y]
=\sum_{\substack{a\le w,\ a\ {\rm squarefree}\\
P^-(a)>y,\ P(a)\le z}}
\Psi(w/a,y)^2.
\end{equation}

Define
\[
E_{w,z}
=\frac1w
\sum_{\substack{a\le w^{3/4},\ a\ {\rm squarefree}\\
P^-(a)>y,\ P(a)\le z}}
\Psi(w/a,y)^2.
\]

By orthogonality \eqref{eq:orth},
and  Lemma~\ref{lem:rankin}, with {\(A=16\)},

\[\E\Psi(v,y)^2=\sfcount(v,y)
\le Cv(\log y)\exp\bigg(-\frac{16\log v}{\log y}\bigg).
\]

Put \(s=\log w/\log y\).
For \(a\le w^{3/4}\), we have $(w/a)^{-16/\log y}\le e^{-4s}$ and recall $\log_2 y\le s$. Hence
\[
\E E_{w,z}
\le\frac1w\sum_{a\le w^{3/4}}\sfcount(w/a,y)
\le C(\log y)e^{-4s}
\sum_{a\le w^{3/4}}\frac1a
\le C(\log y)(1+\log w)e^{-4s}
\le Ce^{-s}.
\]
Here, we have used \(\log w=s\log y\) and \(\log y\le e^s\).

It remains to bound the contribution to \eqref{eq:roughfactor}
from \(a>w^{3/4}\), that is to prove
\begin{equation}\label{eq:ff}
F_{w,z} := \frac1w
\sum_{\substack{w^{3/4}<a\le w,\ a\ {\rm squarefree}\\
P^-(a)>y,\ P(a)\le z}}
\Psi(w/a,y)^2 \le \frac C{\log y}\mathcal E_y.
\end{equation}
Put \(r=\lfloor w/a\rfloor\) and $J_r:=({w}/{(r+1)},{w}/{r}]$.
Then, dropping the squarefree
and upper-prime restrictions, we obtain
\[
F_{w,z} \le
\frac1w\sum_{1\le r<w^{1/4}}\Psi(r,y)^2
\#\{a\in J_r:P^-(a)>y\}.
\]
Since for every \(1\le r<w^{1/4}\), the  length of \(J_r\) satisfies
\(
|J_r|=\frac{w}{r(r+1)}\ge\frac12w^{1/2} \ge y
\), therefore by
\cite[(3.20), p.~86]{MV} and Lemma \ref{lem:prime-estimates},
\[
\#\{a\in J_r:P^-(a)>y\}
=\#\bigg\{a\in J_r:\bigg(a,\prod_{p\le y}p\bigg)=1\bigg\}
\le C|J_r|\prod_{p\le y}\biggl(1-\frac{1}{p}\biggr)
\le\frac{C|J_r|}{\log y}.
\]
Consequently,
\[
F_{w,z}\le\frac C{\log y}
\sum_{r\ge1}\frac{\Psi(r,y)^2}{r(r+1)}
=\frac C{\log y}\mathcal E_y.
\]
The equality follows by integrating the step function
\(\Psi(u,y)^2\) over each interval \([r,r+1)\).
This proves \eqref{eq:ff} and completes the proof.
\end{proof}

\begin{lemma}\label{lem:energyupper}
For \(y\ge1\), put
\begin{equation}\label{eq:euler}
F_y(t)=\prod_{p\le y}(1+\varepsilon_pp^{-1/2-it}).
\end{equation}
Then
\begin{equation}\label{eq:parsevalenergy}
\mathcal E_y=\int_1^\infty\frac{\Psi(u,y)^2}{u^2}\dd u
=\frac1{2\pi}\int_{\RR}
\frac{|F_y(t)|^2}{1/4+t^2}\dd t,
\end{equation}

Moreover, for every fixed \(a>0\),
\begin{equation}\label{eq:energyupper}
\PP\biggl(
\frac{\mathcal E_y}{\log y}>
(\log_2y)^{-1/2+a}
\biggr)\longrightarrow0
\qquad\text{as } y\to\infty.
\end{equation}
\end{lemma}

\begin{proof}
We first prove \eqref{eq:parsevalenergy}, writing out the
application of the Mellin--Plancherel formula
\cite[(5.26)]{MV}. Define
\[
 g_y(s)=
 \begin{cases}
 e^{-s/2}\Psi(e^s,y),&s\ge0,\\
 0,&s<0.
 \end{cases}
\]
The substitution \(u=e^s\) gives
\[
 \int_{\RR}|g_y(s)|^2\dd s
 =\int_0^\infty e^{-s}\Psi(e^s,y)^2\dd s
 =\int_1^\infty\frac{\Psi(u,y)^2}{u^2}\dd u
 =\mathcal E_y.
\]

There are only finitely many squarefree integers whose prime
factors are at most \(y\). Thus \(\Psi(u,y)\) is bounded as
a function of \(u\), and \(g_y\in L^1(\RR)\cap L^2(\RR)\). Next, we obtain
\[
\begin{split}
 \widehat g_y(t)
 &=\int_0^\infty
   \Psi(e^s,y)e^{-(1/2+it)s}\dd s
 =\sum_{\substack{n\ {\rm squarefree}\\P(n)\le y}}
   f(n)\int_{\log n}^\infty
   e^{-(1/2+it)s}\dd s\\
 &=\frac1{1/2+it}
   \sum_{\substack{n\ {\rm squarefree}\\P(n)\le y}}
   \frac{f(n)}{n^{1/2+it}}
 =\frac{F_y(t)}{1/2+it}.
\end{split}
\]
The interchange of the sum and the integral is justified
because the sum is finite. Plancherel's identity now yields
\[
 \mathcal E_y
 =\int_{\RR}|g_y(s)|^2\dd s
 =\frac1{2\pi}\int_{\RR}|\widehat g_y(t)|^2\dd t
 =\frac1{2\pi}\int_{\RR}
   \frac{|F_y(t)|^2}{1/4+t^2}\dd t.
\]
This proves \eqref{eq:parsevalenergy}.

We now prove \eqref{eq:energyupper}. Orthogonality \eqref{eq:orth} and Lemma \ref{lem:prime-estimates} yield
\[
\E|F_y(t)|^2
=\sum_{\substack{n\ge1\\P(n)\le y}}\frac{\mu^2(n)}n
=\prod_{p\le y}(1+p^{-1})
\asymp\log y,
\]
uniformly in \(t\in\RR\).
Split the integral in \eqref{eq:parsevalenergy} as

\[
\mathcal E_y=I_0+I_1+I_\infty =
\frac1{2\pi}\bigg(\int_{|t|<(\log_2y)^{-2}}
+\int_{(\log_2y)^{-2}\le|t|\le(\log_2y)^2}
+ \int_{|t|>(\log_2y)^2}\bigg)
\frac{|F_y(t)|^2}{1/4+t^2}\dd t.
\]

By Tonelli's theorem,

\[
\E(I_0+I_\infty)
\ll(\log y)(\log_2y)^{-2}.
\]

For each \(t\in\RR\), let
\[
\mathcal D(t)=
\bigcap_{j=0}^{\lfloor\log_2y-1\rfloor}
\bigg\{
\big|F_{\exp(e^{-j}\log y)}(t)\big|
\le e^{-j}(\log y)(\log_2y)^2
\bigg\}.
\]
By \cite[Lemma~6.5]{BIRestricted}, applied with \(a_r\) replaced by \(y\), there is an event \(G_y\)
such that
\[
\PP(G_y^c)\le\frac{C}{\log_2y},
\]
and, on \(G_y\), the event \(\mathcal D(t)\) holds
simultaneously for every \(|t|\le(\log_2y)^2+2\).
Put
\[
J=\frac1{2\pi}
\int_{(\log_2y)^{-2}\le|t|\le(\log_2y)^2}
\frac{|F_y(t)|^2\1_{\mathcal D(t)}}{1/4+t^2}\dd t.
\]
Then \(I_1=J\) on \(G_y\).
The proof of \cite[Lemma~6.7, (6.36)]{BIRestricted}, with
cutoff \(y\), and Tonelli's theorem give

\[
\begin{aligned}
\E J
&\le C(\log y)(\log_3y)^2(\log_2y)^{-1/2}
\int_{\RR}\frac{\dd t}{1/4+t^2}
\ll(\log y)(\log_3y)^2(\log_2y)^{-1/2}.
\end{aligned}
\]

Fix \(a>0\) and put   {\(K=(\log y)(\log_2y)^{-1/2+a}\)}.
Thus the union bound and Markov's inequality yield

\[
\begin{split}
\PP(\mathcal E_y>K)
&\le \PP(G_y^c)
+\PP(I_0>K/3)+\PP(J>K/3)
+\PP(I_\infty>K/3)\\
&\le C(\log_2y)^{-1}
+\frac3K\bigl(\E I_0+\E J+\E I_\infty\bigr)\\
&\le C\big(
(\log_2y)^{-1}
+(\log_2y)^{-3/2-a}
+(\log_3y)^2(\log_2y)^{-a}
\big)
\longrightarrow0.
\end{split}
\]
This proves \eqref{eq:energyupper}.
\end{proof}

\begin{proof}[Proof of Proposition \ref{prop:collision}]

\medskip
\noindent
{\sc Step 1.} Put
\[
y=Y^{1/(\log_2Y)^4},
\qquad
U_0=Y^{1/(\log_2Y)^3}=y^{\log_2Y}.
\]
Then   {\(y<y_0\)}, and
\(
\log_2y\sim\log_2Y.
\)
Lemma~\ref{lem:energyupper}, with \(a=1/400\), therefore gives

\begin{equation}\label{eq:energygood}
\PP\biggl(
\frac{\mathcal E_y}{\log y}>
(\log_2Y)^{-1/2+1/200}
\biggr)=o(1).
\end{equation}
Here the difference between \(1/400\) and \(1/200\)
absorbs the replacement of \(\log_2y\) by \(\log_2Y\).

Fix \(1\le j\le K\) and a prime \(p\in I_j\).
By definition,
\begin{equation}\label{eq:collisiondecomp}
d_p(Y)
=
\sum_{\substack{q\in I_j\\q<p}}
\varepsilon_q\Psi(Y/(pq),q^-).
\end{equation}
We split this sum according to the size of the remaining
cofactor range. Put
\[
\begin{aligned}
d_p^{(1)}
&=
\sum_{\substack{q\in I_j,\ q<p\\pq\le Y/U_0}}
\varepsilon_q\Psi(Y/(pq),q^-),\\
d_p^{(2)}
&=
\sum_{\substack{q\in I_j,\ q<p\\Y/U_0<pq\le Y}}
\varepsilon_q\Psi(Y/(pq),q^-).
\end{aligned}
\]
All sums over \(q\) are over primes. Terms with \(pq>Y\)
vanish because \(\Psi(u,q^-)=0\) for \(u<1\).

Put

\[
D_i=\sum_{y_0<p\le y_K}(d_p^{(i)})^2,
\qquad i=1,2.
\]

Since \(d_p(Y)=d_p^{(1)}+d_p^{(2)}\),
\[
D(Y)\le2D_1+2D_2.
\]

\medskip
\noindent
{\sc Step 2.}
We first estimate \(D_1\).
Put \(w_{pq}=Y/(pq)\). Since
\(
w_{pq}\ge U_0=y^{\log_2Y}\) and
\(s_{pq}:=\frac{\log w_{pq}}{\log y}\ge\log_2Y
\),
for sufficiently large \(Y\), we can apply
Lemma~\ref{lem:conditionalsmooth} to
\(\E[\Psi(w_{pq},q^-)^2\mid\FF_y]\).
Conditional orthogonality of the martingale differences therefore gives
\[
\begin{split}
\frac1Y\E[D_1\mid\FF_y]
&=
\sum_{\substack{1\le j\le K,\ p,q\in I_j,\ q<p\\
pq\le Y/U_0}}
\frac1{pq}\frac1{w_{pq}}
\E[\Psi(w_{pq},q^-)^2\mid\FF_y].
\end{split}
\]
Each layer has reciprocal-prime mass \(O(\delta)\), by
\eqref{eq:primesosc} applied with $v=0$, $z=y_{j-1}$ and $U=y_j$. Consequently,
\begin{equation}\label{eq:collisionmass}
r_Y:=
\sum_{j=1}^K
\biggl(\sum_{p\in I_j}\frac1p\biggr)^2
\ll K\delta^2
\ll(\log_2Y)^{-1/20}\log_3Y.
\end{equation}
Applying Lemma~\ref{lem:conditionalsmooth} term by term gives
\begin{equation}\label{eq:collisionconditional}
\frac1Y\E[D_1\mid\FF_y]
\le
Cr_Y\frac{\mathcal E_y}{\log y}
+\mathcal R_y,
\qquad
\E\mathcal R_y\ll\frac{r_Y}{\log Y},
\end{equation}
where \(\mathcal R_y\) is the nonnegative weighted sum of
the errors in \eqref{eq:conditionalsmooth}. The expectation
bound uses \(s\ge\log_2Y\), and hence
\(e^{-s}\le(\log Y)^{-1}\).

Let

\[
G_y=
\biggl\{
\frac{\mathcal E_y}{\log y}
\le(\log_2Y)^{-1/2+1/200}
\biggr\}.
\]
This event is \(\FF_y\)-measurable, and
\(\PP(G_y^c)=o(1)\) by \eqref{eq:energygood}. On \(G_y\),
the first term on the right-hand side of
\eqref{eq:collisionconditional} is at most
\(
C(\log_2Y)^{-109/200}\log_3Y.
\)
For any fixed \(c>0\), conditional Markov's inequality gives

\[
\begin{aligned}
\PP\Big(
D_1&>cY(\log_2Y)^{-53/100}
\Big)
\le
\PP(G_y^c)
+\frac{(\log_2Y)^{53/100}}{cY}
\E[\1_{G_y}D_1]\\
&=
\PP(G_y^c)
+\frac{(\log_2Y)^{53/100}}{c}
\E\biggl[
\1_{G_y}\frac{\E[D_1\mid\FF_y]}Y
\biggr]\\
&\le
o(1)
+O_c\biggl((\log_2Y)^{-3/200}\log_3Y\biggr)
+O_c\biggl(
\frac{r_Y(\log_2Y)^{53/100}}{\log Y}\biggr)=o(1).
\end{aligned}
\]

\medskip
\noindent
{\sc Step 3.}
We next estimate \(D_2\). For all contributing pairs
\[
\frac{Y}{U_0}<pq\le Y,\qquad p>q,
\qquad
\log p\le e^\delta\log q.
\]
Thus

\[
(1+e^\delta)\log q
\ge\log(pq)
>
\big(1-(\log_2Y)^{-3}\big)\log Y.
\]
Since \(\delta\to0\), this implies \(q\ge Y^{1/4}\)
for sufficiently large \(Y\). Also \(q\le\sqrt Y\),
and the remaining cofactor is at most \(U_0\).

For each fixed \(p\in I_j\), we have
\[
d_p^{(2)}
=
\sum_{\substack{q\in I_j,\ q<p\\Y/U_0<pq\le Y}}
\ \sum_{\substack{a\le Y/(pq)\\P(a)<q}}f(qa).
\]
Only squarefree \(a\) contribute. Since \(P(a)<q\), the
largest prime factor of \(qa\) is \(q\). Thus distinct pairs
\((q,a)\) give distinct integers \(qa\), and orthogonality
yields
\[
\E D_2
=
\sum_{j=1}^K
\sum_{\substack{p,q\in I_j,\ q<p\\Y/U_0<pq\le Y}}
\ \sum_{\substack{a\le Y/(pq)\\
a\ {\rm squarefree},\ P(a)<q}}1.
\]
Each triple represents \(n=pqa\), with \(p\) and \(q\) its
largest two prime factors, so its assignment is unique.
Every contributing triple satisfies
\(
a<U_0\),
\(Y^{1/4}\le q\le\sqrt Y\),
\(p\le\frac{Y}{aq}\).
For fixed \(a,q\), there are
\(\pi(Y/(aq))\) possible primes \(p\), where \(\pi(v)\)
denotes the number of primes at most \(v\). Consequently,

\[
\E D_2
\le
\sum_{a\le U_0}
\sum_{\substack{Y^{1/4}\le q\le\sqrt Y\\q\ {\rm prime}}}
\pi\biggl(\frac{Y}{aq}\biggr).
\]

For all \(a,q\) in this sum,
\(
{Y}/(aq)
\ge Y^{1/2-(\log_2Y)^{-3}}
\).
Hence the prime-counting bound yields
\(
\pi({Y}/({aq}))
\ll{Y}/({aq\log Y}).
\)
Lemma \ref{lem:prime-estimates} yields
\[
\E D_2
\ll
\frac{Y}{\log Y}
\sum_{a\le U_0}\frac1a
\sum_{\substack{Y^{1/4}\le q\le\sqrt Y\\q\ {\rm prime}}}
\frac1q
\ll
\frac{Y\log U_0}{\log Y}
=
Y(\log_2Y)^{-3}.
\]
Applying Markov's inequality once more gives

\[
\PP\big(
D_2>cY(\log_2Y)^{-53/100}
\big)
\le C(\log_2Y)^{-247/100}=o(1).
\]

Taking \(c=1/4\) in the bounds for \(D_1\) and \(D_2\),
and using \(D(Y)\le2D_1+2D_2\), proves the proposition.
\end{proof}
\begin{lemma}\label{lem:smoothmax}
For \(u\ge1\), \(y\ge1\), and \(1\le b_0\le b_1\),
\begin{align}
\E\Psi(u,y)^2&\le u, \label{eq:smoothsecond}\\
\E\sup_{b\ge1}|\Psi(u,b)|^2&\le4u,
\label{eq:smoothmaxprime}\\
\E\sup_{b_0\le b\le b_1}
|\Psi(u,b)-\Psi(u,b_0)|^2
&\le4u\sum_{b_0<p\le b_1}\frac1p.
\label{eq:incrementprime}
\end{align}
\end{lemma}

\begin{proof}
Expanding the square and applying \eqref{eq:orth}, we obtain
\[
\E\Psi(u,y)^2
=
\sum_{\substack{m,n\le u\\P(m),P(n)\le y}}
\E[f(m)f(n)]
=
\sum_{\substack{n\le u\\P(n)\le y}}\E f(n)^2
\le u.
\]
This proves \eqref{eq:smoothsecond}.

Since
$\Psi(u,b)=\E[M(u)\mid\FF_b]$, the sequence \(\{\Psi(u,b)\}_b\) is a martingale with
terminal value \(M(u)\). Therefore, by Doob's maximal inequality and \eqref{eq:orth}
\[
\E\sup_{b\ge1}|\Psi(u,b)|^2
\le4\E M(u)^2
=4\sum_{n\le u}\E f(n)^2
\le4u.
\]
This gives \eqref{eq:smoothmaxprime}.
Next, define \(D_b=\Psi(u,b)-\Psi(u,b_0)\), then $\{D_b\}_{b=b_0}^{b_1}$
is a martingale with  initial value zero and terminal value
\[
D_{b_1}
=\sum_{\substack{n\le u\\b_0<P(n)\le b_1}}f(n).
\]
Doob's maximal inequality together with  \eqref{eq:orth} entail
\[
\E\sup_{b_0\le b\le b_1}|D_b|^2
\le4\E|D_{b_1}|^2
=4\sum_{\substack{n\le u\\b_0<P(n)\le b_1}}
\E f(n)^2.
\]
Only squarefree integers contribute to the last sum.
Each such integer has a unique representation \(n=pm\),
where \(p=P(n)\) and \(P(m)<p\). Therefore,
\[
\E\sup_{b_0\le b\le b_1}|D_b|^2
\le 4\sum_{\substack{n\le u\\b_0<P(n)\le b_1}}\E f(n)^2
=
4 \sum_{b_0<p\le b_1}
\sum_{\substack{m\le u/p\\P(m)<p}}\E f(pm)^2
\le4 u\sum_{b_0<p\le b_1}\frac1p.
\]
This proves \eqref{eq:incrementprime}.
\end{proof}

\begin{proof}[Proof of Proposition~\ref{prop:flatness}]
We first replace the prime sums by integrals. We then estimate
the variation of these integrals over the whole real block and finally apply Markov's
inequality. Throughout the proof, put
\[
Y_+=Y^{1+(\log_2Y)^{-100}},\qquad
a_0=\log y_0=\frac{\log Y}{(\log_2Y)^2}.
\]
We use the bounds
\[
|\mathcal X_Y|\le\log Y,\qquad
K\ll(\log_2Y)^{1/20}\log_3Y,\qquad
\log y_K\le2\log Y,
\]
and the inclusion \(\mathcal X_Y\subset[Y,Y_+]\).
Recall, \(Y\in\mathcal X_Y\).

\medskip\noindent
{\sc Step 1. }
Put \(s_0=(\log Y)^{-40}\). For each prime \(p\), define a probability measure
\[
\nu_p(\dd t)
=\frac{\1_{[p,(1+s_0)p]}(t)}{s_0p}\dd t.
\]
Then
\[
 V_j(x)
 =\sum_{p\in I_j}\int b_p(x)^2\,\nu_p(\dd t),\qquad
 D(x)
 =\sum_{j=1}^K\sum_{p\in I_j}
   \int d_p(x)^2\,\nu_p(\dd t).
\]

If \(p\in I_j\), put
\[
\widetilde b_p(x,t)=\Psi(x/t,y_{j-1}),\qquad
\widetilde d_p(x,t)
=\Psi(x/t,t)-\Psi(x/t,y_{j-1}),
\]
and define the averaged clocks
\[
\overline V_j(x)
=\sum_{p\in I_j}\int\widetilde b_p(x,t)^2\nu_p(\dd t),
\qquad
\overline D(x)
=\sum_{j=1}^K\sum_{p\in I_j}
\int\widetilde d_p(x,t)^2\nu_p(\dd t).
\]
The goal of the first step is to prove
\begin{equation}
\label{eq:vvdd}
\E\bigg[
\sum_{j=1}^K\bigg|\frac{V_j(x)-\overline V_j(x)}x\bigg|
+\bigg|\frac{D(x)-\overline D(x)}x\bigg|
\bigg]
\le C \bigl(s_0+x^{-1}+y_0^{-1}\bigr)^{1/2}.
\end{equation}

Fix a deterministic \(x\in\mathcal X_Y\).
For \(p\le t\le(1+s_0)p\), orthogonality gives
\[
\E|b_p(x)-\widetilde b_p(x,t)|^2
=\sum_{\substack{x/t<n\le x/p\\P(n)\le y_{j-1}}}\mu^2(n).
\]
Every pair \((p,n)\) counted here satisfies
\[
\frac{x}{1+s_0}<pn\le x,
\qquad P(n)\le y_{j-1}<p.
\]
Thus \(p\) is the largest prime factor of \(pn\), and the
map \((p,n)\mapsto pn\) is injective over all layers.
Since each \(\nu_p\) has total mass one, summing and
integrating yields
\[
\E\sum_{j=1}^K\sum_{p\in I_j}
\int|b_p(x)-\widetilde b_p(x,t)|^2\nu_p(\dd t)
\le s_0x+1.
\]

For the discarded coefficients, first keep the prime cutoffs
fixed and change only the argument. The intermediate coefficient is
\[
d_p^{\circ}(x,t)
=\Psi(x/t,p^-)-\Psi(x/t,y_{j-1}).
\]
The difference \(d_p(x)-d_p^{\circ}(x,t)\) is a sum over
\(x/t<n\le x/p\), with \(y_{j-1}<P(n)<p\).
The same injective counting argument therefore gives
\[
\E\sum_{j=1}^K\sum_{p\in I_j}
\int|d_p(x)-d_p^{\circ}(x,t)|^2\nu_p(\dd t)
\le s_0x+1.
\]
Next,
\[
\widetilde d_p(x,t)-d_p^{\circ}(x,t)
=\sum_{\substack{n\le x/t\\p\le P(n)\le t}}f(n).
\]
Each integer in this sum has a largest prime factor
\(q\in[p,t]\). Orthogonality and an upper bound for the
number of multiples of \(q\) give
\[
\E|\widetilde d_p(x,t)-d_p^{\circ}(x,t)|^2
\le\frac xp\sum_{p\le q\le(1+s_0)p}\frac1q.
\]
The terms with \(q=p\), summed over \(p>y_0\), contribute
at most
\[
x\sum_{p>y_0}\frac1{p^2}
\le x\sum_{n>y_0}\frac1{n^2}
\le \frac{Cx}{y_0}.
\]
For \(q>p\), \eqref{eq:pnt} and partial summation, as in the proof of Lemma~\ref{lem:prime-averaging}, give, since \(e^{-c\sqrt{\log y_0}}\le s_0\),
\[
\sum_{p<q\le(1+s_0)p}\frac1q
\le\frac{C s_0}{\log p},\qquad p\ge y_0.
\]
Consequently, by \eqref{eq:prime-products}, the total
contribution from \(q>p\) is at most
\[
Cxs_0\sum_{p>y_0}\frac1{p\log p}
\le\frac{Cxs_0}{\log y_0}.
\]
Combining the two changes with
\(|a+b|^2\le2|a|^2+2|b|^2\), we obtain
\[
\E\sum_{j=1}^K\sum_{p\in I_j}
\int|d_p(x)-\widetilde d_p(x,t)|^2\nu_p(\dd t)
\ll s_0x+1+\frac{x}{y_0}.
\]

We now use the coefficient estimates to bound
\[
\sum_{j=1}^K|V_j(x)-\overline V_j(x)|
\qquad\text{and}\qquad
|D(x)-\overline D(x)|.
\]
Orthogonality and uniqueness of the largest-prime
representation give
\[
\E\sum_{j=1}^K V_j(x)\le x,
\qquad
\E D(x)\le x.
\]
Indeed, each term in either expectation corresponds to a
squarefree integer \(pn\le x\), with \(p=P(pn)\), and
such an integer is counted at most once.
The averaged vectors have expected squared norms \(O(x)\),
by the coefficient error estimates and
\(\|b\|^2\le2\|a\|^2+2\|a-b\|^2\).

Using \(a^2-b^2=(a-b)(a+b)\), followed by
Cauchy--Schwarz over the primes and the integration variable,
we obtain
\[
\begin{aligned}
|V_j(x)-\overline V_j(x)|
&=
\bigg|
\sum_{p\in I_j}\int
\bigl(b_p(x)-\widetilde b_p(x,t)\bigr)
\bigl(b_p(x)+\widetilde b_p(x,t)\bigr)
\,\nu_p(\dd t)
\bigg|\\
&\le
\bigg(
\sum_{p\in I_j}\int
|b_p(x)-\widetilde b_p(x,t)|^2\,\nu_p(\dd t)
\bigg)^{1/2}
\bigg(
\sum_{p\in I_j}\int
|b_p(x)+\widetilde b_p(x,t)|^2\,\nu_p(\dd t)
\bigg)^{1/2}\\
&\le
\bigg(
\sum_{p\in I_j}\int
|b_p(x)-\widetilde b_p(x,t)|^2\,\nu_p(\dd t)
\bigg)^{1/2}\sqrt{2\bigl(V_j(x)+\overline V_j(x)\bigr)}.
\end{aligned}
\]
Summing over \(j\) and applying Cauchy--Schwarz once more,
now to the finite sum over the layers, yields
\[
\sum_{j=1}^K|V_j(x)-\overline V_j(x)|
\le
\bigg(\sum_{j=1}^K
\sum_{p\in I_j}\int
|b_p(x)-\widetilde b_p(x,t)|^2\,\nu_p(\dd t)
\bigg)^{1/2}\bigg(2\sum_{j=1}^K
\bigl(V_j(x)+\overline V_j(x)\bigr)\bigg)^{1/2}.
\]
Applying Cauchy--Schwarz to the expectation, and arguing
in the same way for \(D(x)-\overline D(x)\), yields \eqref{eq:vvdd} and completes the first step.

\medskip\noindent
{\sc Step 2.}
Define
\begin{align}
\mathcal V_j(x)
&=\frac1x\int_{y_{j-1}}^{y_j}
\Psi(x/t,y_{j-1})^2\frac{\dd t}{\log t},
\label{eq:intV}\\
\mathcal V_D(x)
&=\frac1x\sum_{j=1}^K\int_{y_{j-1}}^{y_j}
\bigl(\Psi(x/t,t)-\Psi(x/t,y_{j-1})\bigr)^2
\frac{\dd t}{\log t}.
\label{eq:intD}
\end{align}
We shall prove that, with probability \(1-o(1)\),
\begin{equation}\label{eq:replaceprimes}
\sum_{j=1}^K\max_{x\in\mathcal X_Y}
\bigg|\frac{V_j(x)}x-\mathcal V_j(x)\bigg|+
\max_{x\in\mathcal X_Y}
\bigg|\frac{D(x)}x-\mathcal V_D(x)\bigg|
\le(\log Y)^{-2}.
\end{equation}

For each layer, put
\[
w_j(t)=\sum_{\substack{p\in I_j\\p\le t\le(1+s_0)p}}
\frac1{s_0p}.
\]
Tonelli's theorem gives
\[
\overline V_j(x)
=\int\Psi(x/t,y_{j-1})^2w_j(t)\dd t,
\]
and the analogous representation of \(\overline D(x)\)
with squared integrand
\(\bigl(\Psi(x/t,t)-\Psi(x/t,y_{j-1})\bigr)^2\).
The support of \(w_j\) is contained in
\([y_{j-1},(1+s_0)y_j]\).

Apply Lemma~\ref{lem:prime-averaging} with
\(a=y_{j-1}\), \(b=y_j\), and \(\delta=s_0\).
Uniformly over the layers, it gives
\[
w_j(t)=\frac{1+O((\log Y)^{-30})}{\log t}
\quad\text{if }(1+s_0)y_{j-1}<t\le y_j,
\]
and \(0\le w_j(t)\le C/\log t\) on its whole support.
Thus the weights agree, up to the stated relative error,
outside the endpoint strips
\[
[y_{j-1},(1+s_0)y_{j-1}]
\quad\text{and}\quad
[y_j,(1+s_0)y_j].
\]

By orthogonality \eqref{eq:orth},
\[
\E\Psi(x/t,y_{j-1})^2
\le\frac xt
\quad \text{and} \quad
\E\bigl(\Psi(x/t,t)-\Psi(x/t,y_{j-1})\bigr)^2
\le\frac xt.
\]

On the interior
\(
J_j=((1+s_0)y_{j-1},y_j],
\)
the prime-weight estimate gives
\[
\bigg|w_j(t)-\frac1{\log t}\bigg|
\le\frac{C(\log Y)^{-30}}{\log t}.
\]
Since the weights are deterministic, Tonelli's theorem
and the first second-moment bound yield
\[
\E\bigg[
\frac1x\int_{J_j}
\Psi(x/t,y_{j-1})^2
\bigg|w_j(t)-\frac1{\log t}\bigg|\dd t
\bigg]\le
C(\log Y)^{-30}\int_{J_j}\frac{\dd t}{t\log t}.
\]
The same estimate holds for the discarded integrand,
by its second-moment bound above.

Summing over the layers and enlarging \(J_j\) to
\((y_{j-1},y_j]\), we bound the expected total interior
error by
\[
C(\log Y)^{-30}
\sum_{j=1}^K\int_{y_{j-1}}^{y_j}\frac{\dd t}{t\log t}
=
C(\log Y)^{-30} \log\frac{\log y_K}{\log y_0} =
O((\log Y)^{-30}\log_3Y).\]

Each endpoint strip has logarithmic length at most \(s_0\),
and \(\log t\ge a_0\). Their total contribution is at most
\(CKs_0/a_0\). Hence
\[
\E\bigg[
\sum_{j=1}^K
\bigg|\frac{\overline V_j(x)}x-\mathcal V_j(x)\bigg|
+\bigg|\frac{\overline D(x)}x-\mathcal V_D(x)\bigg|
\bigg]\ll
(\log Y)^{-30}\log_3Y+\frac{Ks_0}{a_0}.
\]
Combining this with \eqref{eq:vvdd}, and using \(x\ge Y\), gives
\[
\E\bigg[
\sum_{j=1}^K\bigg|\frac{V_j(x)}x-\mathcal V_j(x)\bigg|
+\bigg|\frac{D(x)}x-\mathcal V_D(x)\bigg|
\bigg]
\ll(\log Y)^{-19}.
\]
The sum of maxima in \eqref{eq:replaceprimes} is bounded
by the sum of these nonnegative errors over all endpoints.
Its expectation is at most \(C(\log Y)^{-18}\), since
\(|\mathcal X_Y|\le\log Y\).
Markov's inequality at \((\log Y)^{-2}\) proves
\eqref{eq:replaceprimes}, with exceptional probability
\(O((\log Y)^{-16})\).

\medskip\noindent
{\sc Step 3.}
All suprema in the rest of the proof are over
\(Y\le x\le Y_+\). The change of variables \(u=x/t\)
in \eqref{eq:intV} gives
\begin{equation}\label{eq:Vchange}
\mathcal V_j(x)
=\int_{x/y_j}^{x/y_{j-1}}
\frac{\Psi(u,y_{j-1})^2}{u^2\log(x/u)}\dd u.
\end{equation}
The symmetric difference of the integration intervals at
\(x\) and \(Y\) is contained in
\[
\mathscr B_j
=[Y/y_j,Y_+/y_j]
\cup[Y/y_{j-1},Y_+/y_{j-1}].
\]
On either integration interval, \(\log(x/u)\ge a_0 = \log y_0\).
Therefore, using \eqref{eq:smoothsecond}, the expected
contribution of these moving endpoints is at most
\[
\frac1{a_0}\int_{\mathscr B_j}
\frac{\E\Psi(u,y_{j-1})^2}{u^2}\dd u
\le\frac1{a_0}\int_{\mathscr B_j}\frac{\dd u}{u}
\le\frac{2\log(Y_+/Y)}{a_0}
=2(\log_2Y)^{-98}.
\]
The sums vanish for \(u<1\), so this estimate also covers
intervals crossing one.

On the common part of the integration intervals, put
\(r=\log(x/Y)\). Then \(0\le r\le\log(Y_+/Y)\) and
\[
\bigg|\frac1{\log(x/u)}-\frac1{\log(Y/u)}\bigg|
=\frac{r}{\log(x/u)\log(Y/u)}
\le\frac{\log(Y_+/Y)}{a_0\log(Y/u)}.
\]
The resulting error is bounded by
\(\log(Y_+/Y)\mathcal V_j(Y)/a_0\).
Moreover, by \eqref{eq:smoothsecond}
\[
\E\mathcal V_j(Y)
\le\frac1{a_0}\int_{Y/y_j}^{Y/y_{j-1}}
\frac{\dd u}{u}
\le\frac{2\log Y}{a_0}
=2(\log_2Y)^2.
\]
Since \(\log(Y_+/Y)/a_0=(\log_2Y)^{-98}\), we conclude that
\[
\E\sup_{Y\le x\le Y_+}
|\mathcal V_j(x)-\mathcal V_j(Y)|
\le C(\log_2Y)^{-96}.
\]
Summing over the layers gives
\begin{equation}\label{eq:vvariation}
\E\sum_{j=1}^K\sup_{Y\le x\le Y_+}
|\mathcal V_j(x)-\mathcal V_j(Y)|
\ll K(\log_2Y)^{-96}.
\end{equation}

\medskip\noindent
{\sc Step 4.}
For \(t\in I_j\), put \(y_-(t)=y_{j-1}\), and define
\[
g_x(u)=
\begin{cases}
\displaystyle
\frac{\Psi(u,x/u)-\Psi(u,y_-(x/u))}
{u\sqrt{\log(x/u)}},
&y_0<x/u\le y_K,\\[6pt]
0,&\text{otherwise}.
\end{cases}
\]
The same change of variables as above gives
\[
\mathcal V_D(x)=\int_0^\infty|g_x(u)|^2\dd u
=\|g_x\|_2^2.
\]
We first estimate the change in \(g_x\) in squared
\(L^2\)-norm. Let
\[
\mathscr B=\bigcup_{j=0}^K[Y/y_j,Y_+/y_j].
\]
These are precisely the intervals in which a cutoff \(x/u\)
can cross a layer boundary as \(x\) varies from \(Y\) to
\(Y_+\). If \(u\notin\mathscr B\), either \(g_x(u)=0\)
for every \(x\), or all the cutoffs \(x/u\) lie in the
same layer. In the latter case, \(y_-(x/u)\) is constant.

Put \(t_0=Y/u\) in this latter case. Adding and subtracting
the numerator at \(Y\), with the denominator at \(x\), gives
\[
g_x(u)-g_Y(u)
={}\frac{\Psi(u,x/u)-\Psi(u,t_0)}
{u\sqrt{\log(x/u)}}
+\frac{\Psi(u,t_0)-\Psi(u,y_-(t_0))}{u}
\bigg(\frac1{\sqrt{\log(x/u)}}
-\frac1{\sqrt{\log t_0}}\bigg).
\]
We estimate these two terms separately.

For the first term, \eqref{eq:incrementprime} together with
\eqref{eq:primesosc} at frequency zero,
uniformly
for the relevant \(y_0\le t_0\le y_K\), gives
\[
\E\sup_{t_0\le v\le(Y_+/Y)t_0}
|\Psi(u,v)-\Psi(u,t_0)|^2
\le4u\sum_{t_0<p\le(Y_+/Y)t_0}\frac1p
\le\frac{Cu\log(Y_+/Y)}{\log t_0}.
\]
Since \(\log(x/u)\ge\log t_0\ge a_0\), the expected
integral of the squared first term, with the supremum
over \(x\) inside the integral, is at most
\[
C\log(Y_+/Y)
\int_{Y/y_K}^{Y/y_0}
\frac{\dd u}{u(\log(Y/u))^2}
\le\frac{C\log(Y_+/Y)}{a_0^2}
\log\frac{y_K}{y_0}
\le C(\log_2Y)^{-96}.
\]

For the second term, the mean value theorem gives
\[
\bigg|\frac1{\sqrt{\log(x/u)}}
-\frac1{\sqrt{\log t_0}}\bigg|
\le
\frac{\log(Y_+/Y)}{2a_0\sqrt{\log t_0}}.
\]
Its absolute value is therefore at most
\(\log(Y_+/Y)|g_Y(u)|/(2a_0)\).
Orthogonality gives
\[
\E\|g_Y\|_2^2
\le\int_{Y/y_K}^{Y/y_0}
\frac{\dd u}{u\log(Y/u)}\le\frac{2\log Y}{a_0}
=2(\log_2Y)^2.
\]
Thus the squared second term contributes at most
\[
C\bigg(\frac{\log(Y_+/Y)}{a_0}\bigg)^2
\E\|g_Y\|_2^2
\ll(\log_2Y)^{-194}.
\]

It remains to treat \(u\in\mathscr B\), where a layer
boundary or a support boundary may be crossed. Whenever
\(g_x(u)\ne0\), its denominator is at least
\(u\sqrt{a_0}\), and its numerator is at most
\(2\sup_{v\ge1}|\Psi(u,v)|\) in absolute value.
Hence
\[
\sup_{Y\le x\le Y_+}|g_x(u)-g_Y(u)|^2
\le\frac{16}{u^2a_0}
\sup_{v\ge1}|\Psi(u,v)|^2.
\]
By \eqref{eq:smoothmaxprime}, its expected integral over
\(\mathscr B\) is bounded by
\[
\frac{C}{a_0}\int_{\mathscr B}\frac{\dd u}{u}
\le\frac{C(K+1)\log(Y_+/Y)}{a_0}
\ll K(\log_2Y)^{-98}.
\]
Combining the estimates inside and outside \(\mathscr B\),
and using \(|a+b|^2\le2|a|^2+2|b|^2\), we obtain
\begin{equation}\label{eq:gvariation}
\E\int_0^\infty
\sup_{Y\le x\le Y_+}|g_x(u)-g_Y(u)|^2\dd u
\ll K(\log_2Y)^{-96}.
\end{equation}

Set
\[
A_Y=\int_0^\infty
\sup_{Y\le x\le Y_+}|g_x(u)-g_Y(u)|^2\dd u.
\]
For every \(x\in[Y,Y_+]\), we have
\(\|g_x-g_Y\|_2\le\sqrt{A_Y}\). Therefore,
\[
\sup_{Y\le x\le Y_+}|\mathcal V_D(x)-\mathcal V_D(Y)|
=\sup_{Y\le x\le Y_+}
\big|\|g_x\|_2^2-\|g_Y\|_2^2\big|
\le2\|g_Y\|_2\sqrt{A_Y}+A_Y.
\]
Taking expectations and applying Cauchy--Schwarz gives
\[
\E\sup_{Y\le x\le Y_+}|\mathcal V_D(x)-\mathcal V_D(Y)|
\le2\sqrt{\E\|g_Y\|_2^2}\sqrt{\E A_Y}+\E A_Y
\ll\sqrt K\,(\log_2Y)^{-47}
+K(\log_2Y)^{-96}.
\]
The ratio of the second term to the first is
\(\sqrt K\,(\log_2Y)^{-49}\), which tends to zero. Therefore we have proved
\begin{equation}\label{eq:dvariation}
\E\sup_{Y\le x\le Y_+}|\mathcal V_D(x)-\mathcal V_D(Y)|
\le C \sqrt K\,(\log_2Y)^{-47}.
\end{equation}

\medskip\noindent
{\sc Step 5.}
By \eqref{eq:vvariation}, \eqref{eq:dvariation}, and
Markov's inequality, with probability \(1-o(1)\),
\[
\sum_{j=1}^K\sup_{Y\le x\le Y_+}
|\mathcal V_j(x)-\mathcal V_j(Y)|
+\sup_{Y\le x\le Y_+}|\mathcal V_D(x)-\mathcal V_D(Y)|
\le\tfrac12(\log_2Y)^{-10}.
\]
Intersect this event with the event in
\eqref{eq:replaceprimes}. Since \(Y\in\mathcal X_Y\),
the triangle inequality gives
\[
\begin{split}
\sum_{j=1}^K\max_{x\in\mathcal X_Y}
\bigg|\frac{V_j(x)}x-\frac{V_j(Y)}Y\bigg|
& \le
2\sum_{j=1}^K\max_{x\in\mathcal X_Y}
\bigg|\frac{V_j(x)}x-\mathcal V_j(x)\bigg|
+\sum_{j=1}^K\sup_{Y\le x\le Y_+}
|\mathcal V_j(x)-\mathcal V_j(Y)|\\
& \le 2(\log Y)^{-2}+\tfrac12(\log_2Y)^{-10}
\le(\log_2Y)^{-10}
\end{split}
\]
for all sufficiently large \(Y\). The same argument yields
\[
\max_{x\in\mathcal X_Y}
\bigg|\frac{D(x)}x-\frac{D(Y)}Y\bigg|
\le
2\max_{x\in\mathcal X_Y}
\bigg|\frac{D(x)}x-\mathcal V_D(x)\bigg|
+\sup_{Y\le x\le Y_+}|\mathcal V_D(x)-\mathcal V_D(Y)|
\le(\log_2Y)^{-10}.
\]
This proves \eqref{eq:flatV} and \eqref{eq:flatD}
simultaneously.

Finally, Proposition~\ref{prop:collision} gives
\[
\frac{D(Y)}Y\le(\log_2Y)^{-53/100}
\]
with probability \(1-o(1)\). On the intersection of this
event with the events already obtained,
\[
\max_{x\in\mathcal X_Y}\frac{D(x)}x
\le\frac{D(Y)}Y
+\max_{x\in\mathcal X_Y}
\bigg|\frac{D(x)}x-\frac{D(Y)}Y\bigg|
\le2(\log_2Y)^{-53/100}.
\]
This proves \eqref{eq:uniformcollision}. All the events
used above have probabilities tending to one, so their
intersection has probability tending to one as well.
\end{proof}


\subsection{Proposition \ref{prop:goodlayers}: the retained layers}
Recall that, for \(p\in I_j=(y_{j-1},y_j]\),
\[
b_p(x)=\Psi(x/p,y_{j-1})
=\sum_{\substack{n\le x/p\\P(n)\le y_{j-1}}}f(n),
\qquad
V_j(x)=\sum_{p\in I_j}b_p(x)^2.
\]
The coefficients \(b_p(x)\) are
\(\FF_{y_{j-1}}\)-measurable. Thus \(V_j(x)\) is the
conditional variance, given \(\FF_{y_{j-1}}\), of the layer sum
\[
R_j(x)=\sum_{p\in I_j}\varepsilon_pb_p(x).
\]
Recall also

\begin{equation}\label{eq:active}
v_j=\max_{x\in\mathcal X_Y}\frac{V_j(x)}x,
\qquad
\mathcal J=\{j:v_j\ge(\log_2Y)^{-59/100}\},
\qquad
q=\sum_{j\in\mathcal J}v_j.
\end{equation}

\begin{proposition}\label{prop:goodlayers}
There is a deterministic constant \(c>0\) such that, with
probability \(1-o(1)\) as \(Y\to\infty\), the following
bounds hold simultaneously:
\begin{align}
c(\log_2Y)^{-51/100}&\le q
\le 2(\log_2Y),
\label{eq:qbounds}\\
(1-(\log_2Y)^{-1})v_j
&\le\frac{V_j(x)}x
\le v_j
\qquad(j\in\mathcal J,\ x\in\mathcal X_Y),
\label{eq:activeflat}\\
\max_{\substack{j\in\mathcal J,\ p\in I_j\\x\in\mathcal X_Y}}
\frac{|b_p(x)|}{\sqrt{V_j(x)}}
&\le\beta_Y
:=\exp\bigg(-\frac{\log Y}{3(\log_2Y)^2}\bigg).
\label{eq:smallcoeff}
\end{align}
\end{proposition}

\begin{lemma}\label{lem:smoothmaxx}
For \(y\ge1\) and \(U\ge2\),
\begin{equation}\label{eq:smoothmaxargument}
\E\max_{0\le u\le U}|\Psi(u,y)|^2
\le CU(1+\log U)^2,
\end{equation}
where \(C>0\) is an absolute constant.
\end{lemma}

\begin{proof}
Apply the Menshov--Rademacher maximal inequality
\cite[Lemma~1, (11), p.~333]{MM} to the orthogonal sequence
\((f(n)\1_{\{P(n)\le y\}})_{1\le n\le\lfloor U\rfloor}\),
omitting its zero terms. By \eqref{eq:orth},
\[
\E\max_{0\le u\le U}|\Psi(u,y)|^2
\le C(1+\log U)^2
\sum_{\substack{n\le U\\P(n)\le y}}\E f(n)^2
\le CU(1+\log U)^2.
\]
\end{proof}

\begin{proof}[Proof of Proposition~\ref{prop:goodlayers}]
We use the event in Proposition~\ref{prop:flatness}, on which
\eqref{eq:flatV} holds. Its probability tends to one.

\medskip\noindent
{\sc Step 1.}
For sufficiently large \(Y\), we have \(y_0<Y^{3/4}\) and
\(y_K\ge Y\). Hence every prime \(Y^{3/4}<p\le Y\) belongs
to one of the layers. If \(p\in I_j\), the identity
\(y_{j-1}=y_j^{e^{-\delta}}\) gives
\[
y_{j-1}\ge p^{e^{-\delta}}
>Y^{(3/4)e^{-\delta}}
>Y^{1/2}
\ge Y/p
\]
for all sufficiently large \(Y\), since
\(\delta=(\log_2Y)^{-1/20}\to0\).
Every integer at most \(Y/p\) then has all its prime factors
at most \(y_{j-1}\). Thus
\[
b_p(Y)=\Psi(Y/p,y_{j-1})=M(Y/p),
\]
and consequently
\[
\frac1Y\sum_{j=1}^K V_j(Y)
\ge\frac1Y\sum_{Y^{3/4}<p\le Y}M(Y/p)^2.
\]

We apply Harper's Proposition~1
\cite[Proposition~1]{HarperLarge} with
\(
A=Y^{3/4}, z=Y=A^{4/3},
W=\frac{\log_3Y}{220}.
\)
These parameters are admissible: \(A^{8/7}=Y^{6/7}\le Y\),
and \(1\le W\le(\log_2A)^{1/100}\) for all sufficiently
large \(Y\). That proposition supplies the deterministic
comparison and the probabilistic estimate which together give
\[
\frac1Y\sum_{j=1}^K V_j(Y)
\ge\frac{c_{\mathrm H}}{\log A}
\int_1^{A^{1/7}}\frac{M(u)^2}{u^2}\dd u
\ge\frac{c_{\mathrm H}e^{-2.2W}}{\sqrt{\log_2A}}
\]
outside an event of probability at most
\(C_{\mathrm H}e^{-0.1W}\).
Since
\(
\log_2A \sim \log_2Y\),
\( e^{-2.2W}=(\log_2Y)^{-1/100}\),
\(e^{-0.1W}=(\log_2Y)^{-1/2200}\),
there is a deterministic \(c_1>0\) such that
\begin{equation}\label{eq:retainedlower}
\PP\bigg(
\frac1Y\sum_{j=1}^K V_j(Y)
\ge c_1(\log_2Y)^{-51/100}
\bigg)
\ge1-C(\log_2Y)^{-1/2200}.
\end{equation}

The inactive layers satisfy the deterministic bound
\begin{equation}\label{eq:inactive}
\sum_{j\notin\mathcal J}v_j
\le K(\log_2Y)^{-59/100}
=o\bigl((\log_2Y)^{-51/100}\bigr).
\end{equation}
Since \(v_j\ge V_j(Y)/Y\),
\[
q
=\sum_{j=1}^K v_j-\sum_{j\notin\mathcal J}v_j
\ge\frac1Y\sum_{j=1}^K V_j(Y)
-\sum_{j\notin\mathcal J}v_j.
\]
Combining \eqref{eq:retainedlower} and \eqref{eq:inactive}
proves the lower bound in \eqref{eq:qbounds}, with
\(c=c_1/2\), for all sufficiently large \(Y\).

For the upper bound, orthogonality gives
\[
\E\sum_{j=1}^K V_j(Y)
=\sum_{j=1}^K\sum_{p\in I_j}
\sum_{\substack{n\le Y/p\\P(n)\le y_{j-1}}}\mu^2(n)
\le Y.
\]
Markov's inequality yields
\[
\PP\bigg(
\frac1Y\sum_{j=1}^K V_j(Y)>\log_2Y
\bigg)
\le\frac1{\log_2Y}.
\]
On the event \eqref{eq:flatV},
\[
0\le\sum_{j=1}^K v_j-\frac1Y\sum_{j=1}^K V_j(Y)
\le\sum_{j=1}^K\max_{x\in\mathcal X_Y}
\bigg|\frac{V_j(x)}x-\frac{V_j(Y)}Y\bigg|
\le(\log_2Y)^{-10}.
\]
Thus, outside an event of probability \(o(1)\),
\[
q\le\sum_{j=1}^K v_j
\le\log_2Y+(\log_2Y)^{-10}
\le 2(\log_2Y).
\]
This completes the proof of \eqref{eq:qbounds}.

\medskip\noindent
{\sc Step 2.}
Fix an active layer \(j\), and choose \(x_j\in\mathcal X_Y\)
such that \(v_j=V_j(x_j)/x_j\). Such a point exists because
the grid is finite. By \eqref{eq:flatV}, for every
\(x\in\mathcal X_Y\),
\[
0\le v_j-\frac{V_j(x)}x
\le\bigg|\frac{V_j(x_j)}{x_j}-\frac{V_j(Y)}Y\bigg|
+\bigg|\frac{V_j(x)}x-\frac{V_j(Y)}Y\bigg|
\le2(\log_2Y)^{-10}.
\]
Since the layer is active, the definition of \(\mathcal J\)
in \eqref{eq:active} gives
\(
v_j\ge(\log_2Y)^{-59/100}
\). Combining with the above inequality we obtain
\eqref{eq:activeflat} simultaneously for all
active layers and endpoints.

\medskip\noindent
{\sc Step 3.}
Put
\(
U={Y^{1+(\log_2Y)^{-100}}}/{y_0}.
\)
For every \(x\in\mathcal X_Y\) and \(p\in I_j\), we have
\(x/p\le U\). Also, \(U\ge2\) and \(\log U\le2\log Y\)
for sufficiently large \(Y\).
For each deterministic cutoff \(y_{j-1}\),
\eqref{eq:smoothmaxargument} and Markov's inequality give
\[
\PP\bigg(
\max_{0\le u\le U}|\Psi(u,y_{j-1})|
>\sqrt U\,(\log Y)^2
\bigg)\le
\frac{CU(1+\log U)^2}{U(\log Y)^4}
\ll(\log Y)^{-2}.
\]
A union bound over the \(K\) cutoffs shows that, outside
an event of probability \(O(K(\log Y)^{-2})=o(1)\),
\[
|b_p(x)|
=|\Psi(x/p,y_{j-1})|
\le\sqrt U\,(\log Y)^2
\]
simultaneously for all layers, their primes, and all endpoints.

On the event \eqref{eq:activeflat}, every active layer satisfies
\[
V_j(x)
\ge(1-(\log_2Y)^{-1})xv_j
\ge\tfrac12Y(\log_2Y)^{-59/100}
\]
for sufficiently large \(Y\). Therefore,
\[
\begin{split}
\frac{|b_p(x)|}{\sqrt{V_j(x)}}
&\le C(\log Y)^2(\log_2Y)^{59/200}\sqrt{\frac UY}\\
&=C(\log Y)^2(\log_2Y)^{59/200}
\exp\bigg(
-\frac{\log Y}{2(\log_2Y)^2}
+\frac{\log Y}{2(\log_2Y)^{100}}
\bigg)\\
&\le\exp\bigg(-\frac{\log Y}{3(\log_2Y)^2}\bigg),
\end{split}
\]
uniformly over the active layers and endpoints.
This is \eqref{eq:smallcoeff}.

Finally, the flatness event and the additional events used
in Steps~1 and~3 all have probabilities tending to one.
Their intersection proves all three assertions simultaneously.
\end{proof}


\subsection{Proposition \ref{prop:covariance}: covariance estimates}\label{app:covariance-proofs}

We estimate the covariances arising from a single prime layer.
For \(Y<Y'\le Y^2\), put
\[
C_Y(x,z)
=
\sum_{Y<p\le Y'}\Psi(x/p,Y)\Psi(z/p,Y).
\]
Conditional on \(\FF_Y\), this is the covariance of the sums
\[
\sum_{Y<p\le Y'}\varepsilon_p\Psi(x/p,Y)
\quad\text{and}\quad
\sum_{Y<p\le Y'}\varepsilon_p\Psi(z/p,Y).
\]
We shall show that, on a logarithmic grid, each endpoint has
only a small number of other endpoints with a large covariance.

\begin{proposition}\label{prop:covariance}
Let \(Y\) be sufficiently large and let \(Y<Y'\le Y^2\).
Let \(x_0\ge(\log Y)^9\) be deterministic, and let \(m\) be a
positive integer satisfying
\(
m\le\frac{\log Y}{(\log_2Y)^{10}}.
\)
Put
\(
\mathcal Z=\{x_0e^r:0\le r<m\}.
\)
For every fixed \(0<\theta<1\),

\begin{equation}\label{eq:covfailure}
\begin{split}
&\PP\biggl(
\max_{x\in\mathcal Z}
\#\biggl\{
z\in\mathcal Z\setminus\{x\}:
|C_Y(x,z)|>
\frac{\sqrt{xz}}{(\log_2Y)^{3/5}}
\biggr\}
>(\log Y)^\theta
\biggr)\ll_\theta
\frac{(\log_3Y)^4}{(\log_2Y)^{1/10}}.
\end{split}
\end{equation}

The bound is uniform in \(Y'\), \(x_0\), and \(m\) satisfying
the stated assumptions.
\end{proposition}

\begin{proof}[Proof of Proposition~\ref{prop:covariance}]
We adapt the proof of \cite[Proposition~4.2, (4.3)]{BIRestricted}.
We check the changes required by the prime interval and the grid,
and retain the quantitative error bounds from that argument.

\medskip\noindent
{\sc Step 1.}
Put
\[
h=(\log Y)^8,\qquad
w=\frac{(\log_2Y)^4}{\log Y},\qquad
K_Y(u)=\sum_{Y<p\le Y'}p^{-1-iu}.
\]
For \(0\le j\le J_*:=\lfloor\log_2Y-1\rfloor\), put
\[
q_j=e^{-j}\log Y,\qquad
F_{j,Y}(t)=\prod_{p\le\exp(q_j)}
(1+\varepsilon_pp^{-1/2-it}),\qquad F_Y=F_{0,Y}.
\]
Fix \(B\ge1\) and \(0<b<1\), to be chosen in Step~3, and define
the barriers from \cite[(6.17)--(6.18)]{BIRestricted} at cutoff \(Y\):
\[
\begin{split}
\mathcal D(t)
&=\bigcap_{j=0}^{J_*}
\{|F_{j,Y}(t)|\le q_j(\log_2Y)^2\},\\
\mathcal A(t)
&=\mathcal D(t)\cap
\bigcap_{0\le j\le\lfloor(1-b)\log_2Y\rfloor}
\{|F_{j,Y}(t)|\le q_j(\log_2Y)^{-B}\}.
\end{split}
\]
The proofs of \cite[Lemmas~6.5, 6.7 and 6.10]{BIRestricted}
apply at every sufficiently large cutoff \(Y\): they use only
the scales \(q_j\) and the independent prime signs, with no
restriction on the sequence of cutoffs or on the endpoint grid.

Let
\[
\Delta=\{(v,t):(\log_2Y)^{-2}\le |v|,|t|\le(\log_2Y)^2,
\ |v-t|\le w\}
\]
and, for \(x,z\in\mathcal Z\), put
\[
I_{xz}=\frac1{4\pi^2}\iint_\Delta
\frac{F_Y(v)\overline{F_Y(t)}\,x^{iv}z^{-it}
\1_{\mathcal A(v)}\1_{\mathcal A(t)}K_Y(v-t)}
{(1/2+iv)(1/2-it)}\dd v\dd t.
\]
We claim that the event
\[
\mathcal T_Y=\biggl\{
\max_{x,z\in\mathcal Z}
\biggl|\frac{C_Y(x,z)}{\sqrt{xz}}-I_{xz}\biggr|
\le(\log_2Y)^{-5/8}\biggr\}
\]
satisfies
\begin{equation}\label{eq:B1-cited-truncation}
\PP(\mathcal T_Y^c)
\ll_{B,b}\frac{(\log_3Y)^4}{(\log_2Y)^{1/10}}.
\end{equation}

Here are the checks needed to apply the proof of
\cite[Lemmas~6.4 and 6.8]{BIRestricted}.
First, the finite expansion
\[
F_Y(t)=\sum_{\substack{n\ {\rm squarefree}\\P(n)\le Y}}
\frac{f(n)}{n^{1/2+it}}
\]
allows Perron's formula to be applied directly to \(\Psi(x/p,Y)\).
The squared truncation error in the proof of
\cite[Lemma~6.4, (6.11)--(6.14)]{BIRestricted} is bounded by
\[
\begin{split}
\sum_{Y<p\le Y'}\ \sum_{\substack{n\ {\rm squarefree}\\P(n)\le Y}}
\frac1{pn}\min\biggl\{1,
\frac1{h^2|\log(x/(pn))|^2}\biggr\}
\ll x^{-1}+h^{-1}+(\log Y)h^{-2}\ll h^{-1}.
\end{split}
\]
Indeed, \((p,n)\mapsto pn\) is injective, so the terms with
\(x/e\le pn\le ex\) are estimated by summing over integers;
the remaining terms are bounded by the Euler product
\(\prod_{p\le Y}(1+p^{-1})\ll\log Y\).
Since \(x\ge(\log Y)^9\) and \(m\le\log Y\), the union bounds
in \cite[(6.15)--(6.16)]{BIRestricted} give a common Perron
error \(O((\log Y)^{-1})\), outside an event of probability
\(O((\log Y)^{-1})\). This requires no upper bound on \(x\).

Second, \cite[Lemma~6.1 and (6.2)]{BIRestricted}, together
with \(Y'\le Y^2\), give
\[
|K_Y(u)|\ll\min\{1,((\log Y)|u|)^{-1}\},\quad |u|\le2h,
\qquad
\int_{-w}^w|K_Y(u)|\dd u\ll\frac{\log_3Y}{\log Y}.
\]
These are the kernel bounds used in the proof of
\cite[Lemma~6.8]{BIRestricted}. All endpoint phases have
modulus one, so its integral error bounds hold simultaneously
for our grid. To retain a rate, use the two thresholds in
\cite[(6.44)]{BIRestricted}. By
\cite[Lemma~6.7]{BIRestricted}, their combined exceptional
probability is
\(O_{B,b}((\log_3Y)^4(\log_2Y)^{-1/10})\).
The other exceptional probabilities in that proof are
\[
O\biggl((\log_2Y)^{-1/4}
+\frac{\log_3Y}{\log_2Y}+\frac1{\log Y}\biggr),
\]
which are absorbed by the same bound. This proves
\eqref{eq:B1-cited-truncation}.

\medskip\noindent
{\sc Step 2.}
We verify the grid conditions used in the proof of
\cite[Lemma~6.11]{BIRestricted}. Write
\(\operatorname{dist}(u,2\pi\mathbb Z)=\inf_{r\in\mathbb Z}|u-2\pi r|\) and put
\[
D_m(u)=\min\{m,C\operatorname{dist}(u,2\pi\mathbb Z)^{-1}\},\qquad
\mathscr K_{m,n}(u)=\sup_{|s-u|\le nw}D_m(s),
\]
where \(D_m(u)=m\) for \(u\in2\pi\mathbb Z\).
The geometric-series formula gives
\[
\biggl|\sum_{z\in\mathcal Z}z^{-iu}\biggr|
=\biggl|\sum_{r=0}^{m-1}e^{-iru}\biggr|
\le D_m(u).
\]
For \(n=2k\le2\log_2Y\), our assumption on \(m\) gives
\[
nwm\le2(\log_2Y)^{-5}.
\]
The periodic-kernel estimate \cite[(6.47)]{BIRestricted}
therefore yields
\[
\int_0^{2\pi}\mathscr K_{m,n}(u)\dd u
\ll 1+\log m+nwm\ll\log_2Y.
\]
Consequently,
\[
\sup_{\substack{J\subset\mathbb R\ {\rm interval}\\
|J|\le(\log_2Y)^2}}
\int_J\mathscr K_{m,n}(u)\dd u\ll(\log_2Y)^3,
\qquad \mathscr K_{m,n}(u)\le m\le\log Y.
\]
These replace precisely \cite[(6.59)--(6.60)]{BIRestricted}.
The factor containing \(x_0\) has modulus one and disappears
when the phase sum is bounded.
The remainder of the proof of
\cite[Lemma~6.11, (6.61)--(6.74)]{BIRestricted} uses only
these bounds, the kernel integral from Step~1, and the joint
Euler-product estimate of \cite[Lemma~6.10]{BIRestricted}.
It thus gives an absolute constant \(C_0\ge1\), independent
of \(B,b\), such that
\begin{equation}\label{eq:B1-cited-moment}
\E\max_{x\in\mathcal Z}\sum_{z\in\mathcal Z}|I_{xz}|^{2k}
\le(C_0k)^{C_0k}(\log_2Y)^{C_0k}
\bigl((\log Y)^{2b}
+(\log Y)(\log_2Y)^{-2Bk}\bigr)
\end{equation}
for every integer \(1\le k\le\lfloor\log_2Y\rfloor\).

\medskip\noindent
{\sc Step 3.}
Fix \(0<\theta<1\), put \(C_1=3C_0+10\), choose
\(b=\theta/8\) and a fixed \(B>16(C_1+1)/\theta\), and set
\[
k=\biggl\lfloor\frac{\log_2Y}{B\log_3Y}\biggr\rfloor.
\]
For sufficiently large \(Y\), the logarithm of the prefactor
in \eqref{eq:B1-cited-moment}, including
\((2(\log_2Y)^{3/5})^{2k}\), is at most
\(C_1k\log_3Y\).
Since \(k\log_3Y=(\log_2Y)/B+O(\log_3Y)\), the two terms
in that bound give, respectively, exponents
\[
2b+C_1/B+o(1)<\theta/2,\qquad
-1+C_1/B+o(1)<-1/2
\]
as powers of \(\log Y\). Hence
\[
\E\biggl[(2(\log_2Y)^{3/5})^{2k}
\max_{x\in\mathcal Z}\sum_{z\in\mathcal Z}|I_{xz}|^{2k}
\biggr]\le2(\log Y)^{\theta/2}.
\]
Put
\[
N_x=\#\{z\in\mathcal Z\setminus\{x\}:
|C_Y(x,z)|>(\log_2Y)^{-3/5}\sqrt{xz}\}.
\]
On \(\mathcal T_Y\), every pair counted by \(N_x\) satisfies
\[
|I_{xz}|>(\log_2Y)^{-3/5}-(\log_2Y)^{-5/8}
\ge\tfrac12(\log_2Y)^{-3/5}
\]
for sufficiently large \(Y\). Thus
\[
\max_xN_x\le(2(\log_2Y)^{3/5})^{2k}
\max_x\sum_z|I_{xz}|^{2k}
\qquad\text{on }\mathcal T_Y.
\]
Markov's inequality and \eqref{eq:B1-cited-truncation} give
\[
\PP\bigl(\max_xN_x>(\log Y)^\theta\bigr)
\le\PP(\mathcal T_Y^c)+2(\log Y)^{-\theta/2}
\ll_\theta\frac{(\log_3Y)^4}{(\log_2Y)^{1/10}}.
\]
Every estimate is uniform in \(Y'\), \(x_0\) and \(m\)
in the stated ranges. This completes the proof.
\end{proof}

\subsection{Lemma~\ref{lem:count}: a quantitative counting estimate}
The following lemma is a quantitative extension of
\cite[Lemma~4.4]{BIRestricted}.
\begin{lemma}\label{lem:count}
Fix \(0<\gamma<1/20\) and \(N>0\). Let \(Y\) be sufficiently
large, and let \(\mathcal X\) be a finite set with
\(|\mathcal X|\le\log Y\). Suppose that
\[
Z_x=\sum_{i=1}^d c_{i,x}\varepsilon_i,
\qquad x\in\mathcal X,
\]
where the coefficients are deterministic and real, and
\(\varepsilon_1,\ldots,\varepsilon_d\) are independent
Rademacher variables. Assume that

\[
\sum_{i=1}^d c_{i,x}^2=1,
\qquad
|c_{i,x}|\le\beta_Y
:=\exp\biggl\{-\frac{\log Y}{3(\log_2Y)^2}\biggr\}.
\]

Suppose further that there is a graph on \(\mathcal X\)
of maximum degree at most \((\log Y)^\gamma\) such that
\[
\E[Z_xZ_z]\le\rho_Y:=4(\log_2Y)^{-1/100}
\]
whenever \(x\) and \(z\) are distinct and nonadjacent.

There is a constant \(C_{\gamma,N}>0\) such that, uniformly
over all \(S\subset\mathcal X\) and
\(
u\ge C_{\gamma,N}\sqrt{\log_3Y}
\)
satisfying

\begin{equation}\label{eq:budget}
|S|\exp\biggl\{-\frac{(1+\gamma)u^2}{2}\biggr\}
\ge(\log Y)^{3\gamma},
\end{equation}
we have
\begin{equation}\label{eq:count}
\PP\biggl(
\#\{x\in S:Z_x\ge u\}
<
\frac{|S|}{8}
\exp\biggl\{-\frac{(1+\gamma)u^2}{2}\biggr\}
\biggr)
\ll_{\gamma,N}(\log_2Y)^{-N}.
\end{equation}

The same assertion holds with \(Z_x\) replaced by \(-Z_x\).
\end{lemma}

\begin{proof}

\medskip
\noindent
{\sc Step 1.}
A greedy coloring of the graph induced on \(S\) uses at most
\((\log Y)^\gamma+1\) colors. Within each color class,
form groups of exactly \(g = \lceil e^{(1+\gamma)u^2/2}\rceil \) vertices and discard the
remaining vertices.

Let \(J\) denote the total number of full groups. Since
fewer than \(g\) vertices are discarded from each color class,
\[
J\ge\frac{|S|}{g}-\bigl((\log Y)^\gamma+1\bigr).
\]
By \eqref{eq:budget},
\(
{|S|}/{g}\ge\frac12(\log Y)^{3\gamma}
\), therefore
for sufficiently large \(Y\)
\begin{equation}\label{eq:count-groups}
J\ge\frac{|S|}{2g}.
\end{equation}
Each group has \(g\le|S|\le\log Y\) vertices, and the
covariance of any two distinct coordinates in a group
is at most \(\rho_Y\).

\medskip
\noindent
{\sc Step 2.}
Fix one of the groups, denote it \(A\), and put
\[
H_A=\max_{x\in A}Z_x.
\]
We follow the proof of \cite[Lemma~4.4]{BIRestricted},
keeping the estimates needed for a quantitative bound on
\(\PP(H_A\le u)\). The statement of that lemma alone gives
convergence to zero, whereas here we need the rate in
\eqref{eq:count-one-group} below.

We first record the range of the threshold. Since
\(|S|\le\log Y\), taking logarithms in \eqref{eq:budget} gives
\[
\frac{(1+\gamma)u^2}{2}
\le\log|S|-3\gamma\log_2Y
\le(1-3\gamma)\log_2Y.
\]
Consequently,
\begin{equation}\label{eq:count-threshold-range}
C_{\gamma,N}^2\log_3Y\le u^2
\le\frac{2(1-3\gamma)}{1+\gamma}\log_2Y
\le2\log_2Y.
\end{equation}
Put
\( t=(1+\gamma/4)u, h=t+4.\)
The level \(t\) is slightly above \(u\). We shall first obtain
a lower bound for \(\PP(H_A>t)\), and then use concentration
to show that \(\PP(H_A\le u)\) is small.

\smallskip\noindent
\emph{(a).}
We apply the exponential-tilting argument leading to
\cite[(4.14)]{BIRestricted}. Indeed, with \(h=t+4\),
the coefficient assumptions and \eqref{eq:count-threshold-range}
give
\[
\sum_{i=1}^d c_{i,x}^4\le\beta_Y^2,
\qquad
\beta_Y^2h^4
\ll\beta_Y^2(\log_2Y)^2=o(1),
\]
uniformly in \(x\in A\) and the admissible thresholds.
Thus the expansions in \cite[(4.12)]{BIRestricted}
have uniform \(o(1)\) errors. The variance bound
\cite[(4.13)]{BIRestricted} holds without modification.
The same change of measure therefore yields, for an
absolute constant \(C>0\),
\begin{equation}\label{eq:count-individual-tail}
\min_{x\in A}\PP(Z_x>t)
\ge p_t:=\exp\{-t^2/2-7t-C\}.
\end{equation}

\smallskip\noindent
\emph{(b).}
The joint-tail estimate \cite[(4.15)]{BIRestricted} gives
\begin{equation}\label{eq:count-joint-tail}
\PP(Z_x>t,Z_z>t)
\le\exp\{-t^2/(1+\rho_Y)\},
\qquad x,z\in A,\quad x\ne z.
\end{equation}
Moreover, since \(g\ge e^{(1+\gamma)u^2/2}\) and
\(t=(1+\gamma/4)u\),
\[
\log(gp_t)
\ge\frac{\gamma u^2}{5}-7(1+\gamma/4)u-C
\longrightarrow\infty
\]
uniformly over the admissible thresholds. In particular,
\(gp_t\ge1\) for all sufficiently large \(Y\).

Applying the second-moment argument in
\cite[(4.16)]{BIRestricted} to
\(
W_A=\sum_{x\in A}\1_{\{Z_x>t\}},
\)
we obtain
\[
\PP(H_A>t)\ge\mathcal R^{-1},
\qquad
\mathcal R:=
\frac1{gp_t}
+\exp\biggl\{
\frac{\rho_Yt^2}{1+\rho_Y}+14t+2C
\biggr\}.
\]
Consequently, for an absolute constant \(C_1>0\),
\begin{equation}\label{eq:count-second-moment-error}
\log(4\mathcal R)
\le C_1(\rho_Yu^2+u+1).
\end{equation}

\smallskip\noindent
\emph{(c).}
Let \(\mu_A\) be a median of \(H_A\).
The convex concentration inequality
\cite[(4.17)]{BIRestricted} applies without modification:
the variance normalization makes the maximum, viewed as
a function of \(((\varepsilon_i+1)/2)_{i=1}^d\), convex
and \(2\)-Lipschitz. Hence
\begin{equation}\label{eq:count-convex-concentration}
\PP(|H_A-\mu_A|\ge s)\le4e^{-s^2/16},
\qquad s>0.
\end{equation}
Combining this with \(\PP(H_A>t)\ge\mathcal R^{-1}\)
gives
\[
\mu_A\ge t-4\sqrt{\log(4\mathcal R)}.
\]
By \eqref{eq:count-second-moment-error},
\[
\frac{4\sqrt{\log(4\mathcal R)}}u
\le4\sqrt{C_1(\rho_Y+u^{-1}+u^{-2})}
\longrightarrow0
\]
uniformly over the groups and admissible thresholds.
Since \(t-u=\gamma u/4\), it follows that
\(\mu_A-u\ge\gamma u/8\) for all sufficiently large \(Y\).
Another application of
\eqref{eq:count-convex-concentration} yields
\begin{equation}\label{eq:groupfailure}
\PP(H_A\le u)
\le4\exp\biggl\{-\frac{\gamma^2u^2}{1024}\biggr\}.
\end{equation}

Choose \(C_{\gamma,N}\) sufficiently large that
\(
\gamma^2C_{\gamma,N}^2/1024\ge N+1.
\)
The lower bound \(u\ge C_{\gamma,N}\sqrt{\log_3Y}\)
then gives
\begin{equation}\label{eq:count-one-group}
\PP\big(\max_{x\in A}Z_x\le u\big)
\le4(\log_2Y)^{-(N+1)}
\ll_{\gamma,N}(\log_2Y)^{-N}.
\end{equation}
This quantitative bound is the additional conclusion
needed here from the argument of
\cite[Lemma~4.4]{BIRestricted}.

\medskip\noindent
{\sc Step 3.}
Let \(J_{\rm bad}\) be the number of groups containing
no coordinate greater than \(u\).
By \eqref{eq:count-one-group} and Markov's inequality,
\[
\E J_{\rm bad}
\ll_{\gamma,N}J(\log_2Y)^{-N},
\qquad
\PP(J_{\rm bad}>J/2)
\ll_{\gamma,N}(\log_2Y)^{-N}.
\]
Outside this exceptional event, at least \(J/2\)
disjoint groups each supply a distinct exceedance.
By \eqref{eq:count-groups},
\[
\#\{x\in S:Z_x\ge u\}
\ge\frac J2
\ge\frac{|S|}{4g}
\ge\frac{|S|}{8}e^{-{(1+\gamma)u^2}/{2}}.
\]
This proves \eqref{eq:count}.

Negating all coefficients preserves the assumptions
and proves the assertion for \(-Z\).
\end{proof}


\section*{Statement on the use of AI tools}

The development of this paper involved an extended series of
mathematical discussions with GPT-6 Astra (OpenAI). These
discussions contributed in particular to the proofs of
Propositions~\ref{prop:flatness} and \ref{prop:goodlayers}. GPT-6 Astra was also used during
the preparation of the manuscript to discuss the presentation,
check drafts and edit the language. Claude (Anthropic) was used
for proofreading and checking the final version.

The choice of the problem, the overall proof strategy and the
formulation of the main results are due to the authors.
The authors take full responsibility for the content of the
paper, the correctness of its mathematical results and proofs,
and any remaining errors.

\end{document}